\documentclass[11pt,leqno]{article}
\usepackage{amsmath,amsthm,amsfonts,amssymb}

\newcommand{\vp}{\varphi}

\newcommand{\CalG}{\mathcal{G}}
\newcommand{\CalS}{\mathcal{S}}

\newcommand{\gap}{\text{\rm gap}}

\newcommand{\Span}{\text{\rm span}}
\newcommand{\blockdiag}{\text{\rm blockdiag}}
\newcommand{\trans}{\text{\rm tr}}
\newcommand{\diag}{\text{\rm diag}}

\newcommand{\RR}{{\mathbb R}}

\newcommand{\CC}{{\mathbb C}}

\newcommand{\EE }{{\mathbb E}}

\newcommand{\mez}{{\frac{1}{2}}}

\newcommand\cA{{\cal  A}}
\newcommand\cB{{\cal  B}}

\newcommand\cU{{\cal  U}}
\newcommand\cH{{\cal  H}}
\newcommand\cV{{\cal  V}}
\newcommand\cW{{\cal  W}}
\newcommand\cC{{\cal  C}}

\newcommand\cG{{\cal  G}}
\newcommand\cK{{\cal  K}}
\newcommand\cL{{\cal  L}}
\newcommand\cN{{\cal  N}}
\newcommand\cE{{\cal  E}}
\newcommand\cF{{\cal  F}}
\newcommand\cP{{\cal  P}}

\newcommand\cO{{\cal O}}

\newcommand\cM{{\mathcal M}}

\newcommand\cS{{\mathcal S}}
\newcommand\Up{{\Upsilon}}
\newcommand\mrp{\mathrm{p}}

\newcommand{\rmp}{{\mathrm{p } }}

\newcommand{\bR}{\mathbb{R}}

\newcommand{\bL}{\mathbb{L}}

\newcommand{\bE}{\mathbb{E}}

\newcommand{\bN}{\mathbb{N}}
\newcommand{\bD}{\mathbb{D}}
\newcommand{\bC}{\mathbb{C}}
\newcommand{\oA}{\overline{A}}
\newcommand{\oB}{\overline{B}}

\newcommand{\ua}{\underline{a}}

\newcommand{\up}{{\underline p}}

\newcommand{\uu}{{\underline u}}

\newcommand{\uug}{{\underline g}}

\newcommand{\uzeta}{\underline \zeta}

\newcommand{\unu}{{\underline \nu}}

\def\eps{\varepsilon }

\def\D{\partial }
\newcommand\adots{\mathinner{\mkern2mu\raise1pt\hbox{.}
\mkern3mu\raise4pt\hbox{.}\mkern1mu\raise7pt\hbox{.}}}

\newcommand{\Id}{{\rm Id }}

\newcommand{\re}{{\rm Re }\, }
\newcommand{\rank}{{\rm rank }}
\renewcommand{\div}{{\rm div}}
\newcommand{\curl}{{\rm curl}}
\newcommand{\na}{{\nabla}}

\newtheorem{theo}{Theorem}[section]
\newtheorem{prop}[theo]{Proposition}
\newtheorem{cor}[theo]{Corollary}
\newtheorem{lem}[theo]{Lemma}
\newtheorem{defi}[theo]{Definition}
\newtheorem{ass}[theo]{Assumption}
\newtheorem{asss}[theo]{Assumptions}

\newtheorem{exam}[theo]{Example}
\newtheorem{rem}[theo]{Remark}
\newtheorem{rems}[theo]{Remarks}
\newtheorem{exams}[theo]{Examples}

\newtheorem{nota}[theo]{Notations}

\numberwithin{equation}{section}

\title{Existence and stability of noncharacteristic
boundary-layers for the compressible Navier-Stokes and viscous MHD
equations}

\author{\sc \small
Olivier Gues\thanks{LATP, Universit\'e de Provence;
gues@cmi.univ-mrs.fr.
Research of O.G. was partially supported by European network
HYKE, HPRN-CT-2002-00282.},
Guy M\'etivier\thanks{MAB, Universit\'e de Bordeaux I;
metivier@math.u-bordeaux.fr.
Research of G.M. was partially supported by European
network HYKE,  HPRN-CT-2002-00282. },
Mark Williams\thanks{
University of North Carolina;
williams@email.unc.edu.
Research of M.W. was partially supported by
NSF grants number DMS-0070684 and DMS-0401252.},
Kevin Zumbrun\thanks{Indiana University;
kzumbrun@indiana.edu: K.Z. thanks the Universities of
Bordeaux I and Provence for their hospitality during visits
in which this work was partially carried out.
Research of K.Z. was partially supported by
NSF grants number DMS-0070765 and DMS-0300487. } }

\begin{document}

\maketitle

\begin{abstract}
For a general class of hyperbolic-parabolic systems including the
compressible Navier-Stokes and compressible MHD equations,  we prove
existence and stability of noncharacteristic viscous boundary layers
for a variety of boundary conditions including classical
Navier-Stokes boundary conditions. Our first main result, using the
abstract framework established by the authors in the companion work
\cite{GMWZ6}, is to show that existence and stability of arbitrary
amplitude exact boundary-layer solutions follow from a uniform
spectral stability condition on layer profiles that is expressible
in terms of an Evans function (uniform Evans stability).  Whenever
this condition holds we give a rigorous description of the small
viscosity limit as the solution of a hyperbolic problem with
``residual" boundary conditions. Our second is to show that uniform
Evans stability for small-amplitude layers is equivalent to Evans
stability of the limiting constant layer, which in turn can be
checked by a linear-algebraic computation. Finally, for a class of
symmetric-dissipative systems including the physical examples
mentioned above, we carry out energy estimates showing that constant
(and thus small-amplitude) layers always satisfy uniform Evans
stability.
This yields existence of small-amplitude multi-dimensional boundary
layers for the compressible Navier-Stokes and MHD equations. For
both equations these appear to be the first such results in the
compressible case.
\end{abstract}



\tableofcontents

\section{Introduction}\label{intro}
\emph{} \qquad In this paper, we study existence and stability of
noncharacteristic viscous boundary layers of hyperbolic--parabolic
systems of the type arising in fluid and magnetohydrodynamics (MHD).
In a companion paper \cite{GMWZ6}, we have shown under mild
structural assumptions that for such layers, maximal linearized
stability estimates, transversality of layer profiles, and
satisfaction of the uniform Lopatinski condition by the associated
residual hyperbolic system all follow from a uniform spectral
stability condition on layer profiles that is expressible in terms
of an Evans function (uniform Evans stability, Definition
\ref{hfevans}). Here we use these abstract results to obtain
existence and stability in interesting physical applications. Our
structural hypotheses are general enough to allow van der Waals
equations of state.

In contrast to our previous joint papers, our main concern here is
to give an analytic verification of the uniform Evans hypothesis for
small amplitude layers (after constructing such layers), and to give
explicit calculations for important physical systems (isentropic
Navier-Stokes, full Navier-Stokes, and viscous MHD). For the
isentropic Navier-Stokes equations we are also able to construct
transversal \emph{large} amplitude layers, and to check maximal
dissipativity of the associated residual hyperbolic boundary
conditions.

The main results of this paper are as follows: (i) assuming the
uniform Lopatinski condition and transversality of layer profiles
(Definitions \ref{vv8} and \ref{deftrans}), we construct arbitrarily
high-order approximate boundary-layer solutions matching an inner
boundary-layer profile to an outer hyperbolic solution; (ii)
assuming uniform Evans stability, we use results of
\cite{GMWZ4,GMWZ6} to show existence and stability of exact
boundary-layer solutions close to the approximate solutions, and
consequently we obtain convergence of viscous solutions to solutions
of the residual hyperbolic problem in the small viscosity limit;
(iii) we show that uniform Evans stability of small-amplitude
boundary layers is equivalent to uniform Evans stability of the
associated limiting constant layer; and (iv), we use (iii) to verify
the uniform Evans condition for small amplitude layers for a class
of {\it symmetric-dissipative systems} that includes the above
physical examples as well as the class introduced by Rousset
\cite{R3}.

In connection with (iii) and (iv) above, we prove existence of
small-amplitude layer profiles  for a variety of boundary
conditions, including mixed Dirichlet-Neumann conditions. These
profiles appear in the leading term of the approximate solutions
referred to in (i).

The above results yield existence and stability of multi-dimensional
small-amplitude noncharacteristic boundary-layer solutions of the
compressible Navier-Stokes and viscous MHD equations. In both cases
these appear to be the first such results. For large-amplitude
layers, the questions of existence and stability are reduced to
verification of the uniform Evans condition. Efficient numerical
methods for such verification are presented, for example, in
\cite{CHNZ,HLyZ2,HLyZ2}.

Although most of the important physical examples can be written in
conservative form, the general theory is presented for the more
general nonconservative case in sections \ref{existsec},
\ref{limits}, \ref{stabsec} and the appendices. As we noted in
\cite{GMWZ6,GMWZ7}, the theory is clearer and in many ways simpler
in the more general setting.

For results on stability of boundary layers for the incompressible
Navier-Stokes equations we refer to \cite{TW,IS}.

 \subsection{Equations and assumptions}\label{equations}

Consider as in \cite{GMWZ6} a
quasilinear hyperbolic--parabolic system
\begin{equation}
\label{visceq} \cL_\eps(u):=
 A_0(u)u_t  + \sum_{j=1}^d   A_j(u) \D_{j} (u)    -
 \eps \sum_{j,k= 1}^d \D_{j} \big( B_{jk}(u) \D_{k} u \big) = 0,
\end{equation}
on $[-T,T]\times\Omega$, where $\Omega\subset \RR^d$ is an open set.
We assume the block structure
\begin{equation}
  \label{struc1}
  A_0(u)  =  \begin{pmatrix}A_0^{11}&0 \\A_0^{21}&A_0 ^{22}\end{pmatrix},
  \quad
  B_{jk} (u) =\begin{pmatrix}0&0 \\0 &B_{jk}^{22}\end{pmatrix},
  \end{equation}
a corresponding splitting
\begin{equation}
  \label{struc0}
   u = (u^1, u^2) \in \RR^{N-N'} \times \RR^{N'},
  \end{equation}
and decoupled boundary conditions
 \begin{equation}
\label{viscbcd}
\left\{\begin{array}{l}
\Upsilon_1   (u^1 )_{ | x \in \partial \Omega  }   = g_1(t,x) , \\
\Upsilon_2   (u^2 )_{ | x \in \partial \Omega }   = g_2(t,x) , \\
\Upsilon_3  (u,   \D_T u^2,   \D_\nu u^2 )_{ | x \in \partial\Omega
} = 0 ,
\end{array}\right.
\end{equation}
where $\D_T$ and $\D_\nu$ denote tangential and
inward normal derivatives
with respect to $\partial \Omega$ and $\Upsilon_3  (u, \D_T u^2,
\D_\nu u^2 ) =
  K_\nu  \D_\nu u^2 +  \sum_{j=1}^p  K_j (u)  V_j u^2 ,
$ where the $V_j(x)$ are smooth vector fields tangent to
$\partial\Omega$ and  $K_\nu$ is constant.
Unless otherwise noted we take $\Omega$ bounded with smooth (that
is, $C^k$ for $k$ large) boundary, but our results apply with no
essential change to other situations such as the case where
$\partial\Omega$ coincides with a half-space outside a compact set.


Setting  $\eps = 0$ in \eqref{visceq} we obtain $\cL_0$, a
first-order operator assumed to hyperbolic. The parameter $\eps$
plays the role of a non-dimensional viscosity and for $\eps
> 0$, the system is assumed to be parabolic or at least partially
parabolic.  Classical examples are the Navier-Stokes equations of
gas dynamics and the equations of magneto-hydrodynamics (MHD).
%
%

We set
\begin{align}\label{y2}
\oA_j=A^{-1}_0A_j,\quad \oB_{jk}=A^{-1}_0B_{jk},
\end{align}
\begin{align}\label{y3}
\oA (u, \xi) = \sum^d_{j=1} \xi_j \oA_j (u) \quad \text{ and
}\quad \oB(u,\xi)=\sum^d_{j,k=1}\xi_j\xi_k\oB_{jk}(u),
\end{align}
and systematically use the notation $M^{\alpha \beta} $ for the
sub-blocks of a matrix $M$ corresponding to the splitting $u =
(u^1, u^2)$.
Note that
  \begin{equation}
  \label{struc2}
  \oB_{j, k}(u):= A_0(u)^{-1}
  B_{jk} (u) =\begin{pmatrix}0&0 \\0 &\oB_{jk}^{22}(u),
  \end{pmatrix},
  \end{equation}
so it is natural to define the \emph{high-frequency principal part}
of \eqref{visceq} by
\begin{equation}
\label{princpart}
\left\{ \begin{array}{l}
\D_tu^1+   \overline A^{11}(u, \D)u^1=0,\\
     \D_t  u^2 - \eps \overline B^{22} (u, \D) u^2=0.
\end{array}\right.
\end{equation}

Our main structural assumptions are modeled on the fundamental
example of the Navier-Stokes equations with general, possibly van
der Waals type equation of state. For applications it is important
to allow the states assumed by solutions of \eqref{visceq} when
$\eps>0$ or when $\eps=0$ to vary in overlapping but not necessarily
identical regions of state space. Moreover, these regions must be
allowed to depend on $(t,x)$. These considerations motivate the
following definition of $\cU$, $\cU_\partial$, and $\cU^*$.

For some $T>0$ let $\cO(t,x)$ be a continuous set-valued function
from $[-T,T]\times \Omega$ to open sets in $\mathbb{R}^N$, and
define graphs
\begin{align}\label{C1}
\begin{split}
&\cU=\{(t,x,\cO(t,x)):(t,x)\in[-T,T]\times\Omega\}\\
&\cU_\partial=\{(t,x_0,\cO(t,x_0)):(t,x_0)\in[-T,T]\times\partial\Omega\}.
\end{split}
\end{align}
For $(t,x_0)\in[-T,T]\times\partial\Omega$ let $\cO^*(t,x_0)$ be
another continuous open-set-valued function satisfying
$\cO^*(t,x_0)\supset\cO(t,x_0)$ and define
\begin{align}\label{C2}
\cU^*=\{(t,x_0,\cO^*(t,x_0)):(t,x_0)\in[-T,T]\times\partial\Omega\}.
\end{align}
Observe that we have
\begin{align}\label{C3}
\cU_\partial\subset\cU\cap\cU^*,
\end{align}
but neither $\cU$ nor $\cU^*$ is a subset of the other.   For
elements of $\cU$ define
\begin{align}\label{C4}
\pi(t,x,\cO(t,x))=\cO(t,x)
\end{align}
and define $\pi$ similarly for elements of $\cU_\partial$ and
$\cU^*$.  Finally, denote by $\pi\cU$ (resp. $\pi\cU^*$,
$\pi\cU_\partial$) the union of the open sets obtained by applying
$\pi$ to elements of $\cU$ (resp. $\cU^*$, $\cU_\partial$).

 The
set $\pi\cU$ is the ``hyperbolic set" where solutions of the
inviscid equation $\cL_0(u)=0$ take their values; $\pi\cU^*$ is the
set where boundary layer solutions $u^\eps$ of the viscous
equations, restricted to a small neighborhood of the boundary, take
their values.  In particular, the layer profiles (Definition
\ref{layprof}) take values in $\pi\cU^*$. The set
$\pi\cU_\partial\subset \pi\cU\cap\pi\cU^*$ is the set of profile
endstates where matching of the two types of solutions occurs; more
precisely, it contains the limits as $z\to \infty$ of layer profiles
(see \eqref{endstate}), or equivalently, the boundary values of
solutions of the associated residual hyperbolic problem (see
\ref{vv3}).


\begin{asss}\label{Hs}
\textup{}

  \textup{(H1)} The matrices  $A_j$ and $B_{jk}$ are $C^\infty$ $N \times N$ real matrix-valued funtions  of
  the variable $u \in \pi\cU \cup \pi\cU^* \subset \RR^N$.  Moreover, for all
   $u \in \pi\cU\cup\pi\cU^*$, $\det A_0(u)\ne 0$.

\textup{(H2)}  There is $c > 0$ such that for all $u \in
\pi\cU\cup\pi\cU^*$ and $\xi \in \RR^d$, the eigenvalues of
$\oB^{22}(u, \xi)$  satisfy $\re \mu \ge c \vert \xi \vert^2 $.

\textup{(H3)}
For all $u \in \pi\cU\cup\pi\cU^*$ and $\xi \in
\RR^d\backslash\{0\}$, the eigenvalues of $\overline A^{11}(u, \xi)
$  are real, semi-simple and of constant multiplicity. Moreover, for
$u\in\pi\cU^*$, $\det \overline A^{11}(u, \nu)\ne 0 $, with the
eigenvalues of the normal matrix $ \overline A^{11}(u, \nu)\ne 0 $
all positive (inflow) or all negative (outflow), where $\nu$ denotes
the inward normal to $\partial \Omega$.

\textup{(H4)}  For all $u\in \pi\cU$ and $\xi\in
\RR^d\setminus\{0\}$, the eigenvalues of $\overline A (u, \xi)$ are
real, semisimple, and of constant multiplicity. Moreover, for
$u\in\pi\cU_\partial$, $\det \overline A (u, \nu) \ne 0 $, with
number of positive (negative) eigenvalues of $\overline A (u, \nu)$
independent of $u$.

\textup{(H5)}  There is $c > 0$ such that for $u\in \pi\cU$ and $\xi
\in \RR^d $, the eigenvalues of $ i \oA (u, \xi)  + \oB (u, \xi ) $
satisfy $\re \mu \ge  c   \frac{ \vert \xi \vert^2 }{ 1 + \vert
\xi\vert^2} $.

\end{asss}

\begin{rem}
\label{remhyp}

1.)In Hypothesis (H4) the statement ``for $u\in\pi\cU_\partial$,
$\det \overline A (u, \nu) \ne 0 $" should be interpreted as
asserting that for $(t,x_0,u)\in\cU_\partial$, we have
$\det\overline{A}(u,\nu(x_0))\neq 0$. A similar remark applies to
(H3) and to later statements of this sort.

2.) Hypothesis (H4) is a  hyperbolicity condition on the inviscid
equation $\cL_0(u)=0$, while \textup{(H2),(H4)} implies
hyperbolic--parabolicity of the viscous equation $\cL_\eps(u)=0$
when $\eps>0$.  (H3) is a hyperbolicity condition on the first
equation in \eqref{princpart}. The conditions on the normal matrices
in \textup{(H3)--(H4)} mean that the boundary is
\emph{noncharacteristic} for both the inviscid and the viscous
equations. Hypothesis \textup{(H5)} is a dissipativity condition
reflecting genuine coupling of hyperbolic and parabolic parts for
$u\in \pi\cU$.

3.)Later we will occasionally drop the $\pi$ on $\pi\cU$ in
statements like $u\in\cU$ below.

\end{rem}

Symmetry plays an important role in applications such as those to
the Navier-Stokes and MHD equations considered here. In particular,
(H5) holds always when the  conditions in the following two
definitions are satisfied \cite{KaS1, KaS2}.

\begin{defi}
\label{defsymm} The system $\eqref{visceq}$ is said to be symmetric
dissipative if there exists a real matrix $S(u)$, which depends
smoothly on $u \in \pi\cU $, such that for all $u \in \pi\cU$ and
all $\xi \in \RR^d \backslash\{0\}$,
the matrix $S (u) A_0(u) $ is symmetric definite positive and
block-diagonal,
$S(u)  A (u, \xi) $ is symmetric, and the symmetric matrix
 $\re  S(u)  B(u, \xi) $ is nonnegative with kernel of dimension $N-N'$.
\end{defi}

Given that $A_0$ and the matrices $B_{jk}$ have the structure
\eqref{struc1}, observe that we have the equivalences: \medbreak

$SA_0$  block diagonal $\Leftrightarrow$ $S$ lower block triangular
$\Leftrightarrow SB$  block diagonal.\footnote{ The block-diagonal
assumption repairs a minor omission in \cite{GMWZ6}; this is needed
to conclude (H5) from symmetric dissipativity plus the genuine
coupling condition. }

\begin{defi}
\label{gcoupling}
A symmetric--dissipative system satisfies the genuine coupling
condition if for all $u \in \pi\cU$ and all $\xi \in \RR^d
\backslash\{0\}$, no eigenvector of $\sum_j \overline A_j \xi_j$
lies in the kernel of $\sum_{j,k} \overline B_{jk}\xi_j \xi_k$.
\end{defi}

The  constant multiplicity condition in Hypothesis (H4) holds for
the compressible Navier--Stokes equations whenever $\oA(u,\xi)$ is
hyperbolic. We are able to treat symmetric-dissipative systems like
the equations of viscous MHD, for which the constant multiplicity
condition fails, under the following relaxed hypothesis. \medbreak

\textbf{Hypothesis H4$'.$ } For all $u\in \pi\cU$ and $\xi\in
\RR^d\setminus\{0\}$, the eigenvalues of $\overline A (u, \xi)$ are
real and are either semisimple and of constant multiplicity or are
totally nonglancing in the sense of \cite{GMWZ6}, Definition 4.3.
Moreover, for $u\in\pi\cU_\partial$ we have $\det \overline A (u,
\nu) \ne 0 $, with the number of positive (negative) eigenvalues of
$\overline A (u, \nu)$ independent of $u$.

\begin{rem}\label{altstructure}
The condition of constant multiplicity in (H3) can
probably be dropped for symmetric--dissipative systems. For
sufficiently small-amplitude boundary layers, we expect that the
condition in (H3) that eigenvalues have a common
sign may be dropped as well;
see Remark \ref{realrmks}.2.
\end{rem}

\begin{nota}\label{defNi}
 \textup{With assumptions as above, $N_+$ (constant)  denotes the number of positive eigenvalues of
 $\overline A_\nu(u):=\overline A(u,\nu)$ for
$u \in \pi\cU_\partial$ and $N^1_+$   the number of positive
eigenvalues of $\overline A^{11}_\nu(u):= \overline A^{11}(u,\nu)$
for $u \in \pi\cU^*$. We also set $N_b = N' + N^1_+ $. }
\end{nota}

  As indicated by block structure \eqref{princpart}, $N_b$  is the correct
number of boundary conditions for the well posedness of
\eqref{visceq}, for solutions with values in $\cU^*\cup\cU$, with
$N'$ boundary conditions for $u^2$  and $N^1_+ $ boundary conditions
for $u^1$. On the other hand, $N_+$ is the correct number of
boundary conditions for the inviscid equation for solutions with
values in $\cU$.

\begin{ass}
\label{assbc}
\textup{(H6)}$ \Upsilon_1 $, $\Upsilon_2$ and $\Upsilon_3$  are  smooth functions of their
arguments  with values in  $\RR^{N^1_+}$,  $\RR^{N'  - N'' }$ and $\RR^{N''}$
respectively, where $N'' \in \{ 0, 1, \ldots, N'\}$.
Moreover, $K_N$ has maximal rank $N''  $ and for all
$u \in \cU^*$ the Jacobian matrices
$\Upsilon'_1(u^1)$  and $\Upsilon'_2(u^2)$ have  maximal rank $N^1_+$
and $N'- N''$ respectively.

\end{ass}

\begin{exams}\label{basicexams}
 \textup{Hypotheses (H1)--(H5) are satisfied by the compressible
Navier--Stokes equations, provided the normal velocity of the fluid
 is nonvanishing on $\cU^*$ and the normal
characteristic speeds (eigenvalues of $\overline A(u,\nu)$) are
nonvanishing on $\cU_\partial$.   We have $N^1_+=1$ or $0$ according
as normal velocity is positive with respect to inward normal
(inflow) or negative (outflow).  Recalling that $\cU^*$ is the set
where boundary layer solutions of the viscous equations, restricted
to a small neighborhood of the boundary, take their values, we see
that these velocity restrictions correspond to having a porous
boundary through which fluid is pumped in or out, in contrast to the
characteristic, no-flux boundary conditions encountered at a solid
material interface for which normal velocity is set to zero.}

\textup{Hypotheses (H1)-(H5), with  (H4) replaced by (H4$'$),  are
satisfied  by the viscous MHD equations with ideal gas equation of
state under similar velocity restrictions on the plasma. See Remark
\ref{yy1} and the discussion following it for the precise velocity
requirements and for examples of boundary conditions for viscous MHD
and the corresponding reduced hyperbolic problem.}

 \textup{We note that our structural hypotheses are sometimes satisfied
 under much weaker
assumptions on the equation of state, which may be of van der Waals
type on $\cU^*$ and just thermodynamically stable on $\cU$; see
\cite{GMWZ4,Z3}.}

\textup{Boundary conditions of the type we consider for the
compressible Navier-Stokes equations are of considerable practical
importance. Recall the equations are
\begin{equation}
\label{NSeq}
\left\{ \begin{aligned}
 & \D_t \rho +  \div (\rho u) = 0
 \\
 &\D_t(\rho  u) + \div(\rho u^tu)+ \na p =
\eps \mu \Delta u + \eps(\mu+\eta) \nabla \div u
 \\
 &
 \D_t(\rho E) + \div\big( (\rho E  +p)u\big)=
\kappa \Delta T +
\eps \mu \div\big( (u\cdot \nabla) u\big) \\
&
\qquad \qquad \qquad \qquad
\qquad \qquad
+ \eps(\mu+\eta) \nabla(u\cdot \div u)
 \end{aligned}\right.
\end{equation}
where $\rho$ denotes density, $u$ velocity, $e$ specific internal
energy, $E=e+\frac{|u|^2}{2}$ specific total energy, $p=p(\rho, e)$
pressure, and $T=T(\rho, e)$ temperature.    Take the unknowns to be
$(\rho,u,T)$, and consider the problem on the exterior
$\Omega=\alpha^{c}$ of a bounded set $\alpha$ with smooth boundary,
with no-slip {\it suction-type} boundary conditions on the velocity,
$$
 u_T|_{\partial\Omega}=0, \quad  u_\nu|_{\partial \Omega}= V(x)< 0,
$$
and either prescribed or insulative boundary conditions on
the temperature,
$$
T|_{\partial \Omega}= T_{wall}(x)
\quad \hbox{\rm or }\quad
\D_\nu T|_{\partial \Omega}= 0.
$$
This corresponds to the situation of an airfoil with microscopic
holes through which gas is pumped from the surrounding flow, the
microscopic suction imposing a fixed normal velocity while the
macroscopic surface imposes standard temperature conditions as in
flow past a (nonporous) plate. This configuration was suggested by
Prandtl and tested experimentally by G.I. Taylor as a means to
reduce drag by stabilizing laminar flow; see \cite{S,Br}. It was
implemented in the NASA F-16XL experimental aircraft program in the
1990's with reported 25\% reduction in drag at supersonic speeds
\cite{Br}.\footnote{ See also NASA site
http://www.dfrc.nasa.gov/Gallery/photo/F-16XL2/index.html}}

\end{exams}

\begin{rem}\label{disconnected}
At the expense of further bookkeeping,
we could equally well define
separate $\cU_{\partial, j}\subset \cU^*_j$ for each connected
component $(\partial\Omega)_j$ of the boundary $\partial\Omega$,
each with distinct values
of $N_+$, $N^1_+$.
where now the sets $\cU_{\partial,j}$ are to be thought
of as the sets of possible boundary values for the hyperbolic problem at
each $(\partial\Omega)_j$.
This would allow, for example, the situation that fluid is pumped in
through one boundary and out another
as considered in \cite{TW} for the incompressible case.
\end{rem}

We follow hypotheses (H1)--(H6), in some cases with (H4$'$) in place
of (H4), throughout the paper; {\it unless otherwise indicated, they
are assumed in all statements and propositions that follow}.

\subsection{Layer profiles, transversality, and $\cC$ manifolds}\label{layerprofiles}

To match solutions of the inviscid problem approaching a constant
value $\underline u$ at $x_0\in \partial \Omega$ to
solutions satisfying the hyperbolic--parabolic boundary conditions, one looks
for  exact solutions  of \eqref{visceq} \eqref{viscbcd}
on the half-space $(x-x_0)\cdot \nu(x_0)\ge 0$ tangent to $\partial
\Omega$ at $x_0$
of the form
\begin{equation}
\label{layprof} u_\eps(t,x) = w \Big( \frac{(x-x_0)\cdot \nu
}{\eps} \Big)   ,
\end{equation}
$\nu(x_0)$ the inward unit normal to $\partial \Omega$ at $x_0$,
such that
\begin{equation}
\label{endstate}
 \lim_{z \to + \infty} w(z) = \underline  u \,.
\end{equation}
Solutions are called \emph{layer profiles}.

The equation for $w$ reads
\begin{equation} \label{layeq}
\left\{ \begin{aligned} &  A_\nu (w) \D_z w   - \D_z \big(
B_{\nu}(w)
  \D_z w  \big) = 0, \quad   z \ge 0, \\
&  \Upsilon (w, 0, \D_z w^2) _{ | z = 0 }  =
(g_1(t,x_0),g_2(t,x_0),0),
  \end{aligned}  \right.
\end{equation}
where $A_\nu(u)=\sum_{j=1}^d A_j(u)\nu_j$ and $B_{\nu}(u):=\sum
B_{jk}(u)\nu_j\nu_k$.   The natural limiting boundary conditions for
the inviscid problem
should be satisfied precisely by those states $\uu$ that are the
endstates of layer profiles.  This leads us to define
\begin{align}\label{vv41}
\cC(t,x_0)=\{\uu:\text{ there is a layer profile }w \in
C^\infty(\overline \RR_+ ; \cU^*)\text{  satisfying
\eqref{endstate}, \eqref{layeq} }.
\end{align}


The profile equation  \eqref{layeq} can be written as a first order
system for $ U = ( w, \D_z w^2)$, which is nonsingular if and only
if $A^{11}_\nu $ is invertible, (H3):
\begin{equation}
\label{layeq2}
\begin{aligned}
   \D_z w^1   &  =   -  ( A_\nu^{11} )^{-1} A^{12}_\nu  w^3  , \\
   \D_z w^2   &  =   w ^3   , \\
 \D_z \big( B^{22}_{\nu}   w ^3 )  & =
 \big( A_\nu^{22} - A_\nu^{21} ( A_\nu^{11})^{-1} A_\nu^{12} \big) w^3 ,
\end{aligned}
\end{equation}
and the matrices are evaluated at $w = (w^1, w^2)$.
%


Consider now the linearized equations of \eqref{layeq} about $w(z)$,
written as a first-order system
\begin{equation}
\label{linlayeq}
 \D_z \dot W- \CalG_\nu(z)\dot W  = 0 , \qquad   z \ge 0,
\end{equation}
\begin{equation}
\label{linlaybc}
 \Gamma_\nu   \dot W_{ | z = 0 }  =  0
\end{equation}
in $\dot W= (\dot w^1, \dot w^2, \dot w^3)$, where
\begin{equation}\label{Gform}
\CalG_\nu(+\infty):= \lim_{z\to +\infty}\CalG_\nu(z)=
\begin{pmatrix}
0 & 0 & -  ( A_\nu^{11} )^{-1} A^{12}_\nu \\
0 & 0 & I\\
0 & 0 & ( B^{22}_{\nu})^{-1} ( A_\nu^{22} - A_\nu^{21} (
A_\nu^{11})^{-1} A_\nu^{12} )
\end{pmatrix}(\underline u)
\end{equation}
and (note decoupling between $u^1$ and $(u^2,u^3):=(u^2, \D_\nu
u^2)$ variables)
\begin{equation}\label{Gammaeq}
 \Gamma_\nu U =  \big(   \Gamma_1  u^1  ,  \Gamma_2 u^2
         ,    K_\nu u^3 \big) .
\end{equation}

\begin{lem}[\cite{GMWZ6,MaZ3}]\label{count}
Let $N^2_-$ denote the number of stable eigenvalues $\Re \mu <0$ of
$\CalG_\nu(+\infty)$, $N^2_+$ the number of unstable eigenvalues
$\Re \mu
>0$, $\CalS$ the subspace of solutions of \eqref{linlayeq} that approach finite limits as $z\to\infty$,
and $\CalS_0$ the subspace of solutions of \eqref{linlayeq} that
decay to $0$. Then,

(i) $N^2_-+N^2_+=N'$ and
\begin{equation}\label{numberbc}
N_++N^2_-=N_b:=N'+N^1_+,
\end{equation}

(ii) profile $w(\cdot)$ decays exponentially to its limit $\uu$ as
$z\to +\infty$ in all derivatives, and

(iii) $\dim \CalS=N+N^2_-$ and $\dim \CalS_0=N^2_-$.
\end{lem}

\begin{proof} A proof is given in  \cite{GMWZ6}, Lemma 2.12.  \end{proof}

  \begin{defi}
  \label{deftrans}
  The profile $w$ is said to be \emph{transversal} if

  \quad i)  there is no nontrivial  solution
  $ \dot w \in \cS_0$ which satisfies the boundary conditions
  $  \Gamma_\nu (\dot w, \D_z \dot w^2)_{| z = 0} = 0$,

  \quad ii)  the mapping $\dot w \mapsto \Gamma_\nu (\dot w, \D_z \dot w^2)_{| z = 0}$
  from $\cS$ to $\CC^{N_b}$  has  rank $N_b$.
  \end{defi}

Definition \ref{deftrans}(i) and (ii) corresponds to the geometric
conditions that the level set $\{W:\Upsilon(W)=\Upsilon(w(0),0,
\partial_z w^2(0))\}$ have transversal intersections in phase space
$W=(w,\partial_z w^2)$ at $ W_0:=(w(0), \partial_z w^2(0)) $ with
the stable and center--stable manifolds, respectively, of
$W_\infty:=(\uu,0)$; see Lemma \ref{geo} below.

The following assumption is the starting point for our construction
of exact boundary layer solutions to \eqref{visceq}.

\begin{ass}\label{C5}
Fix a choice of $(g_1,g_2)$ as in \eqref{viscbcd}.  For
$\cU_\partial$ as in \eqref{C1} we are given a smooth manifold $\cC$
defined as the graph
\begin{align}\label{C6}
\cC=\{(t,x_0,\cC(t,x_0)):(t,x_0)\in[-T,T]\times\partial\Omega\}\subset\cU_\partial,
\end{align}
where each $\cC(t,x_0)$, defined as in \eqref{vv41}, is now assumed
to be a smooth manifold of dimension $N-N_+$ . In addition we are
given a smooth function
\begin{align}\label{C8}
W:[0,\infty)\times\cC\to \pi\cU^*
\end{align}
such that for all $(t,x_0,q)\in\cC$, $W(z,t,x_0,q)$ is a transversal
layer profile satisfying \eqref{layeq} with $\nu=\nu(t,x_0)$ and
converging to $q$
as $z\to\infty$ at an exponential rate that can be
taken uniform on compact subsets of $\cC$.
\end{ass}

This  assumption is hard to check in general. However, in
Proposition \ref{C30} we show that for a large class of problems
including the Navier-Stokes and MHD equations with various boundary
conditions,  Assumption \ref{C5} is always satisfied for
small-amplitude profiles.
In Proposition \ref{C22prop} we give necessary and sufficient
conditions on boundary operators of the form \eqref{viscbcd} in
order for Assumption \ref{C5} to hold in the small-amplitude case.
These boundary conditions include the standard
noncharacteristic boundary
conditions for the Navier-Stokes and viscous MHD equations. In
Proposition \ref{redBC} we give a \emph{local} construction of a
$\cC$ manifold with associated profiles near a given, possibly large
amplitude, transversal profile.
In Proposition \ref{largeC}
we give a global construction of a
$\cC$ manifold with associated profiles near a given family
of large-amplitude profiles, assuming such a family exists
(not clear in general).

As regards the failure of Assumption \ref{C5}, Proposition \ref{C30}
implies, for example, the following nontransversality result.

\begin{prop}\label{C8a}
Whenever $\mathrm{rank}\Up_3=N''>N^2_-$, small-amplitude profiles
are not transversal. Equivalently, whenever the number of scalar
Dirichlet conditions for the parabolic problem, $(N'-N'')+N^1_+$, is
strictly less than the number of scalar boundary conditions for the
residual hyperbolic problem, $N_+$ (see \eqref{vv3}),
small-amplitude profiles are not transversal.

\end{prop}

\subsection{Inviscid solutions, residual boundary conditions, and the uniform Lopatinski condition.}

 Under suitable assumptions we will study the small
viscosity limit of solutions to
\begin{align}\label{vv2}
\begin{split}
&\cL_\eps(u):=
 A_0(u)u_t  + \sum_{j=1}^d   A_j(u) \D_{j} u    -
 \eps \sum_{j,k= 1}^d \D_{j} \big( B_{jk}(u) \D_{k} u \big) = 0,\\
&\Up(u,\partial_Tu^2,\partial_\nu u^2)=(g_1,g_2,0)\text{ on
}[0,T_0]\times \partial\Omega
\end{split}
\end{align}
and demonstrate convergence to a solution $u^0(t,x)$ of the inviscid
hyperbolic problem:
\begin{align}\label{vv3}
\begin{split}
&\cL_0(u^0)=0\text{ on }[0,T_0]\times\Omega\\
&u^0(t,x_0)\in \cC(t,x_0)\text{ for }(t,x_0)\in
[0,T_0]\times\partial\Omega,
\end{split}
\end{align}
where $\cC(t,x_0)$ is the endstate manifold defined in Assumption
\ref{C5} (see also Prop. \ref{C30}).

\begin{defi}\label{vv4}
We refer to the  boundary condition in \eqref{vv3} as the
\emph{residual hyperbolic boundary condition}.
\end{defi}

For $(t,x_0)\in [0,T_0]\times\partial\Omega$ we freeze a state
$p:=u^0(t,x_0)$ and, working in coordinates where the boundary is
$x_d=0$, we define
\begin{align}\label{vv5}
H(p,\zeta):=-A_d(p)^{-1}\left((i\tau+\gamma)A_0(p)+\sum^{d-1}_{j=1}i\eta_j
A_j(p)\right).
\end{align}
Here we have suppressed the dependence of the frozen $A_j$ on
spatial coordinates in the notation.   Let
\begin{align}\label{vv6}
\psi:\bR^N\to\bR^{N_+}
\end{align}
be a defining function for $\cC(t,x_0)$ near $p$, i.e.,
$\cC(t,x_0)=\{u:\, \psi(u)=0\}$,  with $\nabla \psi $ full rank $N_+$.
Then, the residual boundary condition \eqref{vv3} may be expressed,
locally to $p$, as $\Upsilon_{res}(u):=\psi(u)$, hence the
\emph{linearized residual boundary condition}
at $p$ takes the form
\begin{align}\label{vv7}
\Gamma_{res}(p)\dot u=0 \Leftrightarrow \psi'(p)\dot u=0
\Leftrightarrow\dot u \in T_p\cC(t,x_0).
\end{align}

\begin{rem}\label{vvv1}
Suppose $w(z)$ is a solution of \eqref{layeq} converging to
$p=u^0(t,x_0)\in\cC(t,x_0)$ as $z\to\infty$. Let us write the
linearized equations of \eqref{layeq} around $w(z)$ as
\begin{align}\label{vvv2}
\bL(t,x_0,z,\partial_ z)\dot w=0,\;\;\Gamma_{\nu(x_0)}(\dot w,\dot
w^2_z)=0.
\end{align}
 Observe that the tangent space $T_p\cC(t,x_0)$ may be
characterized as the set of limits at $z=\infty$ of solutions to
\eqref{vvv2}. This follows readily from the definition of
$\cC(t,x_0)$; see \cite{Met4}, Prop. 5.5.5.

\end{rem}

\begin{defi}\label{vv8}
1)  The inviscid problem \eqref{vv3} satisfies the \emph{uniform
Lopatinski condition at $p=u(t,x_0)$} provided there exists $C>0$
such that for all $\zeta$ with $\gamma>0$
\begin{align}\label{vv9}
|D_{Lop}(p,\zeta)|:=|\det\left(\bE^-(H(p,\zeta)),\ker\Gamma_{res}(p)\right)|\geq
C.
\end{align}

2) The inviscid problem \eqref{vv3} satisfies the \emph{uniform
Lopatinski condition} provided \eqref{vv9} holds with a constant
that can be chosen independently of $(t,x_0)\in [0,T_0]\times
\partial\Omega$.
\end{defi}

Here by a determinant of subspaces we mean the determinant of the
matrix with subspaces replaced by smoothly chosen bases of column
vectors, specifying $D_{Lop}$ up to a smooth nonvanishing factor;
for example, if the bases are taken to be orthonormal, then
$|D_{Lop}|$ is independent of the choice of bases.  We refer to
\cite{GMWZ7}, section 4.1 for equivalent formulations and further
discussion of the uniform Lopatinski condition.

\begin{theo}\label{vv10}

Given a smooth manifold $\cC$ as in Assumption \ref{C5}, consider
the hyperbolic problem \eqref{vv3}

(i) under hypotheses (H1)-(H5), or alternatively,

(ii) assuming (H1)-(H5), except that (H4) is replaced by (H4') in
the symmetric-dissipative case.

Let $s>\frac{d}{2}+1$ and suppose that we are given initial data
$v^0(x)\in H^{s+1}(\Omega)$ at $t=0$ satisfying corner compatibility
conditions to order $s-1$ for \eqref{vv3}.  Suppose also that the
uniform Lopatinski condition is satisfied at all points
$x_0\in\partial\Omega$, $t=0$. Then there exists a $T_0>0$ and a
function $u^0(t,x)\in H^s([0,T_0]\times \Omega)$ satisfying
\eqref{vv3} with
\begin{align}\label{vv11}
u^0_{|t=0}=v^0,
\end{align}
and so that the uniform Lopatinski condition holds on
$[0,T_0]\times\partial\Omega$.
\end{theo}

\begin{proof}

We refer to \cite{CP}, Chapter 7 for a discussion of corner
compatibility conditions.   For the proof in the case of constant
multiplicity (i.e., when (H4) holds) see \cite{CP}, Chapter 7 and
\cite{Met2}.  For the case of variable multiplicities (including
situations more general than those considered here) see \cite{MZ2}.

\end{proof}

\subsection{Approximate solutions to the viscous problem}\label{multi-scale}

\textbf{}
 \quad The first step in constructing exact solutions to the viscous
problem \eqref{vv2} that converge to a given solution $u^0$ of the
hyperbolic problem \eqref{vv3} in the small viscosity limit is to
construct high order approximate solutions of \eqref{vv2} with that
property.  Following the approach of \cite{GMWZ4,GMWZ7} we construct
approximate solutions using a WKB expansion as described in the
following result.  A more precise statement of Proposition
\ref{approxsoln} and the proof are given in Appendix
\ref{approximate}.

Let
\begin{equation}\label{ZN}
(x_0,   z): \Omega \to \partial \Omega \times \RR^+
\end{equation}
be a smooth map defined for $d(x,\partial \Omega)\le r$, $r>0$
sufficiently small, such that
$$
(x_0, z)^{-1}(x_0,z)= x_0 + z\nu(x_0),
$$
where $\nu(x_0)$ is the inward normal to $\partial \Omega$ at $x_0$,
i.e., $(x_0, z)$ are normal coordinates and $\nabla z=\nu$ on
$\partial\Omega$. Let $\chi(x)$ be a smooth cutoff function
vanishing for $d(x,\partial \Omega)\ge 2r$ and identically one for
$d(x,\partial \Omega)\le r$.

\begin{prop}[Approximate solutions] \label{approxsoln}
Suppose we are given a $\cC$ manifold and associated transversal
profiles $W(z,t,x_0,q)$ as in Assumption \ref{C5}, and also a
solution $u^0\in H^{s_0}$ to the inviscid problem \eqref{vv3} as
described in Theorem \ref{vv10}.  (In particular, the uniform
Lopatinski condition is satisfied on $[0,T_0]\times\partial\Omega$).
Fix positive integers $M$ and $s$ with $M\geq s \geq 1$.  Provided
$s_0$ is sufficiently large relative to $M$ and $s$ (see
\eqref{B28}), there exists an approximate solution $u^\eps_a(t,x)$
to the viscous problem \eqref{vv2} on $[0,T_0]\times\Omega$ of the
form
\begin{align}\label{vv20}
\begin{split}
&u^\eps_a(t,x)=\sum_{0\le j\le M} \eps^j\cU^j(t,x, \frac{z}{\epsilon})+\epsilon^{M+1}u^{M}(t,x),\\
&\mathcal{U}^j(t,x,\frac{z}{\epsilon})=
\chi(x)\left(W^j(\frac{z}{\eps},t,x_0)-W^j(+\infty,t,x_0)\right) +
u^j(t,x).
\end{split}
\end{align}
Here $u^0$ is the given inviscid solution, the leading profile $W^0$
is given by
\begin{align}\label{vv21}
W^0(Z,t,x_0):=W(Z,t,x_0,u^0(t,x_0))
\end{align}
for $W(Z,t,x_0,q)$ as in Assumption \ref{C5}, and the higher
profiles $W^j(Z,t,x_0)$ converge exponentially to their limits as
$Z\to+\infty$.   The approximate solution $u^\eps_a$ satisfies
\begin{align}\label{vv22}
\begin{split}
&\|\cL_\eps(u_a)\|_{H^s([0,T_0]\times\Omega)} \le C\eps^M\\
& \Upsilon(u_a, \D_T u^2_a, \D_\nu u^2_a)=(g_1,g_2,0) \text{ on
}[0,T_0]\times \partial \Omega
\end{split}
\end{align}
for $(g_1,g_2,0)$ as in the original viscous problem \eqref{vv2}.

\end{prop}

\begin{exams}\label{existexams}

\textup{The combination of Theorem \ref{vv10} (for inviscid
solutions $u^0$), Proposition \ref{C30} (for $\cC$ manifolds and
transversal profiles), and Corollary \ref{symmstab} (for
satisfaction of the uniform Evans condition which implies
transversality and uniform Lopatinski) provides us with a variety of
examples, involving both Dirichlet and Neumann conditions for the
Navier Stokes and viscous MHD equations and including all cases
mentioned in
Examples \ref{basicexams}, where the hypotheses of Proposition
\ref{approxsoln} are satisfied for small-amplitude profiles.}

\textup{The large-amplitude case is more difficult.  In section
\ref{Isentropic} we construct $\cC$ manifolds with associated
large-amplitude profiles for the isentropic Navier-Stokes equations.
Transversality of those profiles is verified in Proposition
\ref{ac1} and maximal dissipativity of the residual hyperbolic
problem, which implies the uniform Lopatinski condition, is verified
in section \ref{maxlop}. Thus, here again the hypotheses of
Proposition \ref{approxsoln} are satisfied.}

\textup{Large-amplitude $\cC$ manifolds and transversal profiles are
constructed \emph{locally} near the endstate of a \emph{given}
transversal profile in Proposition \ref{redBC}. In problems with
rotational symmetry it may be possible to promote locally
constructed large-amplitude $\cC$ manifolds to global ones by using
that symmetry.  For large-amplitude layers there are now efficient
numerical methods  available for verification of the uniform Evans
condition. \cite{CHNZ,HLyZ2,HLyZ2}}
\end{exams}

\subsection{The Evans condition, stability and transversality, and the small viscosity limit}\label{evanssec}
\emph{} \quad We refer to the approximate solution $u^\eps_a$
described in Proposition \ref{approxsoln}
as an approximate boundary-layer solution. Under a suitable Evans
condition (Definition \ref{vv50}), we will produce a nearby
\emph{exact} boundary-layer solution $u^\eps$ of \eqref{visceq},
\eqref{viscbcd} of the form
\begin{align}\label{vv43a}
u^\eps=u^\eps_a+v^\eps,
\end{align}
where $v^\eps$ satisfies the ``error problem"
\begin{equation}\label{errorproblem}
\cL_\eps(u_a+v)-\cL_\eps(u_a)= -\cL_\eps(u_a),
\end{equation}
derived from the problems satisfied by $u^\eps$ and $u^\eps_a$, with
boundary data
\begin{equation}\label{errordata}
 \big(\Upsilon(u_a+v,\partial_{T,\nu}(u^2_a+v^2))-\Upsilon (u_a,\partial_{T,\nu}u^2_a)\big)|_{x\in
\partial \Omega}=0
\end{equation}
and forcing term $-\cL_\eps(u_a)$ small of order $\eps^M$ in $H^{s}$
as described in \eqref{vv22}, with $M$ and $s$ large. Here
$v_{|t=0}$ satisfies high order corner compatibility conditions
depending on $u_a$.

Evidently, the problem of estimating convergence error
$\|v^\eps\|_{H^{s}}$ in terms of truncation error
\eqref{errorproblem} amounts to determining $H^{s}\to H^{s}$
stability estimates for \eqref{errorproblem}--\eqref{errordata},
for which the main obstacle is rapid variation of coefficients
in the boundary-layer region.
We focus now on this region, and stability
of associated layer-profiles, i.e., the ``microscopic''
stability problem.

 Fix a point $(t,x_0)$, $x_0\in\partial\Omega$ and
consider again the viscous problem \eqref{vv2}.   We work in local
spatial coordinates $(y,x_d)$  where $x_0$ is $(0,0)$ and $\partial
\Omega$ is given by $x_d=0$.  Consider a planar layer profile
\begin{align}\label{vv43}
u^\eps(t,y,x_d)=w(x_d/\eps)
\end{align}
as in \eqref{layprof}, which is an exact solution to \eqref{vv2} on
$x_d\geq 0$ when the coefficients and boundary data $(g_1,g_2,0)$
are frozen at $(t,x_0)$. Without loss of generality we take
$\eps=1$, set $z=x_d$, and write the linearized equations of
\eqref{vv2} about $w$:
\begin{equation}
\label{linq} \cL'_{w} \dot u  = \dot f , \quad
 \Upsilon' (  \dot u,  \D_y \dot u^2,  \D_z \dot u^2)  _{|
x = 0 } = \dot g.
\end{equation}
Here $\Upsilon'$ is the differential of $\Upsilon $ at $(w(0), 0,
\D_z w^2(0))$ and $\cL'_{w}$ is a differential operator with
coefficients that are smooth functions of $x_d$.

Performing a  Laplace-Fourier  transform of \eqref{linq} in $(t,
y)$, with frequency variables denoted by $\gamma + i  \tau$ and $
\eta$ respectively, yields the family of ordinary differential
systems
 \begin{equation}
\label{linq3} L( z,   \gamma+ i   \tau, i    \eta, \D_z)    u = f,
\quad \Upsilon'  (   u,  i \eta   u^2, \D_z   u^2 ) _{| z = 0} = g,
\end{equation}
\begin{equation}
\label{linL} L      = -  \cB (z)   \D^2 _z   +  \cA (z, \zeta) \D_z
+ \cM (z, \zeta),
\end{equation}
with in particular, $\cB(z) = B_{d,d} (w(z))$ and $\cA^{11} (z,
\zeta) = A^{11}_d(w(z))$.  The matrices $\cA (z, \zeta)$, $\cM (z,
\zeta)$ are written out explicitly in \eqref{coefflin}.

The problem \eqref{linq3} may be written as a first order system
 \begin{equation}
\label{linq4} \D_z U  - \cG(z, \zeta) U = F  , \quad   \Gamma
(\zeta)  U _{| z = 0} = G,
\end{equation}
where $U =  ( u , \D_z u^2)=(u^1,u^2,u^3) \in \CC^{N+ N'}$ and
$\zeta=(\tau,\gamma, \eta)$. The components of $\cG(z,\zeta)$ are
given explicitly in \eqref{cG} and we have
\begin{align}\label{pp1}
\Gamma(\zeta)U=(\Gamma^1u^1,\Gamma^2u^2,\Gamma^3(\zeta)u^3)
\end{align}
with
\begin{align}\label{pp2}
\begin{split}
&\Gamma^1u^1=\Up_1'(w^1(0))u^1,\;\Gamma^2u^2=\Up_2'(w^2(0))u^2,\;\\
&\Gamma^3(\eta)U=K_du^3+K_T(w(0))i\eta u^2.
\end{split}
\end{align}
Observe that when $\zeta=0$ \eqref{linq4} coincides with
\eqref{linlayeq}--\eqref{linlaybc} in the case when $\nu=(0,1)$.

A necessary condition for stability of the inhomogeneous equations
\eqref{linq4} is stability of the homogeneous case $F=0$, $G=0$,
i.e., nonexistence for $\gamma \ge 0$, $\zeta\ne 0$ of solutions $U$
decaying as $z\to +\infty$ and satisfying  $\Gamma(\zeta) U(0)=0$.
These may be detected by vanishing of the {\it Evans function}
\begin{equation}\label{evans1}
D(\zeta):=
\det_{N+N'} (\EE^-(\zeta),\ker \Gamma(\zeta)),
\end{equation}
where
$\EE^-(\zeta)$ is the subspace of initial data at $z=0$ for which
the solution of $\D_z U  - \cG(z, \zeta) U = 0$ decays at
$z=+\infty$.   For high frequencies $|\zeta|\geq R>0$ we also define
in \eqref{rescEvf} a rescaled Evans function $D^{sc}(\zeta)$.

Given a $\cC$ manifold and associated layer profiles $W(Z,t,x_0,q)$
as in Assumption \ref{C5}, along with an inviscid solution $u^0$ as
in Theorem \ref{vv10}, we  define in the same way $D(t,x_0,\zeta)$
and $D^{sc}(t,x_0,\zeta)$ for \emph{every} $(t,x_0)\in [0,T_0]\times
\partial\Omega$ using the associated profile $w(Z)=W(Z,t,x_0,u(t,x_0))$.

\begin{defi}\label{vv50}
1.)  We say that the \emph{uniform Evans condition is satisfied by
the profile $w(z)$} provided there exist positive constants $R$, $C$
such that
\begin{align}\label{vv51}
\begin{split}
&|D(\zeta)|\geq C \text{ for }0<|\zeta|\leq R,\;\gamma\geq
0\text{ and }\\
&|D^{sc}(\zeta)|\geq C \text{ for }|\zeta|\geq R,\;\gamma\geq 0.
\end{split}
\end{align}

2.)We say that the \emph{uniform Evans condition holds on
$[0,T_0]\times \partial\Omega$} provided there exist positive
constants $R$, $C$ such that the estimates \eqref{vv51} hold for the
Evans functions $D(t,x_0,\zeta)$ and $D^{sc}(t,x_0,\zeta)$ defined
above, uniformly for all $(t,x_0)\in [0,T_0]\times
\partial\Omega$.

\end{defi}

In order to understand the behavior of $D(\zeta)$ near $\zeta=0$ it
is helpful to introduce polar coordinates
\begin{align}\label{ab1}
\zeta=\rho\hat\zeta,\;\rho=|\zeta|, \text{ for }\zeta\neq 0,
\gamma\geq 0
\end{align}
and to write
\begin{align}\label{ab2}
D(\zeta)=D(\hat \zeta,\rho),\;\;\;
\bE^-(\zeta)=\bE^-(\hat\zeta,\rho)\text{ for }\zeta\neq 0,\gamma\geq
0.
\end{align}
It is shown in \cite{MZ3}, Theorem 3.3 and \cite{GMWZ6}, Remark 2.31
that the spaces $\bE^-(\zeta)=\bE^-(\hat\zeta,\rho)$ have continuous
extensions to $\rho=0$ under our structural hypotheses (and, more
generally, when there exist $K-$families of symmetrizers for the
linearized viscous problem.)   Hence, under our assumptions on
$\Gamma$,  $D(\hat\zeta,\rho)$ extends continuously to $\rho=0$ for
$\hat\zeta$ with $\hat\gamma\geq 0$.   This continuity allows us to
rephrase the low frequency uniform Evans condition equivalently as
\begin{align}\label{ab3}
D(\hat\zeta,0)\neq 0 \text{ for }|\hat\zeta|=1,\;\;\hat \gamma\geq
0.
\end{align}

The following elementary result allows us to verify Evans conditions
by proving trace estimates.

\begin{lem}[\cite{GMWZ6}, Lemma 2.19]\label{ZZZ}
Suppose that $\bE\subset\bC^n$ and $\Gamma:\bC^n\to\bC^m$, with
$\mathrm{rank} \;\Gamma=\mathrm{dim}\bE=m$.  If
$|\det(\bE,\ker\Gamma)|\geq c>0$, then there is $C$, which depends
only on $c$ and $|\Gamma^*(\Gamma\Gamma^*)^{-1}|$ such that
\begin{align}\label{ZZ}
|U|\leq C|\Gamma U|\text{ for all }U\in\bE.
\end{align}
Conversely, if this estimate is satisfied then
$|\det(\bE,\ker\Gamma)|\geq c>0$, where $c>0$ depends only on $C$
and $|\Gamma|$.

\end{lem}

\begin{rem}\label{Evansest}
By Lemma \ref{ZZZ}, the uniform Evans condition $|D(\zeta)|\ge C>0$
on some subset $S$ of frequencies is equivalent to
\begin{equation}\label{altuni}
|U|\leq  C|\Gamma U| \, \, \hbox{\rm for all }\, U\in \EE^-(\zeta)
\end{equation}
for some  constant $C>0$ independent of $\zeta\in S$.
\end{rem}

We now recall two results from \cite{GMWZ6}.  The first extends an
earlier result of Rousset \cite{R2} and shows that only the low
frequency Evans condition is needed for the construction of high
order approximate viscous solutions $u^\eps_a$.

\begin{lem}[\cite{GMWZ6} Theorem. 2.28]\label{Rousset}
Assume (H1)-(H6) (with (H4$'$) replacing (H4) in the
symmetric-dissipative case), and consider a layer profile $w(z)\to
p$ as $z\to\infty$. The uniform Evans condition holds for low
frequencies, that is, there exist positive constants $r$, $c$ such
that
\begin{align}\label{vv60}
|D(\zeta)|\geq c \text{ for }|\zeta|\leq r,
\end{align}
if and only if $w$ is  transversal and the uniform Lopatinski
condition (Definition \ref{vv8}) holds at $p$ for the residual
hyperbolic problem \eqref{vv5}-\eqref{vv7}.
\end{lem}

The next result shows that the full uniform Evans condition implies
maximal linearized stability estimates for the viscous problem.

\begin{prop}[\cite{GMWZ6}, Theorems 3.9 and 7.2]\label{planstab}
Assume (H1)-(H6) (with (H4$'$) replacing (H4)  in the symmetric
dissipative case) and consider the  problem
\eqref{linq} defined by linearization around a layer profile $w(z)$.
If the uniform Evans condition is satisfied by $w(z)$ (Definition
\ref{vv50}), then solutions to \eqref{linq} satisfy the maximal
stability estimates \eqref{maxest2} and \eqref{maxesthf}.
\end{prop}

Proposition \ref{planstab} is proved by constructing smooth
Kreiss-type symmetrizers for the Laplace--Fourier transformed
equations \eqref{linq4} in the low-, medium-, and high-frequency regimes.
This method of proof is just as important as the result, since the
same symmetrizer construction may be used to obtain maximal
stability estimates for the linearized equations about general
(nonplanar) solutions of the form \eqref{vv20}.  The procedure,
which is used in \cite{MZ1,GMWZ3,GMWZ4}, is to freeze slow variables
in the original linearized viscous problem, take the Laplace-Fourier
transform to obtain ODEs depending on frequency like \eqref{linq4},
construct symmetrizers for the transformed problem, and then
quantize those symbols to produce paradifferential operator
symmetrizers for the original (unfrozen) problem.
The resulting linear estimates can be used to
prove convergence of a nonlinear iteration scheme that yields
existence of an exact solution to the viscous problem \eqref{vv2}
that is close to the approximate solution $u^\eps_a$,
as stated in the following theorem.

Collecting these observations and combining with Lemma \ref{Rousset}
and Theorem \ref{approxsoln}, we obtain the
following main result, which
reduces the problem of proving existence and nonlinear stability of
boundary-layer solutions to verification of the uniform Evans
condition.

\begin{theo}\label{existence}
Consider the viscous problem \eqref{vv2} under assumptions (H1)-(H6)
(or with (H4$'$) replacing (H4) in the symmetric-dissipative case).
Given an inviscid solution $u^0\in H^{s_0}([0,T_0]\times \Omega)$ as
in Theorem \ref{vv10}, suppose that the uniform Evans condition
holds on $[0,T_0]\times
\partial\Omega$ (Defn. \ref{vv50}).  Suppose the constants $k$, $M$, and $s_0$ satisfy
\begin{align}\label{o21}
k>\frac{d}{2}+4,\;\; M>k+2,\;\; s_0>k+\frac{7}{2}+2M+\frac{d+1}{2}.
\end{align}
Then there exists $\eps_0>0$, an approximate solution $u^\eps_a$ as
in \eqref{vv20} satisfying \eqref{vv22}, and an exact solution
$u^\eps$  of \eqref{vv2} such that for $0<\eps\leq \eps_0$
\begin{align}\label{o24}
\begin{split}
&\|u^\eps-u^\eps_a\|_{W^{1,\infty}([0,T_0]\times\Omega)}\le C\eps^{M-k}, \\
&\|u-u^0\|_{L^2(\Omega \times [0,T])}\le C\eps^{1/2},\\
&u^\eps-u^0=O(\eps) \text{ in
}L^\infty_{loc}([0,T_0]\times\Omega^\circ)
\end{split}
\end{align}
where $\Omega^\circ$ denotes the interior of $\Omega$. Moreover, the
linearized equations about either $u^\eps_a$ or $u^\eps$ satisfy
maximal stability estimates.
\end{theo}

\begin{proof}
The proof is by the same iteration scheme as the corresponding
result for shocks, Theorem 6.18 in \cite{GMWZ4}, except that it is
simpler because the function $\psi$ defining the free transmission
boundary in \cite{GMWZ4} is absent in the present fixed boundary
problem.   In the partially parabolic case the iteration scheme,
which is explained in \cite{GMWZ4} section 6.1.1, must be designed
carefully, because the linearized estimates give weaker control over
the ``hyperbolic component" $u^1$ than the ``parabolic component"
$u^2$. Higher $W^{s,\infty}$ norms of $u^\eps-u^\eps_a$ can be
estimated by increasing $k$, $M$, and $s_0$.

The proof of Theorem 6.18 in \cite{GMWZ4} used the constant
multiplicity assumption (H4) to construct symmetrizers. The proof of
our Proposition \ref{planstab} yields symmetrizers under the weaker
assumption (H4$'$) when (H4) fails in the symmetric-dissipative
case.  Those symmetrizers are then used exactly as in the constant
multiplicity case.

\end{proof}

\subsection{Verification of the Evans condition for small-amplitude layers}\label{verification}

\qquad For large-amplitude boundary-layers, the Evans condition may
be checked numerically; see \cite{CHNZ,HLyZ1} in the one-dimensional
case, \cite{HLyZ2} in the multi-dimensional shock case. For
small-amplitude layers, it may be verified analytically in several
interesting cases. In particular, we will show that the uniform
Evans condition  {\it always} holds for small-amplitude layers of
symmetric--dissipative systems with Dirichlet boundary conditions,
and {\it never} (by Proposition \ref{smalltrans}(ii) together with
Lemma \ref{Rousset}) for constant layers of systems with mixed
Dirichlet--Neumann conditions for which the number $N_+$ of incoming
hyperbolic characteristics on $\cU$ exceeds the number of Dirichlet
conditions imposed on $\cU^*$, or, equivalently, the number of
Neumann conditions exceeds $N_b-N_+=N^2_-$.

Our main spectral stability result is the following Theorem
established in Section \ref{limits}, which implies that Evans
stability of small-amplitude layers $w(z)$ is equivalent to Evans
stability of the constant-layer limit $w\equiv w(\infty)$, a {\it
linear-algebraic condition} that can in principle be computed by
hand. This is in sharp contrast to the shock wave case, for which
the small-amplitude limit is a complicated singular-perturbation
problem \cite{Met3,PZ,FS1,FS2}.

\begin{defi}[Small amplitude profiles]\label{cd1}
With $\cU_\partial$ as in \eqref{C1} define
\begin{align}\label{cb1}
\cU_{\partial,\nu}=\{(u,\nu(t,x_0)):(t,x_0,u)\in\cU_\partial\}.
\end{align}
For $\eps>0$ and any compact set $D\subset\cU_{\partial,\nu}$, the
set of \emph{$\eps$-amplitude profiles associated to $D$} is the set
of functions $w(z)=w(z,u,\nu)$  for which there exist $(\uu,\unu)\in
D$ such that:

a)\;$A_\nu(w)\partial_zw-\partial_z(B_\nu(w)\partial_zw)=0$  on
$z\geq 0$,

b)\;$w(z,u,\nu)\to u$ as $z\to \infty$,

c)\; $\|(w,w^2_z)-(\uu,0)\|_{L^\infty(0,\infty)}\leq \eps$,
$|\nu-\unu|\leq\eps$.

\noindent When $\eps$ is small we refer to such profiles as
\emph{small amplitude profiles}.
\end{defi}

\begin{rem}\label{gh1}
Observe that \emph{small amplitude profiles}  are defined without
specifying any boundary condition at $z=0$.    We define the Evans
function for such a $w(z,u,\nu)$ using the same formula as before
\eqref{evans1}, where $\bE^-(\zeta)$ and $\Gamma(\zeta)$ are now
defined using linearization of \eqref{vv2} around $w(z,u,\nu)$.

\end{rem}

\begin{theo}\label{smallred}

For any compact subset $D\subset \cU_{\partial,\nu}$ there exists an
$\eps>0$ such that the uniform Evans condition is satisfied for the
set of $\eps$-amplitude profiles associated to $D$ (Definition
\ref{cd1}) if and only if it is satisfied for the set of constant
layers $\{w(z,u,\nu):w=u\text{ for all } z, (u,\nu)\in D\}$.
\end{theo}

For any bounded set of frequencies the equivalence in the above
theorem can be checked by showing that the Evans function depends
continuously on parameters, including parameters that measure
deviation of profiles from constant layers.  Thus, the main
difficulty is to deal with the unbounded high frequency region. Here
our strategy, carried out in Theorem \ref{aa3}, is to show that high
frequency Evans stability for \eqref{linq4} (the second condition in
\eqref{vv51}) is equivalent to high frequency Evans stability of the
\emph{quasihomogeneous, decoupled, frozen coefficient} system
\eqref{linprinc}. By virtue of the quasihomogeneity and frozen
coefficients, the latter condition can be checked by relatively
simple energy estimates for a \emph{compact} set of frequencies,
namely the parabolic unit sphere.

As a corollary we obtain the following result, established by energy
estimates in Section \ref{stabsec}, which establishes uniform Evans
stability for small amplitude layers in a variety of situations.

\begin{cor}\label{symmstab}
(a)  In the strictly parabolic case ($N=N'$) the uniform Evans
condition is satisfied for sufficiently small-amplitude layers (in
the sense of Definition \ref{cd1}) of symmetric--dissipative systems
with pure Dirichlet boundary conditions,
$$
\rank \Up^3=0,
$$
or with pure Neumann boundary conditions when $\rank \Up^3=N=N^2_-$.

(b)In the partially parabolic case ($N'<N$), the uniform Evans
condition is satisfied for sufficiently small-amplitude layers of
symmetric--dissipative systems with pure Dirichlet boundary
conditions or with mixed boundary conditions when  $\mathrm{rank}
\Up^3=N'=N^2_-$ and $\overline{A}^{11}_\nu$ is \emph{totally
outgoing} ($N^1_+=0$).

(c)In the partially parabolic case when $\mathrm{rank}
\Up^3=N'=N^2_-$ and $\overline{A}^{11}_\nu$ is \emph{totally
incoming} ($N^1_+=N-N'$), Evans stability fails in general  even for
small amplitude profiles (see Example \ref{cegrmk}).

(d)The uniform Evans condition fails for sufficiently small
amplitude solutions with mixed boundary conditions whenever
$\mathrm{rank}\Up^3>N^2_-$; see Corollary \ref{smalltrans}.

\end{cor}

Combining Theorem \ref{existence}, Proposition \ref{C30},
and Corollary \ref{symmstab},
we obtain the following result asserting existence and stability of
small-amplitude boundary-layer solutions for symmetric--dissipative
systems with various types of boundary conditions.
Suppose we are given a smooth, global assignment of states
\begin{align}\label{ww1}
(t,x_0,p(t,x_0))\in\cU_\partial\text{ for all }(t,x_0)\in
[-T,T]\times \partial\Omega
\end{align}
satisfying the viscous boundary condition \eqref{viscbcd}:
\begin{align}\label{ww2}
(\Up_1(p^1(t,x_0)),\Up_2(p^2(t,x_0),\Up_3(p(t,x_0),0,0))=(g_1(t,x_0),g_2(t,x_0),0)
\end{align}
for each $(t,x_0)$. Note that the image of $p$ is compact by
compactness of $\partial\Omega$.

\begin{theo}\label{mainresult}
Consider a symmetric-dissipative system \eqref{vv2} with boundary
conditions of the type described in Corollary \ref{symmstab} (a),(b)
under hypotheses (H1)-(H6), but with (H4$'$) in place of (H4). Given
a smooth global assignment of states $p(t,x_0)$ as in
\eqref{ww1}--\eqref{ww2}, there exists a $\cC$ manifold satisfying
Assumption \eqref{C5} with $p(t,x_0)\in\cC(t,x_0)\subset
\pi\cU_\partial$ for all $(t,x_0)$, and associated small amplitude
profiles $W(z,t,x_0,q)$ satisfying the uniform Evans condition on
$[-T,T]\times\partial\Omega$. The manifold $\cC$ defines a residual
hyperbolic boundary condition as in \eqref{vv3}.

Given initial data $v^0$ satisfying appropriate corner compatibility
conditions for the hyperbolic problem \eqref{vv3}, there exists an
inviscid solution $u^0$ as in Theorem \ref{vv10}, an
approximate solution $u^\eps_a$ as in Proposition \ref{vv20}, and an
exact boundary layer solution $u^\eps$ satisfying all the
conclusions of Theorem \ref{existence} for constants $s_0$, $k$, $M$
as described there.
\end{theo}

\subsection{Application to \emph{identifying} small viscosity limits. }

The exact viscous solutions in Theorems \ref{existence} and
\ref{mainresult} are chosen to satisfy high-order corner
compatibility conditions at $t=0$ that depend on the approximate
solution $u^\eps_a$; the construction of $u^\eps_a$ depends in turn
on having a $\cC$ manifold with associated layer profiles and an
inviscid solution $u^0$ to start with.    In this section we show
how Theorem \ref{existence} can sometimes be used together with
Corollary \ref{symmstab} and our results on $\cC$ manifolds in
section \ref{existsec} to identify small viscosity limits of
solutions to viscous boundary problems like \eqref{vv2}, \emph{even
when neither the  $\cC$ manifold nor the inviscid solution is given
in advance}.

Consider the viscous problem \eqref{vv2} on a half-space
$\Omega=\{x\in\bR^d:x_d\geq 0\}$:
\begin{align}\label{vvvv2}
\begin{split}
&\cL_\eps(u):=
 A_0(u)u_t  + \sum_{j=1}^d   A_j(u) \D_{j} u    -
 \eps \sum_{j,k= 1}^d \D_{j} \big( B_{jk}(u) \D_{k} u \big) = 0,\\
&\Up(u,\partial_Tu^2,\partial_\nu u^2)=(g_1,g_2,0)\text{ on
}[-T,T]\times \partial\Omega,\\
&u=\uu\in\pi\cU_\partial \text{ in }t<0,
\end{split}
\end{align}
where $\uu$ is a constant state and the (nonconstant) boundary data
$(g^1(t,x_0),g^2(t,x_0),0)$ is $C^\infty$, equal in $t<0$ to the
constant
\begin{align}\label{vw3}
\Up(\uu,\partial_T\uu,\partial_\nu\uu)=\Up(\uu,0,0):=(\uug,0),
\end{align}
and also equal to $(\uug,0)$ outside a compact set in
$[-T,T]\times\Omega$. We suppose that \eqref{vv2} is a
symmetric-dissipative system satisfying (H1)-(H6) with (H4$'$) in
place of (H4), and that the boundary conditions are of the type
described in Corollary \ref{symmstab}(a)(b).

Set $\unu=(0,1)$ and note that the constant profile $w(z)\equiv \uu$
satisfies the uniform Evans condition by Corollary \ref{symmstab};
so in particular, it is transversal.  Now apply Proposition
\ref{redBC} to find  neighborhoods $\cO\subset\bR^N$ of $\uu$ and
$\mathbb{O}\subset\bR^{N_b-N''}$ of $\uug$,  smoothly varying
manifolds $\cC_{\unu,g}\subset \cO$ for $g\in\mathbb{O}$, and
transversal small-amplitude profiles
\begin{align}\label{vw4}
W_{\unu,g}(\cdot,q):[0,\infty)\times \cC_{\unu,g}\to \pi\cU_*
\end{align}
such that for each $q\in \cC_{\unu,g}$, $w=W_{\unu,g}(\cdot,q)$
satisfies
\begin{align}\label{vw5}
\begin{split}
&A_\unu(w)\partial_zw-\partial_z(B_\unu(w)\partial_zw)=0\text{ on
}z\geq 0\\
&\Up(w,0,w^2_z)(0)=(g,0)\\
&w(z)\to q\text{ as }z\to\infty.
\end{split}
\end{align}

Provided $T_0$ is sufficiently small, our assumptions on the
boundary data imply that
\begin{align}\label{vw6}
(g^1(t,x_0),g^2(t,x_0))\in\mathbb{O}\text{ for all }(t,x_0)\in
 [-T_0,T_0]\times\partial\Omega.
\end{align}
If we now define
\begin{align}\label{vw7}
\begin{split}
&\cC(t,x_0):=\cC_{\unu,g(t,x_0)}\text{ and }\\
&W(z,t,x_0,q):=W_{\unu,g(t,x_0)}(z,q)\text{ for }(t,x_0)\in
[-T_0,T_0]\times\partial\Omega,
\end{split}
\end{align}
the manifolds $\cC(t,x_0)$ and profiles $W(z,t,x_0,q)$ satisfy the
conditions of Assumption \ref{C5}. Moreover, by Corollary
\ref{symmstab} the uniform Evans condition holds on
$[-T_0,T_0]\times\partial\Omega$.

Next, recalling Lemma \ref{Rousset} we apply Theorem \ref{vv10} to
construct the unique solution $u^0(t,x)\in
H^{s_0}([-T_0,T_0]\times\Omega)$ to the inviscid hyperbolic problem
\begin{align}\label{vw8}
\begin{split}
&\cL_0(u^0)=0\text{ on }[-T_0,T_0]\times\Omega\\
&u^0(t,x_0)\in\cC(t,x_0)\text{ on }[-T_0,T_0]\times \partial\Omega,\\
&u^0=\uu\text{ in }t<0.
\end{split}
\end{align}
on a possibly shorter time interval. From \eqref{vw3} and
\eqref{vw7} we see that
\begin{align}\label{vw9}
\cC(t,x_0)=\cC_{\unu,\uug} \text{ in }t<0,
\end{align}
so corner compatibility conditions are satisfied in \eqref{vw8} to
infinite order.

We now apply Theorem \ref{existence} to obtain  approximate and
exact solutions $u^\eps_a$ and $u^\eps$ to the viscous problem
\eqref{vvvv2} satisfying the estimates \eqref{o24} for constants
$\eps_0$, $M$, $k$, and $s_0$ as in that Theorem.   Note that
$u^\eps$, $u^\eps_a$, and $u^0$ are all equal to $\uu$ in $t<0$.
\emph{Finally notice that smooth solutions $u^\eps$ to the initial
boundary value problem \eqref{vvvv2} are uniquely determined from
the start.}  They must therefore equal the solutions obtained from
Theorem \ref{existence} by the above procedure.  The estimates
involving $u^0$ in \eqref{o24} now allow us to identify the unique
$u^0$ solution to \eqref{vw8} as the small viscosity limit of the
$u^\eps$.

\subsection{Discussion and open problems}

Theorem \ref{mainresult} generalizes to the case of ``real'', or
partially parabolic viscosity, the small-amplitude results obtained
in \cite{GG} for Laplacian second-order terms, and to
multi-dimensions those obtained in \cite{R3} for one-dimensional
symmetric--dissipative systems. It includes as a physical
application existence of small-amplitude boundary-layer solutions
for the equations of compressible gas dynamics with specified in- or
outflow velocity, temperature, and, in the inflow case, specified
density or pressure, a result analogous to those obtained in
\cite{TW} for the incompressible case. It includes also the
corresponding result for the compressible MHD equations with
specified inflow or outflow velocity, temperature, magnetic field,
and, in the inflow case, density or pressure, for parameter regimes
satisfying the structural conditions described in Remark \ref{yy1}.
See Sections \ref{Isentropic}, \ref{full}, and \ref{MHDsec} for
physical examples. For a general physical survey of boundary-layer
behavior, see \cite{S}.

Together with pointwise Green function analyses \cite{YZ,NZ} showing
that uniform Evans stability is sufficient for long-time (i.e.,
time-asymptotic) stability of planar boundary-layer profiles,
Theorem \ref{smallred} yields also long-time stability of
small-amplitude planar profiles with sharp pointwise rates of
decay, sharpening previous results obtained by Rousset \cite{R3}
by energy methods.

Theorem \ref{existence} generalizes to the case of real viscosity
the large-amplitude results obtained in \cite{MZ1} and \cite{GR} for
strictly parabolic viscosities in the multi- and one-dimensional
case, respectively, giving a sharp criterion for existence and
stability of boundary-layer solutions in the small-viscosity limit.
Determination of Evans stability in the large-amplitude case is an
outstanding open problem. (The uniform Evans condition may fail in
general for large-amplitude layers, as demonstrated in \cite{SZ}.)
Numerical testing of the Evans condition for large-amplitude layers
in multi-dimensions would be an interesting direction for further
investigation; see \cite{CHNZ,HLyZ1} for the one-dimensional case,
\cite{HLyZ2} for the multi-dimensional shock case.

We note that small-amplitude stability for symmetric--dissipative
systems might be provable for variable-multiplicity systems under
weaker structural assumptions than those of Theorem
\ref{mainresult}, which are tailored for large-amplitude layers, by
direct energy estimates as in \cite{GG,R3} rather than by first
passing to the constant-layer limit.  This approach becomes quite
complicated in the multidimensional case, but would yield existence
without hard-to-verify structural conditions, and would apply to
some physical cases such as MHD for parameter regimes different from
that described in Remark \ref{yy1}.

 Finally, we discuss the meaning
of the somewhat unexpected instability result of Propositions
\ref{C8a} and Corollary \ref{symmstab}, parts (c) and (d).   In
cases where $N''>N^2_-$, or equivalently, when the number of scalar
Dirichlet conditions imposed in the viscous problem ($N'-N''+N^1_+$)
is strictly less than the number of scalar boundary conditions for
the residual hyperbolic problem ($N_+$), this seems to indicate a
failure of our basic Ansatz, which assumes that the residual
hyperbolic problems should involve only \emph{Dirichlet} boundary
conditions. For example, in the extreme case $N=N'$ (strict
parabolicity), $N_+=N=N_b$ (so $N^2_-=0$ and all characteristics are
incoming), with full Neumann boundary conditions $\D_\nu u(0)=0$ on
$\partial \Omega$ (so $N''=N$), work in progress indicates that
solutions converge in the small-viscosity limit to solutions of the
inviscid problem $\cL_0(u^0)=0$ with \emph{Neumann boundary
conditions} $\D_\nu u(0)=0$ on $\partial \Omega$ , and with no
intervening boundary layer. In other cases where $N''>N^2_-$, we
conjecture that the correct model for limiting behavior is a
residual hyperbolic problem with mixed Dirichlet--Neumann conditions
of appropriate ranks. We plan to address this issue in a future
work.

\medskip
{\bf Plan of the paper.} In Section \ref{existsec} we investigate
the existence and transversality of boundary layers, in particular
in the small-amplitude limit, and construct both small and large
amplitude $\cC$-manifolds. In Section \ref{limits}, we review the
construction of the Evans function, and examine its high-frequency
and small-amplitude limits, establishing the key reduction of
Theorem \ref{smallred} for small-amplitude profiles. In Section
\ref{stabsec}, we establish Corollary \ref{symmstab} for symmetric
dissipative systems using energy estimates. In Section
\ref{maxdiss}, we digress slightly to show maximal dissipativity of
hyperbolic boundary conditions associated with small-amplitude
layers of symmetric dissipative systems when full Dirichlet
conditions are imposed in the viscous problem ($N''=0$).  In
Sections \ref{Isentropic}--\ref{MHDsec}, we carry out explicit
computations for the example systems of isentropic gas dynamics,
full gas dynamics, and MHD. The construction of approximate
solutions is presented in Appendix \ref{approximate}. The tracking
lemma and its connection to construction of high-frequency
symmetrizers are given in Appendix \ref{trackapp}.

\begin{nota}
We do not distinguish between $u^2$ and $u_2$, $J^*$ and $J_*$,
etc.. Sometimes, especially when other subscripts or superscripts
are involved, one choice is more convenient than the other (e.g.,
$u^2_z$).

\end{nota}

\section{Existence of $\cC$-manifolds and layer profiles}\label{existsec}

In this section we will show that it is possible to construct smooth
$\cC$-manifolds as in Assumption \ref{C5}, globally defined on
$\partial\Omega$ with corresponding global smooth families of
\emph{small-amplitude} layer profiles.  We will also show that it is
possible to give a local verification of Assumption \ref{C5} if one
starts with a given, possibly large amplitude, transversal layer
profile.

\subsection{Global $\cC$-manifolds and families of profiles in the
small-amplitude case}\label{globalC}

Here we verify Assumption \ref{C5} in the small-amplitude case for a
variety of boundary conditions of the form \eqref{viscbcd}.
We begin by defining a family of constant layers by giving a smooth presciption
of states $p(t,x_0)\in\mathbb{R}^N$ with
\begin{align}\label{C9}
(t,x_0,p(t,x_0))\in\cU_\partial\text{ for all }(t,x_0)\in
[-T,T]\times\partial\Omega.
\end{align}
We shall not try to show that such smooth global prescriptions are
always possible under structural Assumption \ref{Hs}, but they
clearly exist for the Navier-Stokes and MHD equations (see sections
\ref{Isentropic}, \ref{full}, and \ref{MHDsec}), where one has a
simple characterization of the domains of hyperbolicity and
non-characteristicity in terms of physical quantities like pressure
and velocity.   The prescription \eqref{C9} determines a choice of
boundary data in \eqref{viscbcd} or \eqref{layeq}, namely
\begin{align}\label{C10}
(g_1(t,x_0),g_2(t,x_0),0):=(\Up_1(p^1(t,x_0)),\Up_2(p^2(t,x_0)),0).
\end{align}

\begin{rem}\label{smoothtrivial}
When $\Omega$ is a half-space, so that $\nu$ is constant,
prescription \eqref{C9} is trivially constructed,
consisting of a single state $p$.
\end{rem}

Given $p(t,x_0)$ as in \eqref{C9}, define the compact set
\begin{align}\label{C11}
B:=\{(\nu(x_0),p(t,x_0)):(t,x_0)\in
[-T,T]\times\partial\Omega\}\subset S^{d-1}\times\pi\cU_\partial.
\end{align}
Fix $(\underline{\nu},\underline{p})\in B$.  The next Proposition,
which is a slight modification of Proposition 5.3.5 in \cite{Met4},
characterizes all possible ``small-amplitude" solutions $w(z)$ of
\begin{align}\label{C12}
\begin{split}
&A_\nu(w)\partial_zw-\partial_z(B_\nu(w)\partial_zw)=0\text{ on }z\geq 0,\\
&w(z)\to q\text{ as }z\to \infty
\end{split}
\end{align}
for $(\nu,q)$ near $(\underline{\nu},\underline{p})$.   Define the
$N'\times N'$ matrix
\begin{align}\label{C13}
G_\nu(q):=(B^{22}_{\nu})^{-1} \left( A_\nu^{22} - A_\nu^{21} (
A_\nu^{11})^{-1} A_\nu^{12}\right)(q)
\end{align}
and let $\bE_\mp(G_\nu(q))$ denote the generalized eigenspace of
$G_\nu(q)$ associated to eigenvalues $\mu$ with $\pm\Re\mu<0$.
Denote by $\Pi_{\nu\pm}(q)$ the projections associated to the
decomposition
\begin{align}\label{C14b}
\bR^{N'}=\bE_+(G_\nu(q))\oplus\bE_-(G_\nu(q)),
\end{align}
and fix isomorphisms $\alpha(\nu,q;a)$ linear in
$a\in\bE_-(G_\nu(\up))$ and $C^\infty$ in $(\nu,q)$:
\begin{align}\label{C14}
\alpha(\nu,q;a):\bE_-(G_\nu(\up))\to\bE_-(G_\nu(q))
\end{align}
such that $\alpha(\nu,\up;a)=a$.

\begin{prop}\label{C14a}
There exists a neighborhood $\omega\subset S^{d-1}\times\bR^N$ of
$(\underline{\nu},\underline{p})$ and constants $R>0$, $r>0$ such
that for $(\nu,q)\in\omega$, all solutions $w$ of \eqref{C12}
satisfying
\begin{align}\label{C15a}
\|(w, w^2_z)-(\up, 0)\|_{L^\infty[0,\infty]}\leq R,
\end{align}
are parametrized by a $C^\infty$ function $w=\Phi(z,\nu,q,a)$ on
$[0,\infty)\times \omega^*$, where $\omega^*$ is the set of
$(\nu,q,a)$ with $(\nu,q)\in\omega$ and $a\in\bE_-(G_\nu(\up))$ with
$|a|\leq r$.  The function $\Phi(z,\nu,q,a)$ is the unique solution
of \eqref{C12} satisfying the boundary condition
\begin{align}\label{C15b}
\Pi_{\nu-}w^2_z(0)=\alpha(\nu,q;a),
\end{align}
and $\Phi^2$ has the expansion
\begin{align}\label{C16}
\Phi^2(z,\nu,q,a)=q^2+e^{zG_\nu(q)}G_\nu^{-1}(q)\alpha(\nu,q;a)+O(|a|^2)
\end{align}
uniformly with respect to $(z,q,\nu)$. Moreover, there exist
positive constants $\delta$ and $C$ such that for all $z\in
[0,\infty)$ and $(q,\nu,a)\in\omega^*$:
\begin{align}\label{C17}
|\partial_z\Phi^2(z,\nu,q,a)|+|\Phi(z,\nu,q,a)-q|\leq Ce^{-\delta
z}.
\end{align}
We also denote by $\Phi(z,\nu,q,a)$ the maximal extension of $\Phi$
to $z<0$ as a solution of \eqref{C12}.
\end{prop}

\begin{proof}
\textbf{1. }First we claim there exists a neighborhood
$\omega'\subset S^{d-1}\times\bR^N$ of
$(\underline{\nu},\underline{p})$ and constants $R'>0$, $r'>0$ such
that for $(\nu,q)\in\omega'$, all solutions $w$ of \eqref{C12}
satisfying
\begin{equation}\label{stronger}
\|w^2_z\|_{L^1[0,\infty]}\leq R'\text{ and }
\|w^2_z\|_{L^\infty[0,\infty]}\leq R',
\end{equation}
are parametrized by a $C^\infty$ function $w=\Phi(z,\nu,q,a)$ on
$[0,\infty)\times \omega'_*$, where $\omega'_*$ is the set of
$(\nu,q,a)$ with $(\nu,q)\in\omega'$ and $a\in\bE_-(G_\nu(\up))$
with $|a|\leq r'$.  This may be established by a contraction mapping
argument identical to that given in Proposition 2.2 and Appendix A
(both) of \cite{GMWZ7}. This argument corrects a minor error in
\cite{Met4}, Proposition 5.3.5 and extends that Proposition to the
case of partial viscosity.

\textbf{2. }For $\nu$ near $\unu$, after shrinking $\omega'$ and
$r'$ if necessary and renaming as $\omega$ and $r$, the map
$(q,a)\to(\Phi,\Phi^2_z)(0,\nu,q,a)$ defined for $(\nu,q)\in\omega$,
$|a|\leq r$, defines a diffeomorphism onto the local center--stable
manifold of $(\underline p,0)$ for \eqref{C12} considered as a
first-order system \eqref{layeq2}.  On the other hand for $\nu$ near
$\unu$ all solutions of \eqref{C12} for which
\begin{align}\label{ef2}
\|(w, w^2_z)-(\up, 0)\|_{L^\infty[0,\infty)}\text{ is small}
\end{align}
 lie on that center-stable manifold.  Thus, for $R$ small
enough the assertion in the Proposition holds.
\end{proof}

\begin{rem}\label{CS}
Alternatively, the result of Proposition \ref{C14a}
may be obtained directly by invariant manifold
theory, working with the first-order system \eqref{layeq2}.
For, noting that any constant function is an equilibrium
of the system, and recalling that equilibria lie on any
center manifold, we find by a dimensional count
that the center manifold of the system consists entirely
of equilibria, and the center--stable manifold is foliated
by the union of stable manifolds through each equilibrium
(constant state).  Thus, the only profiles satisfying \eqref{C15a}
are those lying on stable manifolds of rest points $q$, whence
we obtain both the parametrization by $\Phi$ and the decay estimate
\eqref{C17} by an application of the Stable Manifold Theorem.
\end{rem}

The next Proposition gives necessary and sufficient conditions on
the boundary conditions $\Up(w,0,\partial_z w^2)$ in \eqref{layeq}
for the \emph{local} (with respect to $(\nu,p)$) existence of
transversal profiles and corresponding $\cC$ manifolds.  The local
objects will then be patched together using an argument based on
local uniqueness.

For $(\nu,q,a,p)$ near $(\underline{\nu},\up,0,\up)$, define
\begin{align}\label{C18}
\Psi(\nu,q,a,p):=\Up(\Phi(0,\nu,q,a),0,\partial_z\Phi^2(0,\nu,q,a))-(\Up_1(p^1),\Up_2(p^2),0)
\end{align}
Observe that
$\Psi(\underline{\nu},\up,0,\up)=0$
and that every solution of \eqref{C12} which also satisfies
\begin{align}\label{C20}
\Up(w,0,w^2_z)(0)=(\Up_1(p^1),\Up_2(p^2),0)
\end{align}
corresponds to a solution of
$\Psi(\nu,q,a,p)=0$ and vice versa.
Using \eqref{layeq2} and the expansion \eqref{C16}, we readily
compute the $N_b\times(N+N')$ derivative matrix
\begin{align}\label{C22}
\begin{split}
&\Psi_{q^1,q^2,a}(\underline{\nu},\up,0,\up)=\begin{pmatrix}\Up_1'(\up^1)&0&\Up_1'(\up^1)\left(-(A^{11}_{\underline{\nu}})^{-1}(\up)A^{12}_{\underline{\nu}}(\up)G^{-1}_{\underline{\nu}}(\up)\right)\\
0&\Up_2'(\up^2)&\Up_2'(\up^2)G_{\underline{\nu}}^{-1}(\up)\\0&0&K_{\underline{\nu}}\end{pmatrix},
\end{split}
\end{align}
where, for example, the matrix entries in the third column, reading
down, have sizes $N^+_1\times N'$, $(N'-N'')\times N'$, and
$N''\times N'$ respectively.

\begin{prop}\label{C22prop}
(a) For $B$ as in \eqref{C11}, let $(\unu,\up)\in B$.  The constant
layer $\Phi(z,\underline{\nu},\up,0)=\up$ is transversal if and only
if
\begin{align}\label{C23}
\begin{split}
&(i)\text{the $N_b\times N'$ third column of \eqref{C22} is
injective on
}\bE_-(G_{\underline{\nu}}(\up)),\text{ and }\\
&(ii)\text{ if }N''>0, K_{\underline{\nu}}\text{ is of full rank
}N''\text{ on }\bE_-(G_{\underline{\nu}}(\up)).
\end{split}
\end{align}
(b)   Suppose \eqref{C23} holds.  There is a neighborhood
$\omega\subset S^{d-1}\times \pi\cU_\partial$ of $(\unu,\up)$ and
for each $(\nu,p)\in\omega$, there is a manifold $\cC_{\nu,p}$ of
dimension $N-N_+$ and a smooth map
\begin{align}\label{C24}
w_{\nu,p}:[0,\infty)\times \cC_{\nu,p}\to\pi\cU_\partial,
\end{align}
such that for each $q\in\cC_{\nu,p}$, $w_{\nu,p}(\cdot,q)$ satisfies
\eqref{C12},\eqref{C20} and converges at an exponential rate to $q$
as $z\to\infty$.  Moreover, the manifolds $\cC_{\nu,p}$ vary
smoothly with $(\nu,p)\in\omega$.

(c)The endstate-manifolds $\cC_{\nu,p}$ and profiles
$w_{\nu,p}(\cdot,q)$ are uniquely determined by this construction
for $(q,\nu,p)$ near $(\up,\unu,\up)$.  More precisely, when
$(\nu,p)$ lies in charts centered at two distinct base points
$(\unu_k,\up_k)\in B$, $k=1,2$, the corresponding manifolds
$\cC_{\nu,p}^k$ are the same near
$p\in\cC_{\nu,p}^1\cap\cC_{\nu,p}^2$ and for each
$q\in\cC_{\nu,p}^1\cap\cC_{\nu,p}^2$, the profiles
$w_{\nu,p}^k(\cdot,q)$ constructed in the separate charts coincide.
\end{prop}

\begin{proof}
\textbf{(a). }   The first transversality condition in Definition
\ref{deftrans} is equivalent to injectivity of
$\Psi_a(\underline{\nu},\up,0,\up)$ on $\bE_-(G_{\unu}(\up))$, while
the second transversality condition there is equivalent to
surjectivity of
\begin{align}\label{C25}
\Psi_{q,a}(\unu,\up,0,\up):\bR^{N}\times
\bE_-(G_{\unu}(\up))\to\bR^{N_b}
\end{align}
(for more detail see \cite{Met4}, Prop. 5.5.3).  Condition (i) in
\eqref{C23} is equivalent to the first of these conditions, and in
view of Assumption \ref{assbc} condition (ii) is equivalent to the
second.

\textbf{(b). }Since $\Psi_a(\underline{\nu},\up,0,\up)$ has rank
$N^2_-$ on $\bE_-(G_{\unu}(\up))$ and $\Psi_{q,a}$ as in \eqref{C25}
has rank $N_b=N_++N^2_-$, the
Implicit Function Theorem implies that
there exist smooth functions $q(q_-,\nu,p)$, $a(q_-,\nu,p)$, where
$q_-\in\bR^{N-N_+}$ is a vector consisting of $N-N_+$ of the
coordinates of $q$, such that the solutions of $\Psi(\nu,q,a,p)=0$
near $(\unu,\up,0,\up)$ are given precisely by
\begin{align}\label{C26}
(\nu,q(q_-,\nu,p),a(q_-,\nu,p),p)\text{ for }(q_-,\nu,p)\text{ near
}(\up_-,\unu,\up).
\end{align}
For each $(\nu,p)$ the manifold $\cC_{\nu,p}$ is defined by
$q=q(q_-,\nu,p)$ and the profiles are given by
\begin{align}\label{C27}
w_{\nu,p}(z,q)=\Phi(z,\nu,q(q_-,\nu,p),a(q_-,\nu,p)).
\end{align}

\textbf{(c). }Suppose $(\nu,p)$ lies in charts centered at two
different base points $(\unu_k,\up_k)\in B$, $k=1,2$.  Let
$\cC_{\nu,p}^k$ and $w_{\nu,p}^k$ denote the corresponding manifolds
and profiles.  The properties of the functions $\Phi$ and $\Psi$
described in Proposition \ref{C14a} and the discussion following
\ref{C18} show that each of $\cC_{\nu,p}^k$, $k=1,2$ coincides near
$p$ with the set of $q$ such that there exists a $w(z)$ satisfying
\eqref{C12}, \eqref{C20}, and
\begin{align}\label{ef3}
\|(w,w^2_z)-(q,0)\|_{L^\infty(0,\infty]} \text{ is small }.
\end{align}
This description is chart-independent so the manifolds must agree
near $p$.

Suppose $q\in\cC_{\nu,p}^1\cap\cC_{\nu,p}^2$.  Working in chart 2
and using the properties of the functions $\Phi$ and $\Psi$ just
referred to, we conclude that \emph{any} small amplitude profile
satisfying  \eqref{C12}, \eqref{C20}, and \eqref{ef3} must be given
by $w^2_{\nu,p}(z,q)$. In particular, we must have
$w^1_{\nu,p}(z,q)=w^2_{\nu,p}(z,q)$.

\end{proof}

\begin{cor}\label{smalltrans}
\textbf{1. }For pure Dirichlet conditions ($N''=\rank
\Upsilon_3=0$), sufficiently small-amplitude layers are transversal,
and likewise for mixed Dirichlet--Neumann conditions in the extreme
case $\rank \Upsilon_3=N^2_-=N'$.

\textbf{2. }When the number of Neumann boundary conditions $\rank
\Upsilon_3$ exceeds $N^2_-$, constant layers are non-transversal.
Equivalently, a necessary condition for transversality of constant
layers is that the number of (scalar) Dirichlet conditions for the
parabolic problem, $(N'-N'')+N^1_+$, is greater than or equal to the
number of Dirichlet conditions for the residual hyperbolic problem,
$N_+$.
\end{cor}

\begin{proof}
When $N''=0$, the condition \eqref{C23}(i) holds since
$\Up_2'(\up^2)G_{\underline{\nu}}^{-1}(\up)$ is an invertible
$N'\times N'$ matrix.   When $N''=N'=N^2_-$, Assumption \ref{assbc}
together with the fact that $\dim\bE_-(G_{\unu}(\up))=N^2_-$ imply
that $K_\nu$ is invertible on $\bE_-(G_{\unu}(\up))$. Thus, both
conditions in \eqref{C23} hold. When $N''>N^2_-$, it is impossible
for \eqref{C23}(ii) to hold.   The final assertion follows by noting
\begin{align}\label{C28}
\begin{split}
&N_b=N''+(N'-N'')+N^1_+=N_++N^2_-\text{ so }\\
&N''\leq N^2_-\Leftrightarrow (N'-N'')+N^1_+\geq N_+.
\end{split}
\end{align}
\end{proof}

\begin{prop}\label{C30}
Given a smooth global assignment of states
\begin{align}\label{C31}
(t,x_0,p(t,x_0))\in\cU_\partial\text{ for all }(t,x_0)\in
[-T,T]\times\partial\Omega,
\end{align}
suppose that transversality (equivalently, condition \eqref{C23} of Proposition \ref{C22prop}) holds at
every point of
\begin{align}\label{C32}
B:=\{(\nu(x_0),p(t,x_0)):(t,x_0)\in
[-T,T]\times\partial\Omega\}\subset S^{d-1}\times\pi\cU_\partial.
\end{align}
Then for boundary data
\begin{align}\label{C33}
(g_1(t,x_0),g_2(t,x_0),0):=\left(\Up_1(p^1(t,x_0)),\Up_2(p^2(t,x_0)),0\right),
\end{align}
there exists a global $\cC$ manifold and family of profiles $W$
satisfying the conditions of Assumption \ref{C5}.
\end{prop}

\begin{proof}
Use the compactness of $B$ to cover $B$ with a finite number of open
patches centered at basepoints $(\unu,\up)_k$.  Carry out the
construction of Proposition \ref{C22prop}, part (b), in each patch and
use the local uniqueness described in part (c) to define  manifolds
$\cC_{\nu,p}$ and profiles $w_{\nu,p}(\cdot,q)$  for all $(\nu,p)$
in the covering and smoothly varying with $(\nu,p)$. Then the
conditions of Assumption \ref{C5} are satisfied by taking
\begin{align}\label{C34}
\begin{split}
&\cC(t,x_0):=\cC_{\nu(x_0),p(t,x_0)}\\
&W(z,t,x_0,q):=w_{\nu(x_0),p(t,x_0)}(z,q).
\end{split}
\end{align}
\end{proof}

\begin{rem}\label{C35}
\textbf{1. }Part (a) of Proposition \ref{C22prop} can be used to produce
many examples of transversal constant layers involving mixed
Dirichlet-Neumann conditions for any triple $(N'',N^2_-,N')$
satisfying $0\leq N''\leq N^2_-\leq N'$. Together with Proposition
\ref{C30}, this yields examples  where Assumption \ref{C5} is
satisfied for the Navier-Stokes and MHD systems with a variety of
boundary conditions.

\textbf{2. }  In the case $\rank \Upsilon_3=N^2_-=N'$, the matrix
$K_\nu$ is an invertible $N'\times N'$ matrix, so the $\Up_3$
boundary condition in \eqref{layeq} is equivalent to $w^2_z(0)=0$.
With \eqref{layeq2} this shows that the only solutions of the
profile ODE \eqref{layeq} are constant layers.  Since $N'-N''=0$,
$\Up_2$ is absent and we have
\begin{align}\label{C36}
\cC(t,x_0)=\{q:\Up_1(q)=g^1(t,x_0)\}.
\end{align}
Thus, the residual hyperbolic boundary conditions in this case are
the same as the Dirichlet boundary conditions for the parabolic
problem, with Neumann conditions ignored.  The leading term in the
approximate solution \eqref{vv20} is given by
$\cU^0(t,x,\frac{z}{\eps})=u^0(t,x)$, the hyperbolic solution with
no intervening boundary layer.

\end{rem}

\subsection{Local $\cC$-manifold associated to a given transversal profile}
\begin{prop}\label{redBC}
Suppose that $w$ is a given, not necessarily small amplitude,
transversal layer profile satisfying
\begin{align}\label{pp8}
\begin{split}
&(a)\; A_\unu(w)\partial_zw-\partial_z(B_\unu(w)\partial_zw)=0\text{
on
}z\geq 0\\
&(b) \;\Up(w,0,w^2_z)(0)=(\uug^1,\uug^2,0):=(\uug,0)\\
&(c) \;w(z)\to p \text{ as }z\to\infty.
\end{split}
\end{align}
Then in  a neighborhood $\cO\subset\bR^N$ of $p$ and for $(\nu,g)$
near $(\unu,\uug)$, there is a smooth manifold $\cC_{\nu,g}\subset
\cO$ of dimension $N-N_+$ and a smooth map
\begin{align}\label{pp9}
W_{\nu,g}:[0,\infty)\times \cC_{\nu,g}\to\pi\cU^*,
\end{align}
such that for each $q\in\cC_{\nu,g}$, $W_{\nu,g}(\cdot,q)$ satisfies
\eqref{pp8}(a),(b) with $(\nu,g,q)$ in place of $(\unu,\uug,p)$ and
converges at an exponential rate to $q$ as $z\to\infty$.
Moreover,  these manifolds and profiles vary smoothly as $(\nu,g)$
varies near $(\unu,\uug)$.
\end{prop}

\begin{proof}
\textbf{1. }For $\Phi(z,\nu,q,a)$ and $\alpha(\nu,q;a)$ as in
Proposition \ref{C14a},  the given profile must satisfy for some
sufficiently large $z_0$
\begin{align}\label{pp10}
w(z)=\Phi(z-z_0,\unu,p,\ua)\text{ for }\ua\text{ such that
}\Pi_{\unu,-}w^2_z(z_0)=\alpha(\unu,p;\ua).
\end{align}
This follows from the fact that for some $z_0$ sufficiently large,
the condition \eqref{C15a} of that Proposition holds with
$[0,\infty)$ replaced by $[z_0,\infty)$ (see Prop. 5.3.6 of
\cite{Met4} for details).

\textbf{2. }For $(\nu,q,a,g)$ near $(\underline{\nu},p,\ua,\uug)$,
instead of \eqref{C18} we now define
\begin{align}\label{pp11}
\Psi(\nu,q,a,g):=\Up(\Phi(0,\nu,q,a),0,\partial_z\Phi^2(0,\nu,q,a))-(g^1,g^2,0)
\end{align}
and observe that $\Psi(\unu,p,\ua,\uug)=0$. Transversality of $w$
implies that
\begin{align}\label{pp111}
\begin{split}
&\Psi_a(\unu,p,\ua,\uug):\bE_-(G_{\unu}(p))\to\bR^{N_b}\text{ and
}\\
&\Psi_{q,a}(\unu,p,\ua,\uug):\bR^N\times
\bE^-(G_\unu(p))\to\bR^{N_b}
\end{split}
\end{align}
have ranks $N^2_-$ and  $N_b=N_++N^2_-$ respectively.
The Implicit Function theorem implies that there exist smooth, locally unique, functions
$q(q_-,\nu,g)$, $a(q_-,\nu,g)$, where $q_-\in\bR^{N-N_+}$ is a
vector consisting of $N-N_+$ of the coordinates of $q$, such that
the solutions of $\Psi(\nu,q,a,g)=0$ near $(\unu,p,\ua,\uug)$ are
given precisely by
\begin{align}\label{pp13}
(\nu,q(q_-,\nu,g),a(q_-,\nu,g),g)\text{ for }(q_-,\nu,g)\text{ near
}(p_-,\unu,\uug).
\end{align}
For each $(\nu,g)$ the manifold $\cC_{\nu,g}$ is defined by
$q=q(q_-,\nu,g)$ and the profiles are given by
\begin{align}\label{C277}
W_{\nu,g}(z,q)=\Phi(z-z_0,\nu,q(q_-,\nu,g),a(q_-,\nu,g)).
\end{align}
\end{proof}

\subsection{Global $\cC$-manifold for a family of transversal
profiles}\label{globallargeC}

Similarly as in \eqref{C9} for the small-amplitude case,
assume that we are given a smooth prescription of large-amplitude
profiles, in the form of a $C^\infty$ function
\begin{align}\label{n3}
w(z,t,x_0):[0,\infty)\times [-T,T]\times
\partial\Omega\to \bR^N
\end{align}
that defines a transversal layer profile for each $(t,x_0)$:
\begin{align}\label{n1}
\begin{split}
&(a)\;A_{\nu(x_0)}(w)\partial_zw-\partial_z(B_{\nu(x_0)}\partial_z
w)=0\\
&(b)\;\Up(w,0,w^2_z)_{|z=0}=(\Up^1(w^1(0,t,x_0)),\Up^2(w^2(0,t,x_0)),0)\\
&\qquad \qquad \qquad \qquad
:=(g^1(t,x_0),g^2(t,x_0),0)\\
&(c) \;w(z)\to w(\infty,t,x_0):=q(t,x_0)\text{ as }z\to\infty.
\end{split}
\end{align}

\begin{rem}\label{lsmoothtrivial}
When $\Omega$ is a half-space, so that $\nu$ is constant,
assignment \eqref{C31} exists trivially, consisting of a single profile.
\end{rem}

\begin{cor}\label{largeC}
Given a smooth transversal family \eqref{n3},
there is a smooth  manifold $\cC$ defined as the graph
\begin{align}\label{n4}
\cC=\{(t,x_0,\cC(t,x_0)):(t,x_0)\in [-T,T]\times
\partial\Omega\}\subset \cU_\partial,
\end{align}
where $\cC(t,x_0)\subset \bR^N$ is an $N-N_+$ dimensional manifold
containing $q(t,x_0)$ and consisting of states $r$ near $q(t,x_0)$
for which there exists a transversal layer profile
\begin{align}\label{n6}
W(z,t,x_0,r):[0,\infty)\times \cC\to \pi\cU^*
\end{align}
satisfying \eqref{n1} with $q(t,x_0)$ in (c) replaced by $r$ and
\begin{align}
\|W(z,t,x_0,r)-w(z,t,x_0),
\partial_z(W^2(z,t,x_0,r)- w^2(z,t,x_0))\|_{L^\infty} \text{
small }.
\end{align}

  The
profiles $W(z,t,x_0,r)$  are $C^\infty$ in all arguments and
converge to their endstates at an exponential rate that can be taken
uniform on compact subsets of $\cC$.

\end{cor}

\begin{proof}

\textbf{1. }Cover the compact set
$B=\{(\nu(x_0),q(t,x_0)):(t,x_0)\in [-T,T]\times \partial\Omega\}$
by a finite number of charts centered at points $(\nu,q)_k$ for
which we have defined
\begin{align}
\alpha(\nu,q;a):\bE^-(G_{\nu_k}(q_k))\to\bE^-(G_{\nu}(q))
\end{align}
as in \eqref{C14}.  Fix $(t,x_0)$. As in the proof of Proposition
2.8, we have for some $k$
\begin{align}\label{na1}
w(z,t,x_0)=\Phi(z-T,\nu(x_0),q(t,x_0),\ua),\;\;\ua\in\bE^-(G_{\nu_k}(q_k))
\end{align}
for some large $T$, where $\ua$ is such that
\begin{align}\label{na2}
\Pi_{\nu(x_0),-}(q(t,x_0))\;\partial_zw^2(T,t,x_0)=\alpha(\nu(x_0),q(t,x_0),\ua)
\in\bE^-(G_{\nu(x_0)}(q(t,x_0))),
\end{align}
with $\Pi_{\nu,-}(q)$ the projection of $\bC^{N'}$ onto
$\bE^-(G_\nu(q))$ along $\bE^+(G_\nu(q))$.

\textbf{2. }Setting  $W(z)=(w(z),w^2_z(z))$ we rewrite \eqref{n1}(a)
as a first-order system:
\begin{align}\label{n99}
\partial_zW=\cH_{\nu(x_0)}(W).
\end{align}
The stable manifold of $(q(t,x_0),0)$ for \eqref{n99} \emph{near
$(w(z,t,x_0),w^2_z(z,t,x_0))$ for any fixed $z\in[0,\infty)$}, can
be parametrized:
\begin{equation}\label{na5}
\begin{aligned}
\cW^s&(q(t,x_0);\nu(x_0);z)=\\
&\{\begin{pmatrix}\Phi(z-T,\nu(x_0),q(t,x_0),a)\\\Phi^2_z(z-T,\nu(x_0),q(t,x_0),a)\end{pmatrix}
: a\in \bE^-(G_{\nu_k}(q_k))\text{ near }\ua\}
\end{aligned}
\end{equation}
Similarly, the center-stable manifold of $(q(t,x_0),0)$ near
$(w(z,t,x_0),w^2_z(z,t,x_0))$ is
\begin{align}\label{na7}
\begin{split}
&\cW^{cs}(q(t,x_0);\nu(x_0);z)=\\
&\qquad\{\begin{pmatrix}\Phi(z-T,\nu(x_0),q,a)\\\Phi^2_z(z-T,\nu(x_0),q,a)\end{pmatrix}:q\text{
 near }q(t,x_0), a\in \bE^-(G_{\nu_k}(q_k))\text{ near }\ua\}.
\end{split}
\end{align}
If $z= T$ we get the parts of the respective manifolds near
$$
(w(T,t,x_0),w^2_z(T,t,x_0)).
$$
If $z= 0$ we get the parts near $(w(0,t,x_0),w^2_z(0,t,x_0))$. (Here
we use that $\Phi(z,\nu,q,a)$ is a maximal extension to $z<0$).
\emph{Note that for any fixed $z$, these manifolds are uniquely
determined locally near $(w(z,t,x_0),w^2_z(z,t,x_0))$.}

\textbf{3. }The set
\begin{align}\label{na10}
\{(r,0):r\in\cC(t,x_0)\}
\end{align}
is, near $(w(\infty,t,x_0),0)$, the set of endstates at infinity
under the flow of \eqref{n99} of the $N-N_+$ dimensional initial
manifold
\begin{align}\label{na8}
\begin{split}
&\bC_{initial}(t,x_0):=\\
&\quad \{U=(u^1,u^2,u^3):(\Up_1(u^1),\Up_2(u^2),0)=(g^1(t,x_0),g^2(t,x_0),0)\}\cap\\
&\quad\quad\cW^{cs}(q(t,x_0);\nu(x_0);0)\subset\bR^{N+N'},
\end{split}
\end{align}
where we have evaluated \eqref{na7} at $z=0$.  As a check observe
that the intersection \eqref{na8} is transversal and has dimension
\begin{align}\label{na9}
(N+N'-N_b)+(N+N^2_-)-(N+N')=N-N_+.
\end{align}
Considering the uniqueness of $\cW^{cs} (q(t,x_0);\nu(x_0);0)$ near
$$
(w(0,t,x_0),w^2_z(0,t,x_0)),
$$
this description of $\cC(t,x_0)$ establishes that it is uniquely
determined by the local construction of Proposition 2.8 near
$w(\infty,t,x_0)$.   This allows the global $\cC$ manifold of
Corollary \ref{largeC} to be constructed by patching together local
$\cC$ manifolds.

\textbf{4. }The profiles $W(z,t,x_0,r)$ in \eqref{n6} are given by
\begin{align}\label{n100}
W(z,t,x_0,r)=W_{\nu(x_0),g(t,x_0)}(z,r)
\end{align}
for $W_{\nu,g}$ as in (2.33) of GMWZ5.   For each point $(r,0)$ in
\eqref{na10}, there is a unique point $U_0=(u_0,u^3_0)\in
\bC_{initial}(t,x_0)$ that is mapped to it under the flow of
\eqref{n99}. $w(z)=W(z,t,x_0,r)$ is uniquely characterized by the
properties that:\medbreak

a)$(w,w_z)$ satisfies \eqref{n99}

b)$\Up(w,0,w^2_z)_{|z=0}=(g^1(t,x_0),g^2(t,x_0),0)$

c)$(w,w^2_z)\to (r,0)$ as $z\to +\infty$.

d) $(w,w^2_z)(0)=U_0\in\bC_{initial}(t,x_0)$. \medbreak

 This characterization shows that the locally constructed profiles
$$
W(z,t,x_0,r)
$$
in \eqref{n100} patch together consistently.  The
local construction of Prop. 2.8 then shows that they are $C^\infty$
in all arguments.

\end{proof}

\section{The Evans function and asymptotic limits}\label{limits}

 We now study the Evans function and its high-frequency and
small-amplitude limits. We first recall the {\it conjugation lemma}
of \cite{MZ1}, which implies that a first-order system
$U'=\cG(z,\zeta,p)U$ whose coefficient matrix converges
exponentially to its limit $\cG(\infty,\zeta,p)$ as $z\to +\infty$,
may be converted by a smooth, exponentially trivial local change of
coordinates
\begin{equation}\label{b37z}
U=P(z,\zeta,p)V=(I+Q(z,\zeta,p))V
\end{equation}
to its limiting constant-coefficient equation $
V'=\cG(\infty,\zeta,p)V$.
Here, we have adjoined to the arguments $\zeta,z$ also dependence on
model parameters $p$, assumed to be at least continuous.

Let $\cG(z,\zeta,p)$ be as in \eqref{linq4}, a frequency-dependent
matrix arising from linearization around a profile $w(z)$ such that
for some positive constants $C$, $\beta$, uniform with
respect to model parameters $p$,
\begin{align}\label{b37y}
|w(z)-w(\infty)|\leq C e^{-\beta z},
\end{align}
and also $p\to (w, \partial_z w_2)(\cdot, p)$
is continuous as a function from $p$ to $L^\infty(0,+\infty)$.
Thus, also,
\begin{align}\label{b37yyy}
|\cG(z, \cdot)-\cG(\infty,\cdot)|\leq C e^{-\beta z}
\end{align}
 and $\cG(\cdot,
\zeta, p)$ is continuous as a function from $p$ to
$L^\infty(0,+\infty)$.

\begin{lem}[\cite{MZ1}, Lemma 2.6]\label{conjugation}
Let $\beta>0$ be as in \eqref{b37y}.
For all $\underline{\zeta}\in\bR^{d+1}$
and model parameters $p$,
there are a neighborhood $\omega \times P$ of $(\underline{\zeta},p)$,
a matrix
$P(z,\zeta,p)=I+Q(z,\zeta,p)$ that is $C^\infty$ on $[0,\infty)\times
\omega$ with derivatives uniformly continuous in $p$,
and positive constants $C$, $\alpha$ with $0<\alpha<\beta$
such that

(i)\;$P$ and $P^{-1}$ are $C^\infty$ and bounded with bounded
derivatives:
\begin{align}\label{Thetabds}
 |\D_z^j\D_\zeta^k Q|\le C_{jk}e^{-\alpha z},
\end{align}

(ii)\;$P(z,\zeta,p)$ satisfies
\begin{align}\label{b37g}
\partial_z P=\cG(z,\zeta,p)P-P\cG(\infty,\zeta,p)\text{ on }z\geq 0.
\end{align}
\end{lem}

Observe that $U$ satisfies \eqref{linq4} on $ z\geq 0$ if and only
if $V$ defined by $U=P(z,\zeta,p)V$ satisfies
\begin{align}\label{b37w}
\partial_zV=\cG(\infty,\zeta,p)V+P^{-1}F,\;\;\Gamma P(0,\zeta,p)V|_{z=0}=G.
\end{align}
This  implies that the decaying space $\EE^-(\zeta,p)$ as in
\eqref{evans1} is exactly the image under $P(0,\zeta,p)$ of the
stable subspace of $\cG(\infty,\zeta,p)$, denoted
$\bE^-_\infty(\zeta,p)$. Thus, by the calculation of  \cite{GMWZ6},
Lemma 2.12, $\bE^-(\zeta,p)$ has dimension $N_b=\rank\; \Gamma$ for
$\gamma\ge 0$, $\zeta\ne 0$. The Evans determinant \eqref{evans1}
\begin{align}\label{pp6}
D_p(\zeta)=\det(\bE^-(\zeta,p),\ker\Gamma(\zeta,p)),
\end{align}
now denoted with additional dependence on model parameters $p$, is
then well-defined on $\gamma\ge 0$, $\zeta\ne 0$ and depends
smoothly on $\zeta$ and continuously (in all $\zeta$-derivatives) on
$p$. We record this as a corollary, of which we shall later make
important use. For quantitative bounds estimating the modulus of
continuity in $p$, see \cite{PZ} Prop. 2.4
or Cor. C.3, \cite{HLZ}.

\begin{rem}The conjugator $P(z,\zeta,p)$ is constructed by a
fixed point argument as the solution of an integral equation. The
exponential decay \eqref{b37yyy} is needed to make the integral
equation contractive in $L^\infty[M,+\infty)$ for $M$  sufficiently
large.  The continuity of $P$ with respect to $p$ is then immediate,
by continuous dependence on parameters of fixed point solutions, a
quite general result.

\end{rem}

\begin{cor}\label{contcor}
Let $w(z,p)$ be a family of layer profiles depending continuously on
parameters $p$ in the sense that $p\to (w(\cdot,p),w^2_z(\cdot,p))$
is a continuous function from $p$ to $L^\infty[0,+\infty)$, and let
$\Gamma(\zeta,p)$ be as in \eqref{pp6}.
Then, the Evans function \eqref{pp6} depends continuously on $p$.

\end{cor}

\begin{proof}
\textbf{1. }Set $W(z,p):=(w(z,p),w^2_z(z,p))$. Continuity in $p$, by
boundedness of $A_j$, $B_{jk}$, and derivatives, is inherited by the
coefficient matrices $\cG(\cdot,\zeta,p)$ appearing in the
linearized eigenvalue equations from continuity of $W$. Likewise,
continuity of the linearized boundary operator $\Gamma(\zeta,p)$
follows from boundedness of $\Upsilon$ and derivatives.   In view of
our rank conditions on $\Gamma(\zeta,p)$ and the continuity of
$P(0,\zeta,p)$ for all $\zeta\in\overline{\bR}^{d+1}$, we see from
the definition of the Evans function \eqref{pp6} that it is
sufficient to establish continuity of $\bE^-_\infty(\zeta,p)$ with
respect to $p$ for $\zeta\neq 0$ and continuity of
$\bE^-_\infty(\hat\zeta,\rho,p)$ with respect to $p$ at $\rho=0$
(recall \eqref{ab2},\eqref{ab3}).

\textbf{2. }For $\zeta \ne 0$, the continuity of
$\bE^-_\infty(\zeta,p)$ follows by the fact that the limiting
coefficient matrix $\cG_\infty(\zeta,p)$ has a spectral gap, whence
the stable subspace  varies continuously by standard matrix
perturbation theory \cite{Kat}.  Continuity
$\bE^-_\infty(\hat\zeta,\rho,p)$ at $\rho=0$ is more difficult to
show and follows from the existence of $K$-families of viscous
symmetrizers as in \cite{MZ3}, Theorem 3.3.

\end{proof}

\subsection{Maximal stability estimates and high-frequency
scaling}\label{maxscale}

We next recall from \cite{GMWZ6} the appropriate scaling of the
Evans function for high-frequencies $|\zeta|\ge R>0$. The maximal
stability estimate for \eqref{linq4} on this frequency domain is
(see \cite{GMWZ6})
 \begin{equation}
\label{maxesthf}
\begin{aligned}
(1+ \gamma)    \| u^1  \|_{L^2(\RR_+)}     + &\Lambda   \| u^2
\|_{L^2(\RR_+)} +  \| \D_z u^2  \|_{L^2(\RR_+)}
\\
+ (1+ \gamma)^{\frac{1}{2}}     |u^1(0)  |   + &\Lambda^\mez   |
u^2(0)| + \Lambda^{-\mez} | \D_z u^2 (0)|  \le
\\
  & C  \big(     \| f ^1  \|_{L^2(\RR_+)}    +  \Lambda^{- 1}  \| f ^2  \|_{L^2(\RR_+)}  \big)
  \\
   & +   C  \big(    (1+ \gamma)^{\frac{1}{2}}     |g^1  |   +
 \Lambda^\mez   | g^2 | + \Lambda^{-\mez} | g^3|  \big),
      \end{aligned}
\end{equation}
 where $C$ is an independent constant and $\Lambda$ is the natural parabolic weight
 \begin{equation}
\label{defLambda} \Lambda(\zeta)  = \big(  \tau^2 + \gamma^2 + |
\eta |^4 \big) ^{1/4}.
\end{equation}
Together with corresponding low-frequency estimates (see
\eqref{maxest2} below), this yields maximal spatio-temporal
stability estimates by Parseval's identity \cite{GMWZ4,GMWZ6}.

 Taking $f = 0$ in \eqref{maxesthf} yields the necessary condition
\begin{equation}
\label{stabcondhf}
\begin{aligned}
\\
  (1& + \gamma)^{\frac{1}{2}}     |u^1   |   +
 \Lambda^\mez   | u^2 )| + \Lambda^{-\mez} |u^3 |  \le
 \\
& C \big(  (1+ \gamma)^{\frac{1}{2}}     |\Gamma_1 u^1   |   +
 \Lambda^\mez   | \Gamma_2 u^2 )| + \Lambda^{-\mez} |\Gamma_3 (\zeta) (u^2, u^3) | \big)
\end{aligned}
\end{equation}
$ \forall \zeta  \in \overline \RR^{d+1}_+$, $ | \zeta | \ge R$,
$\forall U = (u^1, u^2, u^3)  \in \EE_- (\zeta)$. This can be
reformulated in terms of a  \emph{rescaled Evans function}  (see
\cite{MZ1}). Introduce  maps defined on $\CC^{N+ N'}$ and $\CC^{N_b}
$ respectively by
\begin{equation}
\label{reschf}
\begin{aligned}
J_\zeta (u^1, u^2, u^3) &  := \big((1+ \gamma)^{\frac{1}{2}}   u^1,
 \Lambda^\mez  u^2 ,   \Lambda^{-\mez} u^3 \big)  \\
 J_\zeta (g^1, g^2, g^3) & := \big((1+ \gamma)^{\frac{1}{2}}   g^1,
 \Lambda^\mez  g^2 ,   \Lambda^{-\mez} g^3 \big)  .
 \end{aligned}
\end{equation}
Note that $ J_\zeta  \Gamma (\zeta)  U =  \Gamma^{sc}(\zeta) J_\zeta
U $ with
 \begin{equation}
\label{splitGamma}
 \Gamma^{sc}  U =  \big(   \Gamma_1  u^1  ,  \Gamma_2 u^2
         ,    K_d u^3  + \Lambda^{-1} K_{T} (\eta)  u^2  \big)
\end{equation}
(note: decoupled, bounded). Thus \eqref{stabcondhf} reads
\begin{equation}
\label{stabconhfsc} \forall U \in J_\zeta \EE^- (\zeta) :\quad | U |
\le C     | J_\zeta \Gamma (\zeta) J_\zeta^{-1} U |=C|\Gamma^{sc}U|.
\end{equation}
Introducing  the \emph{rescaled Evans function}
\begin{equation}
\label{rescEvf} D^{sc}( \zeta) := \big|  \det   \big( J_{\zeta}\EE^-
(\zeta) ,  J_\zeta  \ker \Gamma(\zeta)\big) \big|=\big|\det   \big(
J_{\zeta}\EE^- (\zeta) ,  \ker \Gamma^{sc}(\zeta)\big) \big|,
\end{equation}
and using Lemma \ref{ZZZ}, we see that this stability  condition is
equivalent to the following definition.

\begin{defi}\label{hfevans}
(a)  Given a profile $w$, the linearized equation $\eqref{linq}$
satisfies  the \emph{uniform Evans condition for high frequencies}
when there are $c> 0$ and $R > 0$ such that $|D^{sc} (\zeta)| \ge c
$ for all $\zeta \in \overline \RR^{d+1}_+  $ with $ | \zeta | \ge
R$.

(b)  The linearized equation $\eqref{linq}$ satisfies  the
\emph{uniform Evans condition} when there are $c> 0$ and $R > 0$
such that
\begin{align}\label{pp7}
|D(\zeta)|\geq c\text{ for }|\zeta|\leq R\text{ and } |D^{sc}
(\zeta)| \ge c \text{ for  }| \zeta | \ge R.
\end{align}

\end{defi}

For completeness, we recall also the maximal stability estimates for
low- and medium-frequencies $|\zeta|\le R$, of
  \begin{equation}
\label{maxest2}
\begin{aligned}
\varphi    \| u  \|_{L^2(\RR_+)}     +
 \| \D_z u^2    \|_{L^2(\RR_+)}      +  | u  (0)& |  + |  \D_z u^2 (0) |
 \le
\\  &  C  \big(   \frac{1 }{\varphi }  \| f  \|_{L^2(\RR_+)}    + | g | \big),
      \end{aligned}
\end{equation}
where $\vp = (\gamma + | \zeta |^2) ^{\frac{1}{2}}$, for $\zeta \in
\overline \RR^{d+1}_+ \backslash\{0\}$, $| \zeta | \le R$. Taking
$f=0$ yields the necessary condition $ | u  (0) |  + |  \D_z u^2 (0)
|    \le        C | g | $
corresponding to the standard Evans condition \eqref{evans1}.

\begin{rem}\label{homgen}
By Theorem 3.9 of \cite{GMWZ6} the uniform Evans condition for low
and medium frequencies implies the maximal estimate \eqref{maxest2}.
Taking $f=0$ in \eqref{maxest2} and using Remark \ref{Evansest}, we
see that the following are equivalent for $|\zeta|\leq R$:

a) the uniform Evans condition for $|\zeta|\leq R$,

b) the full bounded frequency estimate \eqref{maxest2},

c) the trace estimate in \eqref{maxest2} when $f=0$.

Similarly, by Theorem 7.2 of \cite{GMWZ6} the uniform Evans
condition for high frequencies implies the maximal estimate
\eqref{maxesthf}.   Using Remark \ref{Evansest} again, we deduce the
equivalence of:

a)the uniform Evans condition for $|\zeta|\geq R$,

b) the full high frequency estimate \eqref{maxesthf},

c) the trace estimate in \eqref{maxesthf} when $f=0$.

\end{rem}

\subsection{The high-frequency limit}\label{HF}

\text{\quad}In this section we show that the rather complicated
high-frequency condition of Definition \ref{hfevans} may be reduced
to a simple and natural linear algebraic condition corresponding
roughly to well-posedness of the principal part \eqref{princpart} of
equations \eqref{visceq}, frozen at $z=0$, under boundary conditions
\eqref{viscbcd}.
Using the fact that the linearized boundary conditions are {fully
decoupled}, we prove in Theorem \ref{aa3} that satisfaction of the
uniform Evans condition for sufficiently high frequencies is
equivalent to satisfaction of the uniform Evans condition
for the {\it decoupled system}
\begin{align}\label{linprinc}
\begin{split}
&(a)\;u^1_t + \sum_j \overline A_j^{11}(w(0)) \D_j u^1=0,\\
&(b)\;u^2_t - \sum_{j,k} \overline B_{jk}^{22}(w(0)) \D_j\D_k u^2=0,
\end{split}
\end{align}
with boundary conditions $\Gamma_1$ and $(\Gamma_2, \Gamma_3)$,
respectively.

Let $\bE^-(\zeta)$ denote as before the set of initial data at $z=0$
of decaying solutions of $\partial_z U-\cG(z,\zeta)U=0$.  Our proof
of Theorem \ref{aa3} is based on showing that in the limit as
$|\zeta|\to\infty$, the space $J_\zeta\bE^-(\zeta)$ approaches, or
``tracks", the rescaled stable subspace of an appropriate frozen
coefficient problem.

Since subsystems \eqref{linprinc}(a)--(b) are quasi-homogeneous, the
uniform Evans condition is equivalent simply to nonvanishing of the
decoupled Evans functions on the unit sphere in $\gamma \ge 0$, a
linear algebraic computation.  This reduction to a compact set of
frequencies (Corollary \ref{hfcrit}) is essential in our later
verifications of high frequency stability.

To state these results precisely we write the first order systems
obtained from \eqref{linprinc}(a), (b) by Fourier-Laplace transform
as
\begin{align}\label{aa1}
\begin{split}
&(a)\partial_z u^1-\cG_1(\zeta)u^1=0,\; \Gamma_1 u^1=g^1\\
&(b) \partial_z U^*-\cG_2(\zeta)U^*=0,\; \Gamma_* U^*=(g^2,0),
\end{split}
\end{align}
where $U^*=(u^2,u^2_z)$ and $\Gamma_*=(\Gamma_2,\Gamma_3)$. Let
$e_{-,h}(\zeta)$ and $e_{-,p}(\zeta)$ denote the stable subspaces of
the matrices $\cG_1(\zeta)$ and $\cG_2(\zeta)$ respectively. Setting
\begin{align}\label{A1}
\begin{split}
&J^1_\zeta(u^1)=(1+\gamma)^{\frac{1}{2}}u^1,\;\;\Gamma_1^{sc}u^1=\Gamma_1u^1\\
&J_\zeta^*(u^2,u^3):=\left(\Lambda^{\frac{1}{2}}u^2,\Lambda^{-\frac{1}{2}}u^3\right)\\
& \Gamma_*^{sc}(\zeta)(u^2,u^3)=\left(\Gamma_2 u^2
         ,    K_d u^3  + \Lambda^{-1} K_{T} (\eta)
         u^2\right),
\end{split}
\end{align}
we define rescaled Evans functions for \eqref{aa1}(a) and (b) by
\begin{align}\label{aa2}
\begin{split}
&D^{sc}_1(\zeta)=\det_{\bC^{N-N'}}\left(J^1_\zeta
e_{-,h}(\zeta),\ker\Gamma_1\right)\;\; (=\det\left(
e_{-,h}(\zeta),\ker\Gamma_1\right) \text{ clearly}),\\
&D^{sc}_2(\zeta):=\det_{\CC^{2N'}} \Big( J_\zeta^*
e_{-,p}(\zeta),\ker \Gamma_*^{sc}(\zeta)\Big).
\end{split}
\end{align}

The main result of this section is the following theorem.

\begin{theo}\label{aa3}
Let $D^{sc}$ be the rescaled Evans function defined in
\eqref{rescEvf}.   Then \emph{frozen-coefficients stability} for
\eqref{aa1}, that is, the existence of  positive constants
$c_1$,$c_2$, $R'$ such that
\begin{align}\label{aaa3}
|D^{sc}_1(\zeta)|\geq c_1 \text{ and }|D^{sc}_2(\zeta)|\geq
c_2\text{ for all }|\zeta|\geq R',
\end{align}
is equivalent to the rescaled uniform Evans condition, that is, the
existence of positive constants $c$, $R$ such that
\begin{align}\label{aaa4}
|D^{sc}(\zeta)|\geq c \text{ for all }|\zeta|\geq R.
\end{align}

\end{theo}
\begin{proof}

\textbf{1. Frequency zones.} There are two frequency zones to
consider. For $\delta>0$ sufficiently small we define the
\emph{elliptic zone}
\begin{align}\label{ee1}
\cE_\delta:=\{\zeta=(\tau,\gamma,\eta):\gamma\geq\delta|\zeta|,
|\eta|\geq \delta |\zeta|\}
\end{align}
and the complementary \emph{coupling zone}
\begin{align}\label{ee3}
\cC_\delta:=\{\zeta:0\leq \gamma\leq
\delta|\zeta|\}\cup\{\zeta:|\eta|\leq \delta|\zeta|\}.
\end{align}
Here we show how the discussion of \cite{GMWZ6}, section 7 can be
adapted to prove Theorem \ref{aa3} in the more difficult case where
$\zeta\in \cC_\delta$.    The elliptic case is proved in Appendix B,
which also includes a new discussion (see Proposition \ref{cc10}) of
the connections between symmetrizer estimates and tracking.

\textbf{2.  Symbols. }Let $\Gamma^m$ denote the space of
\emph{homogeneous symbols} of order $m$, that is, $C^\infty$
functions $h(z,\zeta)$ such that for all $\alpha\in\bN^{d+1}$, all
$k\in\bN$, and some $\theta>0$,  there are constants $C_{\alpha,k}$
such that for $|\zeta|\geq 1$,
\begin{align}\label{ee3a}
\begin{split}
&|\partial^\alpha_\zeta h|\leq C_{\alpha,0}|\zeta|^{m-|\alpha|},
\text{ if }k=0,\\
&|\partial^k_z\partial^\alpha_\zeta h|\leq C_{\alpha,k}e^{-\theta z}
|\zeta|^{m-|\alpha|}, \text{ if }k>0.
\end{split}
\end{align}

Let $P\Gamma^m$ denote the space of \emph{parabolic symbols} of
order $m$, that is, $C^\infty$ functions $h(z,\zeta)$ satisfying
similar estimates with $|\zeta|$ replaced by $\Lambda$.

\textbf{3. Conjugation to block diagonal form in $\cC_\delta$.}
Consider again the linearized problem \eqref{linq4}
\begin{align}\label{ee4}
\partial_z U=\cG(z,\zeta)U+F,\;\;\Gamma(\zeta)U(0)=G.
\end{align}
where the components of $\cG$ are given explicitly in \eqref{cG}.
Set
\begin{align}\label{ee4e}
\cP^{22}=\begin{pmatrix}0&|\zeta|I\\|\zeta|^{-1}\cG^{32}&\cG^{33}\end{pmatrix}.
\end{align}
In Lemma 7.3 of \cite{GMWZ6} it is shown that there exist positive
constants $c$, $R$, and $\delta$ such the distance between the
spectrum of $\cG^{11}(z,\zeta)$ and the spectrum of
$\cP^{22}(z,\zeta)$ is larger than $c|\zeta|$ for $\zeta\in
\cC_\delta$, $|\zeta|\geq R$.   For such $\zeta$ Lemmas 7.5 and 7.6
of \cite{GMWZ6} use this spectral separation to construct a change
of variables $U=\cV\hat U$,\;$F=\cV\hat F$ that transforms
\eqref{ee4} to
\begin{align}\label{ee5}
\partial_z \hat U=\hat\cG(z,\zeta)\hat U+\hat F,\;\;\hat \Gamma(\zeta)\hat
U(0)=G.
\end{align}
Here $\hat \cG=\hat\cG_p+\cG'$ and with obvious notation
\begin{align}\label{ee6}
\hat\cG_p=\begin{pmatrix}\hat\cG^{11}&0&0\\0&0&I\\0&\cG^{32}-\cV^{31}\cG^{12}&\cG^{33}\end{pmatrix},\;\;\cG'=\begin{pmatrix}\Gamma^{-1}&\Gamma^0&\Gamma^{-1}\\\Gamma^{-1}&\Gamma^0&\Gamma^{-1}\\
\Gamma^{0}&\Gamma^0&\Gamma^{0}\end{pmatrix}.
\end{align}
We have $\cV=\cV_I(z,\zeta)\cV_{II}(z,\zeta)$, where
\begin{align}\label{ee7}
\cV_I=\begin{pmatrix}I&0&0\\|\zeta|^{-1}\cV^{21}&I&0\\\cV^{31}&0&I\end{pmatrix},\;\;\begin{pmatrix}I&\cV^{12}&|\zeta|^{-1}\cV^{13}\\0&I&0\\0&0&I\end{pmatrix},\;\;\cV^{ij}\in\Gamma^0.
\end{align}

\textbf{4. Rescaling.} It is helpful to work with an equivalent
rescaled system.  Let $J:=J_\zeta$ be the scaling operator
introduced in \eqref{reschf} (There are really two such operators,
but we use $J$ to denote both.)  For any matrix $M$ of size
$(N+N')\times (N+N')$ or $N_b\times (N+N')$, and for any
$U\in\bC^{N+N'}$ or $U\in\bC^{N_b}$ define
\begin{align}\label{ee8}
M_s=JMJ^{-1}, U_s=JU.
\end{align}
In expressions like $\hat U_s$ or $\hat\cG_s$ where the order of the
``hat" and scaling operators may be unclear, the scaling operator
always comes last.

With $V=\hat U_s$ the system \eqref{ee5} is then equivalent to
\begin{align}\label{ee9}
\partial_z V=\hat\cG_s(z,\zeta)V+\cV^{-1}_s F_s,\;\;\hat
\Gamma_s(\zeta)V(0)=G_s.
\end{align}
Observe that $\hat\Gamma_s=\Gamma_s\cV_s$ and that $\Gamma_s$ is the
same as the operator $\Gamma^{sc}$ defined in \eqref{splitGamma}. We
may write $\hat\cG_s=\hat\cG_{ps}+\cG'_s$ where
\begin{align}\label{ee10}
\begin{split}
&\hat\cG_{ps}=\begin{pmatrix}\hat\cG^{11}&0&0\\0&0&\Lambda\\0&\Lambda^{-1}(\cG^{32}-\cV^{31}\cG^{12})&\cG^{33}\end{pmatrix},\\
&\cG'_s=\begin{pmatrix}\Gamma^{-1}&(1+\gamma)^{\frac{1}{2}}\Lambda^{-\frac{1}{2}}\Gamma^0&(1+\gamma)^{\frac{1}{2}}\Lambda^{\frac{1}{2}}\Gamma^{-1}\\\Lambda^{\frac{1}{2}}(1+\gamma)^{-\frac{1}{2}}\Gamma^{-1}&\Gamma^0&\Lambda\Gamma^{-1}\\
\Lambda^{-\frac{1}{2}}(1+\gamma)^{-\frac{1}{2}}\Gamma^{0}&\Lambda^{-1}\Gamma^0&\Gamma^{0}\end{pmatrix}.
\end{split}
\end{align}

It will be important to know the exact form of $\cV_s$.  Direct
computation of $J \cV J^{-1}$ gives $\cV_s=$
\begin{align}\label{ee11}
\begin{pmatrix}1&(1+\gamma)^{\frac{1}{2}}\Lambda^{-\frac{1}{2}}\cV^{12}&0\\0&I&0\\
\Lambda^{-\frac{1}{2}}(1+\gamma)^{-\frac{1}{2}}\cV^{31}&\Lambda^{-1}\cV^{31}\cV^{12}&I\end{pmatrix}.
\end{align}
The only entry of this matrix that is not obviously bounded as
$|\zeta|\to\infty$ is the $(12)$-entry.  By equation (7.36) of
\cite{GMWZ6} we have
\begin{align}\label{ee12}
\cV^{12}(z,\zeta)=O(|\eta|/|\zeta|),
\end{align}
so boundedness of the $(12)$-entry follows from
\begin{align}\label{ee13}
(1+\gamma)^{\frac{1}{2}}|\eta|/|\zeta|=\frac{(1+\gamma)^{\frac{1}{2}}}{|\zeta|^{\frac{1}{2}}}\frac{|\eta|}{|\zeta|^{\frac{1}{2}}}\leq
|\eta|^{\frac{1}{2}}\leq\Lambda^{\frac{1}{2}}.
\end{align}
A similar computation of $\cV_s^{-1}$ shows
\begin{align}\label{ee14}
|\cV_s|\leq C, \;\; |\cV^{-1}_s|\leq C \text{ uniformly for }  z\geq
0,\; \zeta\in\cC_\delta,
\end{align}
and furthermore
\begin{align}\label{ee15}
\cV_s\approx\begin{pmatrix}1&O(1)&O(1)\\0&1&0\\0&0&1\end{pmatrix},\;\;\cV^{-1}_s\approx\begin{pmatrix}1&O(1)&O(1)\\0&1&0\\0&0&1\end{pmatrix}\text{
for }|\zeta| \text{ large}.
\end{align}

\textbf{5. Incoming-outgoing estimates. }Consider now the principal
part of \eqref{ee9}:
\begin{align}\label{ee27}
\partial_z V=\hat\cG_{ps}(z,\zeta)V+\cV^{-1}_s F_s,\;\;
\Gamma_s\cV_sV(0)=G_s.
\end{align}
By Lemma 7.13 of \cite{GMWZ6} there is a smooth change of variables
$\cW\in P\Gamma^0$ such that, if we set
\begin{align}\label{ee24}
T(z,\zeta):=\begin{pmatrix}I&0\\0&\cW\end{pmatrix},
\end{align}
then
\begin{align}\label{ee25}
T^{-1}\hat\cG_{ps}T=\begin{pmatrix}\hat\cG^{11}&0&0\\0&P_+&0\\0&0&P_-\end{pmatrix}:=\cG_{B,p}(z,\zeta),
\end{align}
with $P_\pm$ having their eigenvalues satisfying $\pm\Re \mu\geq
c\Lambda$.  Defining $W$ by
\begin{align}\label{ee26}
V=TW,\;  W=(w^1,w^+,w^-),
\end{align}
we can write \eqref{ee27} equivalently as
\begin{align}\label{ee28}
\partial_z W=\cG_B(z,\zeta) W+T^{-1}\cV^{-1}_sF_s,\;\;\Gamma_s\cV_s T
W=G_s,
\end{align}
where now $\cG_B=\cG_{B,p}-T^{-1}T_z$.

Define the outgoing ($W^+$) and incoming ($W^-$) parts of $W$ by
\begin{align}\label{ee29a}
W^+=\begin{pmatrix}w^1_+\\w^+\\0\end{pmatrix},\;\;W^-=\begin{pmatrix}w^1_-\\0\\w^-\end{pmatrix},
\end{align}
where $w^{1+}=w^1$, $w^{1-}=0$ when $N^1_+=0$, $w^{1+}=0$,
$w^{1-}=w^1$ when $N^1_+=N-N'$.  With $\|U\|:=\|U\|_{L^2(\bR_+)}$ we
define norms
\begin{align}\label{ee29}
\begin{split}
&(a) \|W\|_s=(1+\gamma)^{\frac{1}{2}}\|w^1\|+\Lambda^{\frac{1}{2}}\|w^+\|+\Lambda^{\frac{1}{2}}\|w^-\|\\
&(b) \|F\|'_s=(1+\gamma)^{-\frac{1}{2}}\|F^1\|+\Lambda^{-\frac{1}{2}}\|F^2\|+\Lambda^{-\frac{1}{2}}\|F^3\|\\
&(c) |W(0)|_s=|W(0)|.
\end{split}
\end{align}

Proposition 7.11 and Corollary 7.14 of \cite{GMWZ6} imply that for
large enough $R$ and $\zeta\in\cC_\delta$ with $|\zeta|\geq R$, we
have the following estimates:
\begin{align}\label{ee30e}
\begin{split}
&\|W^+\|_s+|W^+(0)|\leq C\|(\partial_z-\cG_{B,p})W^+\|'_s\\
&\|W^-\|_s\leq C\|(\partial_z-\cG_{B,p})W^-\|'_s+|W^-(0)|.
\end{split}
\end{align}

Suppose now that $V$ is a solution of \eqref{ee9}.   Let $T$ be
exactly as above and define $W$ by $V=TW$, $W=(w^1,w^+,w^-)$. Now
$W$ satisfies
\begin{align}\label{ee30h}
\partial_z W=\cG_B(z,\zeta) W+T^{-1}\cG'_s TW+T^{-1}\cV^{-1}_sF_s,\;\;\Gamma_s\cV_s T
W=G_s,
\end{align}
>From \eqref{ee30e} we deduce the following estimates  for $W$ by
treating the extra terms $T^{-1}\cG_s'TW$ and $T^{-1}T_zW$ as
forcing terms:
\begin{align}\label{ee30g}
\begin{split}
&\|W^+\|_s+|W^+(0)|\leq C\|(\partial_z-\cG_{B,p})W^+\|'_s+\eps(\zeta)\|W\|_s\\
&\|W^-\|_s\leq
C\|(\partial_z-\cG_{B,p})W^-\|'_s+|W^-(0)|+\eps(\zeta)\|W\|_s.
\end{split}
\end{align}
Here we have used the explicit forms of $\cG'_s$ \eqref{ee10} and
$T$. For example, the first row of $T^{-1}\cG'_sTW$ contributes
(recall \eqref{ee29}(b))
\begin{align}\label{ee34}
(1+\gamma)^{-\frac{1}{2}}\|\Gamma^{-1}w^1\|+C\|\Lambda^{-\frac{1}{2}}\Gamma^0
(w^+,w^-)\|+C\|\Lambda^{\frac{1}{2}}\Gamma^{-1} (w^+,w^-)\|,
\end{align}
which is a term of the form $\eps(\zeta)\|W\|_s$.

Defining the outgoing and incoming parts of $V$ by
\begin{align}\label{ee31}
V^+=TW^+,\;\;V^-=TW^-,
\end{align}
we deduce from \eqref{ee30g} and \eqref{ee14} the following
estimates for solutions to \eqref{ee9}:
\begin{align}\label{ee30}
\begin{split}
&(a)\;\|V^+\|_s+|V^+(0)|\leq C\|F_s\|'_s+\eps(\zeta)\|V\|_s\\
&(b)\;\|V^-\|_s\leq C\|F_s\|'_s+C|V^-(0)|+\eps(\zeta)\|V\|_s,
\end{split}
\end{align}
where $\eps(\zeta)\to 0$ as $|\zeta|\to \infty$.

\textbf{6. The stable subspace of $\hat\cG_{ps}(0,\zeta)$.}
Let $E^f_-(\zeta)$ be the stable subspace of the frozen operator
$\hat\cG_{ps}(0,\zeta)$.
>From \eqref{ee25} we see that when $N^1_+=0$ (resp. $N^1_+=N-N'$)
\begin{align}\label{ee30a}
\begin{split}
&E^f_-(\zeta)=\{T(0,\zeta)\begin{pmatrix}0\\0\\w^-\end{pmatrix}: w^-
\in \bC^{N'}\},\\
&(\text{resp., },
E^f_-(\zeta)=\{T(0,\zeta)\begin{pmatrix}w^{1-}\\0\\w^-\end{pmatrix}:w^{1-}\in\bC^{N-N'},w^-\in\bC^{N'}\}),
\end{split}
\end{align}
>From this and \eqref{ee29a} we conclude that for $V^-$ as in
\eqref{ee31}
\begin{align}\label{gg3}
V^-(0)\in E^f_-(\zeta).
\end{align}

Proposition 7.10 of \cite{GMWZ6} shows that, just like
$\cG_1(\zeta)$ in \eqref{aa1}(a), the matrix $\hat\cG^{11}$ has
eigenvalues with only positive, respectively negative, real part if
$\oA_d^{11}$ is outgoing ($N^1_+=0$), respectively incoming
($N^1_+=N-N'$).  Moreover, for $J_*:=J_\zeta^*$ as in \eqref{A1} and
$\cG_2$ as \eqref{aa1}(b), the $(32)$-entry of $\hat\cG_{ps}$
differs from that of $J_*\cG_2(\zeta)J_*^{-1}$ by a term that is
$O(1)$.   These observations imply that for $|\zeta|$ large
\begin{align}\label{ee18}
E^f_-(\zeta)\text{ is close to }\begin{pmatrix}J^1_\zeta
e_{-,h}(\zeta)\\0\end{pmatrix}\oplus\begin{pmatrix}0\\J^*_\zeta
e_{-,p}(\zeta)\end{pmatrix}   \;\;\;\text{  (recall }\eqref{aa2}).
\end{align}

\textbf{7. Incoming-outgoing estimates imply tracking.} Suppose that
$U(z)$ is a decaying solution of
\begin{align}\label{gg1}
\partial_z-\cG(z,\zeta)U=0.
\end{align}
Then $U(0)\in\bE^-(\zeta)$, or equivalently,
\begin{align}\label{gg2}
V(0)=J\cV^{-1}U(0)\in J\cV^{-1}\bE^-(0).
\end{align}
>From \eqref{ee30} we easily obtain
\begin{align}\label{gg4}
\begin{split}
&\|V^-\|_s\leq C|V^-(0)|+\eps(\zeta)\|V^+\|_s\\
&\|V^+\|_s+|V^+(0)|\leq C|V^-(0)|+\eps(\zeta)\|V^+\|_s,
\end{split}
\end{align}
and thus
\begin{align}\label{gg5}
|V^+(0)|\leq \eps(\zeta)|V^-(0)|.
\end{align}

Using \eqref{gg5}, \eqref{gg2}, and \eqref{ee30a} we obtain
\begin{align}\label{gg6}
J\cV^{-1}\bE^-(\zeta)\approx E^f_-(\zeta) \text{ for }
\zeta\in\cC_\delta,\;|\zeta|\geq R.
\end{align}
Applying $J \cV J^{-1}$ to \eqref{gg6} we find
\begin{align}\label{gg7}
J\bE^-(\zeta)\approx J\cV J^{-1} E^f_-(\zeta)=\cV_sE^f_-(\zeta).
\end{align}
Since $\cV_s$ has the upper triangular form \eqref{ee15} as
$|\zeta|\to\infty$, we conclude from \eqref{gg7} and \eqref{ee18}
that
\begin{align}\label{gg8}
J\bE^-(\zeta)\approx  \begin{pmatrix}J^1_\zeta
e_{-,h}(\zeta)\\0\end{pmatrix}\oplus\begin{pmatrix}0\\J^*_\zeta
e_{-,p}(\zeta)\end{pmatrix}   \;\;\text{ for }|\zeta|\geq R,
 \zeta\in\cC_\delta.
\end{align}

\textbf{8. Conclusion. }Since
\begin{align}\label{gg9}
\ker\Gamma_s=\begin{pmatrix}\ker\Gamma^1\\0\end{pmatrix}\oplus\begin{pmatrix}0\\\ker
\Gamma^{sc}_*\end{pmatrix},
\end{align}
\eqref{gg8} implies the equivalence in Theorem \ref{aa3} for
$|\zeta|\geq R$, $\zeta\in\cC_\delta$.  Together with the proof for
the elliptic zone $\cE_\delta$ in Appendix B, this concludes the
proof of Theorem \ref{aa3}.

\end{proof}

As a corollary we obtain the following simple criterion for
high-frequencies,

\begin{cor}\label{hfcrit}
Under (H1)--(H6), satisfaction of the high-frequency uniform Evans
condition for $|\zeta|\ge R$, $R>0$ sufficiently large, is
equivalent to nonvanishing on the positive parabolic sphere
$\gamma\ge 0$, $|\lambda |+ |\eta|^2=1$ (equivalently, the positive
unit sphere $\gamma \ge 0$, $|\zeta|=1$) of the parabolic Evans
function
\begin{align}\label{A1a}
d^2(\zeta):=\det_{\CC^{2N'}} \Big(  e_{-,p}(\zeta),\ker
\Gamma_*^{sc}(\zeta)\big).
\end{align}
\end{cor}

\begin{proof}

The system \eqref{aa1}(b) has the form
\begin{align}\label{A2}
\partial_zU^*=\mathcal{G}_2(\zeta)U^*:=\begin{pmatrix}0&1\\\mathcal{M}&\mathcal{A}\end{pmatrix}U^*,
\end{align}
where $\cM(\zeta)$ and $\cA(\zeta)$ are quasihomogeneous of degrees
two and one respectively.
The equation \eqref{A2} can
be written equivalently as
\begin{align}\label{A3}
\begin{split}
&\partial_z
(J_\zeta^*U^*)=\Lambda\hat{\cG}_2(\zeta)J_\zeta^*U^*,\text{
where }\\
&\hat{\cG}_2(\zeta):=\begin{pmatrix}0&1\\\frac{\cM}{\Lambda^2}&\frac{\cA}{\Lambda}\end{pmatrix}.
\end{split}
\end{align}
This shows that $U^*\in e_{-,p}(\zeta)$ if and only if
$J_\zeta^*U^*$ is in the stable subspace of $\hat{\cG}_2(\zeta)$.
With $\hat{\zeta}:=\zeta/\Lambda$ and writing $E_-(M)$ for the
stable subspace of any matrix $M$, we thus have
\begin{align}\label{A6}
J_\zeta^*e_{-,p}(\zeta)=E_-(\hat{\cG}_2(\zeta))=E_-(\hat{\cG}_2(\hat{\zeta}))=e_{-,p}(\hat{\zeta}).
\end{align}
Since we clearly have
\begin{align}\label{A7}
\ker \Gamma^{sc}_*(\zeta)=\ker \Gamma^{sc}_*(\hat{\zeta}),
\end{align}
it follows from \eqref{A6} that $|D^2(\zeta)|\geq C>0$ for $|\zeta|$
large if and only if $d^2(\zeta)$ is nonvanishing on the parabolic
unit sphere.

Since the eigenvalues of $\overline A^{11}_d$ were assumed all
positive (totally incoming) or all negative (totally outgoing), we
have either $\ker \Gamma_1=\{0\}$, $e_{-,h}(\zeta)=\CC^{N-N'}$ or
else $\ker \Gamma_1= \CC^{N-N'}$, $e_{-,h}(\zeta)=\{0\}$. In either
case the hyperbolic stability condition is trivially satisfied.
\end{proof}

The following result verifies high-frequency stability for many
physical cases, including the applications we will consider here.

\begin{prop}\label{symmhf} Consider a layer profile $w(z)$ as in \eqref{vv43}
and the linearized equations about $w(z)$ given by \eqref{linq} (or
\eqref{linq4}. In the symmetric-dissipative case (Defn.
\ref{defsymm}) the uniform high-frequency Evans condition is
satisfied either for full Dirichlet conditions $\rank \Upsilon_3= 0$
or full Neumann conditions $\rank \Upsilon_3= N'$ on the parabolic
variable $u^2$.
\end{prop}

\begin{proof}
\textbf{1. }By Corollary \ref{hfcrit} and Remark \ref{Evansest}, the
uniform high-frequency Evans condition in the case of  full
Dirichlet (resp. full Neumann) boundary conditions on $u^2$ is
equivalent to the estimate
\begin{align}\label{Z1}
|u^2_z(0)|\leq C|g|\;\;   \text{ (resp. } |u^2(0)|\leq C|h|)
\end{align}
for decaying solutions of
\begin{align}\label{Z2}
\begin{split}
&\lambda u^2-B^{22}_{dd}u^2_{zz}-i\sum_{k\neq
d}(B_{dk}^{22}+B_{kd}^{22})\eta_ku^2_z+\sum_{j,k\neq
d}\eta_j\eta_kB_{jk}^{22}=0\\
&u^2(0)=g\;\;\text{ (resp. } u^2_z(0)=h),
\end{split}
\end{align}
where the constant $C$ in \eqref{Z1} is independent of
$(\lambda,\eta)$ in the positive parabolic unit sphere.  As in
\eqref{linprinc} the coefficients in \eqref{Z2} are evaluated at
$w(0)$.

The estimates below are similar, but not identical, to those given
in section \ref{strictparabcase}. Here we highlight the differences
and refer to that section for extra detail.   We now take
\begin{align}\label{pp19}
|\lambda|+|\eta|^2=1.
\end{align}

\textbf{2. Dirichlet conditions. } Taking the real part of the inner
product of $u^2$ with \eqref{Z2}, we obtain after integration by
parts as in \eqref{real2}:
\begin{equation}
\label{real3}
\begin{aligned}
(\Re \lambda +|\eta|^2)|u_2|_2^2 + |u_2'|_2^2 \le C( |g||u_2'(0)|+
|\eta||g|^2).
\end{aligned}
\end{equation}
Here the last term on the right is a ``Garding error" that is
explained just below \eqref{j1}.

Similarly, taking the real part of the inner product of $-u_2''$
with \eqref{Z2}, we obtain as in \eqref{realderiv}
\begin{equation}
\label{realderiv3}
\begin{aligned}
(\Re \lambda +|\eta|^2)|u_2'|_2^2 + |u_2''|_2^2 \le
C\left((|\lambda|+|\eta|^2\right) |g||u_2'(0)|+ |\eta||u_2'(0)|^2).
\end{aligned}
\end{equation}
The small differences between the estimates here and in section
\ref{strictparabcase} reflect the absence of the matrices $A_k$ in
\eqref{Z2}.

Using the Sobolev bound
\begin{equation}\label{Sobb}
|u_2'(0)|^2\le |u_2'|_2|u_2''|_2\le C_\delta|u_2'|_2^2 +
\delta|u''|_2^2
\end{equation}
we immediately deduce
\begin{align}\label{pp22}
|u_2'(0)|\leq C|g|
\end{align}
from \eqref{real3}, \eqref{realderiv3}, and \eqref{pp1}.

\textbf{3. Neumann conditions. } Taking inner products as above, but
now taking imaginary parts in order to estimate
$|\Im\lambda||u_2|_2^2$ and $|\Im\lambda||u'_2|_2^2$, we obtain
after combining estimates:
\begin{align}\label{real4}
 (|\lambda| +|\eta|^2)|u_2|_2^2 + |u_2'|_2^2 \le C( |g||u_2'(0)|+
|\eta||g|^2)
\end{align}
and
\begin{equation}
\label{realderiv4}
\begin{aligned}
(|\lambda| +|\eta|^2)|u_2'|_2^2 + |u_2''|_2^2 \le
C\left((|\lambda|+|\eta|^2\right) |g||u_2'(0)|+ |\eta||u_2'(0)|^2).
\end{aligned}
\end{equation}

Using the Sobolev inequality
\begin{align}\label{a432}
|u_2(0)|^2\le |u_2|_2|u_2'|^2_2\leq
\delta|u_2|^2_2+C_\delta|u_2'|_2^2.
\end{align}
and the estimates on $u_2$ and $u_2'$ coming from the first terms on
the left in \eqref{real4} and \eqref{realderiv4} respectively, we
obtain
\begin{align}\label{pp16}
|u_2(0)|\leq C|h|.
\end{align}

\end{proof}

\begin{rem}\label{genBC}
 Recalling (see Remark \ref{homgen}) that the Evans condition is
equivalent to maximal stability estimates, we find from the
reduction to \eqref{linprinc} that for decoupled boundary conditions
the assumed ranks of $\Gamma_1$ and $(\Gamma_2, \Gamma_3)$ are {\it
necessary} in order to obtain maximal high-frequency stability
estimates. For example, specifying density or pressure rather than
velocity for outgoing isentropic gas-dynamical flow (see Section
\ref{Isentropic}) would result in degraded stability estimates.
\end{rem}

Note that the above results hold for {\it arbitrary-amplitude
layers}.

\subsection{The small-amplitude limit}\label{smallamp}
With Corollary \ref{hfcrit} we are now able to verify the uniform
Evans condition for high frequencies (Definition \ref{hfevans}) by
reducing to the consideration of a bounded set of frequencies.  This
puts us in position to prove Theorem \ref{smallred}.

\begin{proof}[Proof of Theorem \ref {smallred}]
\textbf{1. Preliminaries. }It is sufficient to show that uniform
Evans
stability of the constant layers $w(z,u,\nu)\equiv u$, $(u,\nu)\in
D$ implies uniform Evans stability for sufficiently small amplitude
profiles associated to elements of $D$. (The reverse direction is
trivial, zero-amplitude being included in the set of small-amplitude
profiles.)
By compactness of $D$, it is sufficient to establish stability of
small-amplitude layers in the vicinity of the constant layer
$w(z)=w(z,\uu,\unu)\equiv\uu$ associated to a single element
$(\uu,\unu)\in D$.  Recall that $\eps$-amplitude profiles
$w(z,u,\nu)$ as in Definition \ref{cd1} satisfy
\begin{align}\label{bg1}
\begin{split}
&(a)\; A_\nu(w)\partial_zw-\partial_z(B_\nu(w)\partial_zw)=0\text{
on
}z\geq 0\\
&(b)\;w\to u\text{ as }z\to \infty,\\
&(c)\;\|(w,w^2_z)-(\uu,0)\|_{L^\infty}\leq \eps, \;|\nu-\unu|\leq
\eps
\end{split}
\end{align}
for some $\eps>0$.

\textbf{2. Parameters.} From the assumed Evans stability of $w\equiv
\uu$, we have transversality of the constant layer by Lemma
\ref{Rousset}. Thus, by Proposition \ref{C14a} specialized to the
vicinity of a single point, there exists a neighborhood
$\omega\subset \bR^N\times S^{d-1}$ of
$(\underline{u},\underline{\nu})$ and constants $R>0$, $r>0$ such
that for $(\nu,q)\in\omega$, all solutions $w$ of \eqref{bg1}(a),(b)
satisfying
$$
\|(w, w^2_z)-(\uu, 0)\|_{L^\infty[0,\infty]}\leq R,
$$
are parametrized by a $C^\infty$ function $w=\Phi(z,\nu,q,a)$ on
$[0,\infty)\times \omega^*$, where $\omega^*$ is the set of
parameters $(\nu,q,a)$ with $(\nu,q)\in\omega$ and
$a\in\bE_-(G_\unu(\up))$ with $|a|\leq r$.

Let $D_{\uu,\unu}(\zeta)$ denote the Evans function for \eqref{vv2}
 corresponding to linearization around the constant state $\uu$,
let $D_{\nu,q,a}(\zeta)$ be the Evans function arising from a
profile $w=\Phi(z,\nu,q,a)$.
and note that
\begin{equation}\label{code}
D_{\uu,\unu}=D_{\unu,\uu,0}.
\end{equation}
Using similar notation for the Evans function $d^2(\zeta)$ defined
in \eqref{A1a} we have:
\begin{align}\label{hh4}
\begin{split}
&(a)d^2_k(\zeta)=\det\left(e_{-,p,k}(\zeta),\ker\Gamma_{*,k}^{sc}(\zeta)\right),\;\;k=(\uu,\unu),\;
(\nu,q,a)\\
&(b)D_k(\zeta)=\det\left(\bE^-_k(\zeta),\ker\Gamma_k(\zeta)\right),\;\;k=(\uu,\unu),\;
(\nu,q,a).
\end{split}
\end{align}

 Observe that for $\nu$ near $\unu$,  given $\delta>0$ there exists
$0<\eps<R$ such that
\begin{align}\label{hhh4}
\|(w, w^2_z)-(\uu, 0)\|_{L^\infty[0,\infty]}\leq
\eps,\;\;|\nu-\unu|<\eps \Rightarrow |\nu-\unu|+|q-\uu|+|a|<\delta.
\end{align}
This follows from the fact that by Proposition \ref{C14a},  for
$\nu$ near $\unu$, $(\Phi,\Phi^2_z)$ defines a diffeomorphism from a
neighborhood of $(q,a)=(\uu,0)$ into the center-stable manifold of
$(\uu,0)$ for \eqref{bg1}(a), written as a first-order system.

\textbf{2. High frequencies.}  By Proposition \ref{hfcrit} the
uniform Evans condition for high frequencies is equivalent to the
existence of $c>0$ such that $|d^2(\zeta)|\geq c$ for $\zeta$ on the
positive unit sphere $S^+:=\{\zeta:|\zeta|=1, \gamma\geq 0\}$, a
compact set. Thus, it suffices to show that in a small neighborhood
of any $\uzeta\in S^+$, the subspaces $e_{-,p,k}(\zeta)$, $k=\uu,(\nu,q,a)$
(resp. $\ker\Gamma_{*,k}^{sc}(\zeta)$, $k=\uu,(\nu,q,a)$) are close when
$|\nu-\unu|+ |q-\uu|+|a|$ is small enough. Recall from
\eqref{linprinc}(b) and \eqref{pp2} that those spaces depend on the
profile only through $w(0)$.
When $|\nu-\unu|+ |q-\uu|+|a|$ is small,
we have  $w(0)\approx \uu$, so
\begin{align}\label{hh5}
d^2_{\nu,q,a}(\zeta)\approx d^2_{\uu,\unu}(\zeta).
\end{align}

\textbf{3. Bounded frequencies}.  By compactness it suffices to show
that for $\zeta$ near some fixed $\uzeta$, the corresponding spaces
appearing in \eqref{hh4}(b) are close
for $|\nu-\unu|+ |q-\uu|+|a|$
sufficiently small.  This is true for the spaces
$\ker\Gamma_k(\zeta)$,
$k=(\uu,\unu),\;(\nu,q,a)$,
since they depend on the profile only through $w(0)$. The treatment
of $\bE^-_k(\zeta)$ requires the conjugator $P(z,\zeta,k)$ of Lemma
\ref{conjugation}, where now we write $k$ instead of $p$ for
parameters..

For $\zeta\neq 0$ let $\bE^-_\infty(\zeta,k)$ be the stable subspace
of $\cG(\infty,\zeta,k)$, which depends on the profile only through
$w(\infty)$, and recall from the discussion below \eqref{b37w} that
\begin{align}\label{hh6}
\bE^-(\zeta,k)=P(0,\zeta,k)\bE^-_\infty(\zeta,k),\;\;
k=(\nu,q,a),\;(\uu,\unu).
\end{align}
The dependence of $P(0,\zeta,k)$ on the profile is \emph{not}
through $w(0)$ alone
but on the entire profile. However, recalling from Lemma
\ref{conjugation} and (the proof of) Corollary \ref{contcor} that
$P(0,\zeta,k)$ and $\bE^-_\infty(\zeta,k)$ depend continuously on
the parameter $k$, we conclude  from \eqref{hh6} and the definition
of the Evans function that
$$
D_{\nu,q,a}(\zeta)\approx D_{\uu,\unu}(\zeta),
$$
for $|\nu-\unu|+ |q-\uu|+|a|$ sufficiently small. For $\zeta$ near
$0$, we replace $\bE^-_\infty(\zeta,k)$ by
$\bE^-_\infty(\hat\zeta,\rho,k)$ (recall \eqref{ab2}) in this
argument. (This bounded frequency argument is the same as the proof
of Corollary \ref{contcor}, but with the relevant parameters now
explicitly identified.)
\end{proof}

\section{Uniform Evans stability of small-amplitude layers for symmetric--dissipative systems}\label{stabsec}

In this section we prove Corollary \ref{symmstab}, which shows that
the uniform Evans condition Definition \ref{hfevans}(b) holds for
small-amplitude layers for symmetric-dissipative systems under
several types of boundary conditions. By Theorem \ref{smallred} it
suffices to show stability of constant layers for
symmetric--dissipative systems.

\subsection{The strictly parabolic case}\label{strictparabcase}

For clarity, we first carry out the simpler strictly parabolic case.

\begin{proof}[Proof of Corollary \ref{symmstab} in the case $N=N'$]
\textbf{}

Instability for $\rank \Upsilon_3>N^2_-$ follows again by
Proposition \ref{smalltrans} combined with Lemma \ref{Rousset},.

\textbf{1. Dirichlet boundary conditions. } By Theorem
\ref{smallred} it is sufficient to prove stability of the constant
layer $w(z)=p$. The matrix $Y_2'(p)$ is now an invertible $N\times
N$ matrix, so the boundary condition $Y_2'(p)u(0)=h$ can be written
$u(0)=Y_2'^{-1}(p)h:=g$.   With $\lambda=i\tau+\gamma$ we consider
decaying solutions $u(x_d,\lambda,\eta)$ of the Fourier-Laplace
transformed problem with coefficients evaluated at $p$:
\begin{align}\label{res}
\begin{split}
&(a)\;\;\lambda A_0 u + A_d u' +i \sum_{k\ne d } A_k\eta_k u  \\
&- B_{dd}u'' - i\sum_{k\ne d} (B_{dk}+B_{kd})\eta_k u' +\sum_{j,k\ne
d} \eta_j
\eta_k B_{jk}u= 0,\\
&(b) \;\;u(0)=g,
\end{split}
\end{align}
where the $A_j$ are symmetric, $A_0$ is positive definite, and
$B_{jk}$ is dissipative:
\begin{equation}
\label{diss} \Re \sum_{jk}B_{jk}\xi_j \xi_k \ge \theta
|\xi|^2,\;\theta>0, \text{ for all }\xi\in \RR^d.
\end{equation}
Note that by the coordinate change
\begin{align}\label{coord}
u=(A_0)^{-1/2}w
\end{align}
we can take $A_0=I$ without loss of generality, another advantage of
constant-coefficients.  We do this in the remainder of section
\ref{stabsec}.

By Remark \ref{homgen} we must establish the trace estimate
\begin{equation}\label{b1}
|u'(0)|\leq C\Lambda |g|,
\end{equation}
where $\Lambda\sim |1,\tau,\gamma|^{1/2}+|\eta|$.  In this case it
is easy to treat frequencies of all sizes, so here we do not make
use of the reduction to bounded frequencies effected by Corollary
\ref{symmhf}.

Taking the real part of the inner product of $u$ with \eqref{res},
we obtain
\begin{equation} \label{real}
\begin{aligned}
&\Re \lambda \langle u,  u\rangle - \frac{1}{2} g\cdot A_d g - \Re
g\cdot B_{dd}u'(0) +
\langle u', \Re B_{dd}u'\rangle \\
&\quad -\sum_{k\ne 1}  \langle \eta_ku, \Re i
(B_{dk}+B_{kd})u'\rangle +\sum_{j,k\ne d} \eta_k \eta_j \langle u,
\Re B_{jk}u\rangle =0.
\end{aligned}
\end{equation}
By \eqref{diss}, we obtain, extending to the whole line by $u\equiv
0$ on $x\le -\frac{1}{|\eta|}$ and $u= (x|\eta|+1)u(0)$ on
$-\frac{1}{|\eta|}\le x\le 0$, taking the Fourier transform, and
accounting for errors introduced by extension, the G\"arding
inequality
\begin{equation}
\label{j1}
\begin{aligned}
&\langle u', \Re B_{dd}u'\rangle
+\sum_{k\ne d}  \langle \eta_ku, \Re i (B_{dk}+B_{kd})u'\rangle \\
&\quad +\sum_{j,k\ne d} \eta_k \eta_j \langle u, \Re B_{jk}u\rangle
\ge \theta(|u'|^2_2+ |\eta|^2 |u|_2^2)
- C|\eta||u(0)|^2,
\end{aligned}
\end{equation}
where $\theta>0$.   Here $C|\eta||u(0)|^2$ is the error due to
extension in the G\"arding inequality.  It is an upper bound for the
left side of \eqref{j1}, computed using the explicit formula for the
extension of $u$ and inner products in $x\leq 0$.
 Combining \eqref{j1} with \eqref{real}, we obtain
\begin{equation}
\label{real2}
\begin{aligned}
(\Re \lambda +|\eta|^2)|u|_2^2 + |u'|_2^2 \le C( |g||u'(0)|+
(1+|\eta|)|g|^2).
\end{aligned}
\end{equation}

Similarly, taking the real part of the inner product of $-u''$ with
\eqref{res}, we obtain after integration by parts,
\begin{equation}
\label{realderiv}
\begin{aligned}
(\Re \lambda +|\eta|^2)|u'|_2^2 + |u''|_2^2 \le
C\left((|\lambda|+|\eta|+|\eta|^2\right) |g||u'(0)|+
(1+|\eta|)|u'(0)|^2).
\end{aligned}
\end{equation}

Using the Sobolev bound
\begin{equation}\label{Sob}
|u'(0)|^2\le |u'|_2|u''|_2\le C_\delta\Lambda|u'|_2^2 +
\delta|u''|_2^2/\Lambda,
\end{equation}
and the estimates on $|u'|_2^2$ and $|u''|_2^2$ coming from
\eqref{real2} and \eqref{realderiv} respectively, we obtain
\begin{align}\label{Sob2}
\Lambda|u'(0)|^2\leq C(\Lambda^2|g||u'(0)|+\Lambda^3|g|^2).
\end{align}
Since
\begin{align}
\Lambda^2|g||u'(0)|\leq
C_\delta\Lambda^3|g|^2+\delta\Lambda|u'(0)|^2,
\end{align}
we find that \eqref{Sob2} implies
\begin{align}\label{sob3}
\Lambda|u'(0)|^2\leq C\Lambda^3|g|^2.
\end{align}

\textbf{2. Neumann boundary conditions. }In place of \eqref{res}(b)
we now have the boundary condition
\begin{align}\label{a40}
u'(0)=h,
\end{align}
and we will continue to write $u(0)=g$ in what follows.  By Remark
\ref{homgen} it suffices to establish the trace estimate
\begin{align}\label{a41}
\Lambda|u(0)|\leq C|h|
\end{align}
for a constant $C$ independent of $\zeta$.   To shorten the argument
we now make use of the reduction to bounded frequencies $|\zeta|\leq
R$ provided by Proposition \ref{symmhf}.   In this regime the
estimate \eqref{a41} is equivalent to
\begin{align}\label{a42}
|u(0)|\leq C\Lambda^2|h|,
\end{align}
so we proceed now to prove the latter estimate.

Again we use the Sobolev inequality
\begin{align}\label{a43}
|u(0)|^2\le |u|_2|u'|^2_2\leq \delta|u|^2_2+C_\delta|u'|_2^2.
\end{align}
Letting $\Lambda':=\gamma^{\frac{1}{2}}+|\eta|$, for $\Lambda'\geq
c>0$ we have from \eqref{real2} and \eqref{realderiv}
\begin{align}\label{a44}
\begin{split}
&|u|^2_2\leq \frac{C}{\Lambda'^2}|g||h|+\frac{C}{\Lambda'}|g|^2\\
&|u'|^2_2\leq C\Lambda^2|g||h|+\frac{C}{\Lambda'}|h|^2.
\end{split}
\end{align}
Substituting into \eqref{a43} and taking $\delta$ small enough, we
absorb the terms involving $g$ to obtain
\begin{align}\label{a45}
|u(0)|^2\leq C\Lambda^4|h|^2.
\end{align}

By assumption $N^2_-=N'=N$ and the fact that $N_+=N_b-N^2_-=0$, the
boundary term $-A_d g\cdot g$ in \eqref{real} has favorable sign.
Instead of \eqref{real2} we find
\begin{align}\label{a46}
(\Re \lambda +|\eta|^2)|u|_2^2 + |u'|_2^2 +|u(0)|^2\le C( |g||h|+
|\eta||g|^2).
\end{align}
For $\Lambda'\leq c$ with $c>0$ small enough, we can absorb the
terms involving $g$ in \eqref{a46} to deduce \eqref{a45}.   This
completes the proof of \eqref{a42} for bounded frequencies.

\textbf{3. Instability.} The fact that the Evans condition fails for
$\rank \Upsilon_3>N^2_-$ follows by Corollary \ref{smalltrans}
combined with Lemma \ref{Rousset}.
\end{proof}

\begin{rem}\label{garding}
 The argument of \cite{GG} establishing stability of
small-amplitude Dirichlet profiles in the non-constant coefficient,
Laplacian-viscosity case required the weighted Poincar\'e estimate
\begin{align}\label{a50}
 \int_0^{+\infty} |w'(z)||u(z)|^2\, dz\le \|zw'\|_{L^1(0,+\infty)}
\|u'\|_{L^2(0,+\infty)},
\end{align}
established there for $u(0)=0$.  This estimate was also used in the
one-dimensional treatments of \cite{GS,R3}.  In the multidimensional
case of general $B_{jk}$, the approach of \cite{GG} also requires a
careful estimate of the error due to extension in the application of
Garding's inequality. These technicalities disappear in the
constant-coefficient limit, a major simplification. However, it may
be possible to use the argument of \cite{GG} in certain cases
involving variable multiplicities not covered by our structural
assumptions.

\end{rem}

\subsection{The partially parabolic case}\label{gencase}
We now treat the case that $B_{jk}$ is only semidefinite.

\begin{proof}[Proof of Corollary \ref{symmstab} ($N> N'$)]
\textbf{}
Instability for $\rank \Upsilon_3>N^2_-$ follows again by
Corollary \ref{smalltrans} combined with Lemma \ref{Rousset}.

\textbf{1. Dirichlet boundary conditions. } It is sufficient to
consider the constant-coefficient equation
\begin{equation}\label{a1}
\lambda  u + A_d u' +i \sum_{k\ne d } A_k\eta_k u  - B_{dd}u'' -
i\sum_{k\ne d} (B_{dk}+B_{kd})\eta_k u' +\sum_{j,k\ne d} \eta_j
\eta_k B_{jk}u= 0,
\end{equation}
where $A_j$ are symmetric, $A_0$ positive definite, and $B_{jk}$ is
now block-diagonal and dissipative:
\begin{equation}\label{a2}
\Re \sum \xi_j \xi_k B^{22}_{jk}>0 \text{ for }\xi\in \RR^d\setminus
\{0\}.
\end{equation}
For later reference we record the first and second components of
\eqref{a1}:
\begin{align}\label{e5.20}
\begin{split}
&(a)\;(i\tau+\gamma)u_1
+i \sum_{k\ne d } A_k^{11}\eta_k u_1
+i \sum_{k\ne d } A_k^{12}\eta_k u_2
+A_d^{11} u_1'+A_d^{12} u_2'=0\\
&(b)\;(i\tau+\gamma)u_2
+i \sum_{k\ne d } A_k^{21}\eta_k u_1 +i \sum_{k\ne d }
A_k^{22}\eta_k
u_2+A_d^{21} u_1'+A_d^{22} u_2'\\
&- B_{dd}^{22}u_2'' - i\sum_{k\ne d} (B_{dk}^{22}+B_{kd}^{22})\eta_k
u'_2+\sum_{j,k\ne d} \eta_j \eta_k B_{jk}^{22}u_2= 0.
\end{split}
\end{align}
By Proposition \ref{symmhf}, we need only consider bounded
frequencies $|\zeta|\le R$; this greatly simplifies the analysis.


{\it Case (i) Totally outgoing flow.} We first consider the totally
outgoing case $A_d^{11}<0$, for  which the boundary conditions in
\eqref{a1} are
\begin{align}\label{a3}
u_2(0)=g.
\end{align}
By Remark \ref{homgen} a trace estimate that is equivalent to the
(bounded-frequency) Evans hypothesis is
\begin{align}\label{a4}
|u_1(0)|+|u'_2(0)|\leq C\Lambda |g|,
\end{align}
where $\Lambda\sim |1,\tau,\gamma|^{1/2}+|\eta|$. We proceed to
prove \eqref{a4} assuming \eqref{a2} and $A_d^{11}<0$.

Pairing \eqref{a1} with $u$ we obtain the usual Friedrichs estimate
\begin{align}\label{a8}
\begin{split}
&\gamma|u|^2+|\eta|^2|u_2|^2+|u_2'|^2+|u_1(0)|^2\leq\\
&\quad C(|u_2'(0)||u_2(0)|+|u_2(0)|^2+|\eta||u_2(0)|^2),
\end{split}
\end{align}
where $|\cdot|$ denotes an $L^2(x)$ norm for interior terms and a
$\CC^n$ norm for boundary terms, and we have dropped the subscript
``2" on interior norms. Here the last term on the right represents
the G\"arding error (from extension) as well as a boundary term from
integration by parts. The $|u_1(0)|^2$ on the left is there because
of the favorable sign
of $A_d^{11}$.   From \eqref{a8} we obtain
\begin{align}\label{a8a}
|u_1(0)|^2\leq C\Lambda|g|^2+\delta|u_2'(0)|^2.
\end{align}

Similarly, differentiating \eqref{a1}
and pairing with $u'$ we obtain
\begin{align}\label{a9}
\begin{split}
&\gamma|u'|^2+|\eta|^2|u_2'|^2+|u_2''|^2+|u_1'(0)|^2\leq\\
&\quad C\left(|u_2''(0)||u_2'(0)|+(1+|\eta|)|u_2'(0)|^2\right),
\end{split}
\end{align}
where again the terms on the right represent either G\"arding error
or boundary terms from integration by parts.

We now examine $|u_2'(0)|$.   First we have
\begin{align}\label{a10}
|u_2'(0)|^2\leq |u_2'||u_2''|\leq C_\delta\Lambda
|u_2'|^2+\frac{\delta}{\Lambda}|u_2''|^2.
\end{align}
>From \eqref{a8} we find easily
\begin{align}\label{a10a}
C_\delta\Lambda|u_2'|^2\leq C\Lambda^2|g|^2+\delta |u_2'(0)|^2.
\end{align}
We claim that the last term in \eqref{a10} satisfies
\begin{align}\label{a10b}
\frac{\delta}{\Lambda}|u_2''|^2\leq
C\Lambda^2|g|^2+\delta|u'_2(0)|^2+\delta |u_1(0)|^2.
\end{align}
With \eqref{a8a} this will complete the proof of \eqref{a4}.

To analyze the last term in \eqref{a10} we first use
\eqref{e5.20}(b) to estimate $|u_2''(0)|$:
\begin{align}\label{a12}
|u_2''(0)|\leq
C\left(|\lambda||u_2(0)|+|\eta||u(0)|+|u_1'(0)|+|u_2'(0)|+|\eta|^2|u_2(0)|+|\eta||u_2'(0)|\right),
\end{align}
and then substitute in \eqref{a9} to get
\begin{align}\label{a13}
\begin{split}
&\gamma|u'|^2+|\eta|^2|u_2'|^2+|u_2''|^2+|u_1'(0)|^2\leq\\
&\qquad C\big((|\lambda||u_2(0)|+|\eta||u(0)|+|u_1'(0)|+|u_2'(0)|\\
&\qquad \qquad +|\eta|^2|u_2(0)|+|\eta||u_2'(0)|)|u_2'(0)|
+(1+|\eta|)|u_2'(0)|^2\big)\leq\\
&\qquad C\big(|\lambda||u_2(0)|+|\eta||u(0)|+|\eta|^2|u_2(0)|
+(1+|\eta|)|u_2'(0)|\big)|u_2'(0)|.
\end{split}
\end{align}
Here we have used the $|u_1'(0)|^2$ on the left to absorb a term,
and then enlarged $C$.

>From \eqref{a13} we have
\begin{align}\label{a14}
\frac{\delta}{\Lambda}|u_2''|^2\leq\frac{\delta}{\Lambda}\left((|\lambda||u_2(0)|+|\eta||u(0)|+|\eta|^2|u_2(0)|)|u_2'(0)|\right)+\frac{\delta}{\Lambda}(1+|\eta|)|u_2'(0)|^2,
\end{align}
and the second term on the right can be absorbed in \eqref{a10}. We
have
\begin{align}\label{a15}
\frac{\delta}{\Lambda}|\lambda||u_2(0)||u_2'(0)|\leq
\delta\Lambda|u_2(0)||u_2'(0)|\leq
\delta\Lambda^2|g|^2+\delta|u_2'(0)|^2.
\end{align}
The other terms in the estimate of $\frac{\delta}{\Lambda}|u_2''|^2$
are similar or easier to handle, so this concludes the proof of
\eqref{a4}.

{\it Case (ii) Totally incoming flow.} It remains to treat the
totally incoming case $A_d^{11}>0$, with full Dirichlet boundary
conditions $u(0)=(u_1(0),u_2(0))=g$. A trace estimate that is
equivalent to the (bounded-frequency) Evans hypothesis is
\begin{align}\label{a4ii}
|u'_2(0)|\leq C\Lambda^{\frac{3}{2}} |g|,
\end{align}
where $\Lambda\sim |1,\tau,\gamma|^{1/2}+|\eta|$.

Making the same energy estimates as in the totally outgoing case, we
find that the only differences are that (i) there now appears a term
$C|u_1(0)|^2$ in the righthand side of \eqref{a8}, and (ii) there
now appears a term $C|u_1'(0)|^2$ in the righthand side of
\eqref{a13}. Difference (i) is harmless, since $C|u_1(0)|^2\le
C|g|^2$ is of the same order as $C|u_2(0)|^2$ terms already
appearing on the righthand side of \eqref{a8}.

Difference (ii) can be handled by estimating $|u_1'(0)|$ using
\eqref{e5.20}(a) as
$$
|u_1'(0)|\le |(A_1^{11})^{-1}| \big( C|\zeta||u(0)| + C|u_2'(0)|
\big) \le C((|\lambda|+|\eta|)|g| + |u_2'(0)|)
$$
to see that $C|u_1'(0)|^2$ contributes terms of order
$C((|\lambda|+|\eta|)^2|g|^2 + |u_2'(0)|^2)$.  Following the
previous argument and using $(|\lambda|+|\eta|)^2\leq \Lambda^4$, we
obtain \eqref{a4ii} as claimed.

\textbf{2. Neumann boundary conditions. }  We now assume
$N^2_-=N'=\mathrm{rank}\Up_3$.   Consider first the \emph{totally
outgoing} case, so $A^{11}_d<0$ and $N_+=N^1_+=0$.   The boundary
conditions for \eqref{a1} are now
\begin{align}\label{a21}
u_2'(0)=h
\end{align}
and for bounded frequencies it suffices to show
\begin{align}\label{a22}
|u_1(0)|+|u_2(0)|\leq C\Lambda^2|h|.
\end{align}
Writing $u_2(0)=g$ and $\Lambda'=\gamma^{\frac{1}{2}}+|\eta|$, since
$N_+=0$ we have now in place of \eqref{a8}
\begin{align}\label{a20}
&\gamma|u|^2+|\eta|^2|u_2|^2+|u_2'|^2+|u(0)|^2\leq
C(|u_2'(0)||u_2(0)|+|\eta||u_2(0)|^2).
\end{align}
As before for $\Lambda'\leq c$ with $c>0$ small enough, we easily
absorb the terms involving $g$ in \eqref{a20} to obtain \eqref{a22}.

To estimate $u_2(0)$ we use
\begin{align}\label{a23}
|u_2(0)|^2\leq \delta|u_2|^2+C_\delta|u_2'|^2.
\end{align}
For $\Lambda'\geq c>0$ we have from \eqref{a20} and \eqref{a13}
\begin{align}\label{a24}
\begin{split}
&|u_2|^2\leq\frac{C}{\Lambda'^2}|h||g|+\frac{C}{\Lambda'}|g|^2\\
&|u_2'|^2\leq
\frac{C}{\Lambda'^2}|h|\left(|\lambda||g|+|\eta||u(0)|+|\eta|^2|g|+(1+|\eta|)|h|\right)\\
&\qquad\quad\leq
C\Lambda^2|g||h|+\frac{C}{\Lambda'}|u(0)||h|+C|g||h|+\frac{C}{\Lambda'}|h|^2
\end{split}
\end{align}
>From \eqref{a23}, \eqref{a24} we obtain, after absorbing some terms
from the right
\begin{align}\label{a25}
|u_2(0)|^2\leq C\Lambda^4|h|^2+C|u_1(0)||h|.
\end{align}
Substituting the estimate on $|u_1(0)|$ from \eqref{a20} into
\eqref{a25}, we easily obtain \eqref{a22} after adding the estimates
for $u_1(0)$ and $u_2(0)$ and absorbing terms.

Consider finally the \emph{totally incoming case} where
$A^{11}_d>0$.   We have $\mathrm{rank} \Up_3=N^2_-=N'$, so
$N_+=N^1_+=N-N'$ and the boundary conditions for \eqref{a1} are
\begin{align}\label{a26}
u_1(0)=g_1, u_2'(0)=h.
\end{align}
Evans stability can fail in this case, even for small amplitude
profiles.  See Example \ref{cegrmk}.

\end{proof}

\begin{rems}\label{realrmks}
 1.) We note that the argument in the partially parabolic
case is even simpler  than the treatment of the one-dimensional case
in \cite{R3}, thanks mainly to the reduction to finite frequencies
and to trace rather than interior estimates. In particular, we do
not require the ``Kawashima-type'' estimate used in \cite{R3} to
obtain an interior estimate on $|u_1'|$. Nor do we require a
weighted Poincare estimate of the type \eqref{a50} used in
\cite{R3}.

2.) In the small-amplitude analysis we expect that one may drop the
constant-sign assumption on $A^{11}_d$, substituting as in \cite{R3}
the assumption that that $\Gamma_1$ be maximally dissipative with
respect to $A^{11}_d$, i.e., that $A^{11}_d$ be negative definite on
$\ker \Gamma_1$.  For {\rm bounded frequencies} the above arguments
go through in this case essentially unchanged. For {\rm  high
frequencies}  the exponential weights of \cite{GMWZ6,Z3} are not
necessary in the small-amplitude case. It may be possible to modify
the high-frequency argument of \cite{GMWZ6} to work without the
constant-sign assumption.

3.) In the case of Dirichlet boundary conditions
($\mathrm{rank}\Up_3=0$), the arguments of this section may be
simplified still further by making use of Proposition
\ref{realmaxdiss} below.  This Proposition together with Lemma
\ref{Rousset} yields uniform Evans stability for small frequencies
$|\zeta|\le r$, $r>0$ sufficiently small. Using also the reduction
to bounded frequencies (Cor. \ref{symmhf}), it follows that in order
to show uniform Evans stability we must only show nonvanishing of
the Evans function $D(\zeta)$; that is, nonexistence of nontrivial
solutions of the eigenvalue equation \eqref{a1} with homogeneous
forcing and boundary data, $f=g=0$. Under these conditions the
required estimates become almost trivial.
\end{rems}

\begin{exam}[A  counterexample]\label{cegrmk}
\textup{
Finally, we show that stability may fail in general for
Neumann boundary conditions with $N^2_-=N'=\mathrm{rank}\Up_3$,
in the  \emph{totally incoming} case $A^{11}_d>0$.
Consider the linear constant-coefficient system
$$
u_t+ A_1u_{x_1}+ A_2 u_{x_2}= \begin{pmatrix} 0\\ \Delta_x u_2 \\
\end{pmatrix},
\qquad
A_1=\begin{pmatrix} 0 & 1 \\ 1 & 0 \end{pmatrix},
\,
A_2=\begin{pmatrix} 1 & a \\ a & b \end{pmatrix}
$$
on $x_2>0$ with boundary conditions
$$
u_1|_{x_2=0}=0, \quad \D_{x_2} u_2|_{x_2=0}=0,
$$
under the assumptions
$$
b>0,\quad b-a^2<0.
$$
Then, $A_2^{11}>0$, and $N^2_-=1= \mathrm{rank}\Up_3$.
Seeking layer profiles for this linear system, we have
immediately, since these satisfy a first-order ODE in $u_2$
with initial value at $x_2=0$ by the Neumann condition an
equilibrium value,
that these are exactly the {\it constant solutions}, from
which we deduce immediately that the residual hyperbolic boundary
condition is exactly $u_1=0$.
}

\textup{
Applying now Lemma \ref{Rousset}, we find that a necessary
condition for uniform low-frequency stability
is satisfaction of the uniform Lopatinski condition for
$$
u_t+ A_1u_{x_1}+ A_2 u_{x_2}= 0
$$
with boundary condition $u_1|_{x_2=0}=0$, or
\begin{equation}\label{lop}
r_1\ne 0
\quad \hbox{\rm for }\quad
r \in \EE_- \Big(-(A_2)^{-1}(\gamma + i\tau + i\eta A_1)\Big)
\end{equation}
for each $\tau, \, \eta \in \RR$, $\gamma\in \RR^+$,
where $\EE_-$ as usual denotes the limit of the stable subspace
as $\gamma \to 0^+$.
But, \eqref{lop} is clearly violated for
$$
r=(0,1)^T, \,
\eta=1, \,
\gamma=0 , \,
\tau= b/a,
$$
in which case
$-(A_2)^{-1}(\gamma + i\tau + i\eta A_1)r= \mu r$
for $\mu=i/a$.
Continuing  eigenvector and eigenvalue $r=r(\gamma)$, $\mu=\mu(\gamma)$
while varying $\gamma$ in the positive direction, we obtain
by standard matrix perturbation theory, noting that $r^TA_2$ by
symmetry is an associated left eigenvector, that
$$
(d\mu/d\gamma)|_{\gamma=0}=
\frac{
-r^TA_2 (A_2)^{-1}\D_\gamma (\gamma + i\tau + i\eta A_1)r}
{r^T A_2 r}=
-A^{22}_2<0,
$$
verifying that $r\in \EE_-$.
}

\textup{Alternatively, we may recall from \cite{MZ2} that
$2\times 2$ two-dimensional hyperbolic constant-coefficient
systems with both outgoing and incoming characteristics and
for which $A_1$ and $A_2$ do not commute
satisfy the uniform Lopatinski condition if and only if
they are maximally dissipative, i.e., $A_2<0$ on the kernel
$\Span \{(0,1)\}$ of the residual boundary condition,
or $A_2^{22}<0$, to make the same conclusion without calculation.
}

\textup{Note that this example does not yield one-dimensional
low-frequency instability, and in fact is one-dimensionally stable.
For, taking the real part of the inner product of $u$ against the associated
eigenvalue equation $\lambda u+ A_2 u'=u_2''$,
we obtain $\Re \lambda |u|_2^2 + |u_2'|_2^2=0$, yielding
$u_2\equiv 0$.  Substituting into the $u_1$ equation, we find
by direct computation that $u_1= e^{-\lambda z}u_1(0)=0$.
This gives one-dimensional Evans stability for bounded frequencies;
the result for high frequencies follows by Proposition \ref{symmhf}.
}
\end{exam}

\subsection{Maximal dissipativity of residual hyperbolic boundary conditions}\label{maxdiss}

Before presenting calculations for example systems, we digress
slightly to complete the picture of qualitative behavior.
Transversality and uniform Lopatinski condition follow for
small-amplitude profiles of symmetric dissipative systems by Evans
stability combined with Lemma \ref{Rousset}, under mild structural
conditions on multiplicity of characteristics. Here, we note that
they may alternatively be verified directly, and without any
assumptions on multiplicity of characteristics.  Proposition
\ref{realmaxdiss} yields the additional information that residual
boundary conditions for small-amplitude layers of symmetric
dissipative systems satisfy not only the uniform Lopatinski
condition,
but also the stronger condition of maximal dissipativity: $S A_d<0$
on the kernel of the linearized hyperbolic boundary conditions
$\Gamma_{res}$, where $S>0$ is the symmetric positive definite
matrix symmetrizing $ A_j$.

This result was established first in \cite{BRa} for the case of
symmetric, strictly parabolic $B_{jk}$. It was established in Lemma
4.3.1, \cite{Met4} for the strictly parabolic (not necessarily
symmetric) case; see also Lemma 7 \cite{BSZ}. The argument is based
on dissipative integral estimates of a similar flavor to those used
to establish uniform Evans stability in Section \ref{stabsec}.

\begin{prop}\label{realmaxdiss}
Consider the class of symmetrizable dissipative systems (Definition
\ref{defsymm}) and the class of decoupled boundary conditions
\eqref{viscbcd} that are full Dirichlet in the parabolic variable
$u_2$ ($N''=0$) and maximally dissipative in the hyperbolic variable
$u_1$: that is, $(SA_d)^{11}$ is negative definite on
$\ker\Gamma_1$. Then sufficiently small-amplitude noncharacteristic
boundary layers are transverse, and the associated residual boundary
conditions are maximally dissipative.
\end{prop}

\begin{proof}
\textbf{1. }By continuity both assertions reduce to the
corresponding assertions for the limiting constant layer, say
$w(z)=p=(p^1,p^2)$. Transversality was already proved in Corollary
\ref{smalltrans}, which applies also for the case of variable
multiplicities.
As in section \ref{stabsec} (see \eqref{coord}) we may without loss
of generality take $S=A_0=I$ and $A_\nu=A_d$ (or
$\nu=(0,\dots,0,1)$). We must show that $A_d<0$ on the kernel of the
residual boundary condition $\Gamma_{res}$ for the linearized
hyperbolic problem at $p$.

Let $C_{\nu,p}$ denote the manifold of states $q$ near $p$ such that
there exists a profile $W(z,q)$ satisfying \eqref{layeq} with
\begin{align}\label{a60}
(g_1,g_2):=(\Up_1(p^1),\Up_2(p^2))
\end{align}
and $W(z,q)\to q$ as $z\to\infty$.   Let $\dot{\cC}_{\nu,p}$ be the
space of $\dot{q}$ such that there is a solution
$\dot{w}(z,\dot{q})$ of the linearized profile problem
\eqref{linlayeq}, \eqref{linlaybc} with $\dot{w}(z,\dot{q})\to
\dot{q}$ as $z\to \infty$. The entries of $\cG_+(\nu)$ in
\eqref{Gform} are now evaluated at $w(z)=p$. It is not hard to check
(see \cite{Met4}, Prop. 5.5.5) that
\begin{align}\label{a61}
T_p\cC_{\nu,p}=\dot{\cC}_{\nu,p}:=\ker\Gamma_{res}.
\end{align}
The linearized hyperbolic problem at $p$ is
\begin{align}\label{a61a}
\begin{split}
&v_t+\sum_{j=1}^d A_j(p)\partial_jv=f \\
&v|_{x_d=0}\in T_p\cC_{\nu,p}.
\end{split}
\end{align}
Therefore, we must show
\begin{align}\label{a62}
A_d<0 \text{ on }\dot{\cC}_{\nu,p}.
\end{align}

\textbf{2. }Set $\Gamma_1:=\Up_1'(p^1)$ and define
\begin{align}\label{a63}
\cN=\{(n,0)\in\bR^N: n\in\ker \Gamma_1\}.
\end{align}
In Lemma \ref{a64} below we show
\begin{align}\label{a65a}
\dot{\cC}_{\nu,p}=\cN\oplus\bE_-(A_d^{-1}B_{dd}),
\end{align}
where $\bE_-(M)$ denotes the stable subspace of $M$.  Since
$\cN\subset\ker B_{dd}$ and $A_d<0$ on $\cN$  (recall $A_d^{11}<0$
on $\ker\Gamma_1$), the result thus follows by Lemma \ref{neutral}
below.
\end{proof}

\begin{lem} [\cite{Z1}] \label{neutral}
Let $A$ and $B$ be $N\times N$ matrices, where $A$ is symmetric and
invertible, and $B$ is positive semidefinite, $\Re(B)\geq 0$, satisfying
in addition the block structure condition $\ker B=\ker \Re(B)$. Given
any subspace $\cN$ on which $A$ is negative definite, then $A$ is
negative definite also on the subspace $\bE_-(A^{-1}B)\oplus (
\cN\cap \ker B)$, where $\bE_-(M)$ denotes the stable subspace of
$M$.
\end{lem}

\begin{proof}
Suppose that $x_0\neq 0$ lies in the subspace $\bE_-(A^{-1}B)\oplus
(\cN\cap \ker B)$, i.e.,
$$x_0=x_1+x_2$$
where $x_1\in\bE_-(A^{-1}B)$, $x_2\in ( {\cN}\cap \ker B)$.
Define $x(t)$ by the ordinary differential equation $
x'=A^{-1}Bx,\quad x(0)=x_0. $ Then $x(t)\rightarrow x_2$ as
$t\rightarrow +\infty$ and thus
\begin{align}\notag
\lim_{t\rightarrow +\infty}\langle x(t),Ax(t)\rangle =\langle
x_2,Ax_2\rangle \leq 0,
\end{align}
with equality only if $x_2=0$.
On the other hand,
$$
\langle x, Ax\rangle'  =  2\Re\langle A^{-1}Bx,Ax\rangle
                  =  2\Re\langle Bx,x\rangle   \geq  0,
$$
yielding $\langle x_0, Ax_0\rangle \leq \langle x_2,Ax_2\rangle\leq
0$.
Thus,
$\langle x_0, Ax_0\rangle <0$ unless $x_2=0$ and $\langle \Re(B)x,x\rangle
\equiv 0$, in particular, $\langle \Re(B)x_1, x_1\rangle$, so that
$x_1\in \ker \Re(B)=\ker(B)\subset \ker A^{-1}B$.
But, this is impossible, since $x_1\in \EE_-(A^{-1}B)$.
\end{proof}

It just remains to prove:
\begin{lem}\label{a64}
Let $\dot{\cC}_{\nu,p}$ be as in \eqref{a61}, the space of $q$ such
that there is a solution $\dot{w}(z,q)$ of the linearized profile
problem \eqref{linlayeq}, \eqref{linlaybc} with $\dot{w}(z,q)\to q$
as $z\to \infty$.   Then
\begin{align}\label{a65}
\dot{\cC}_{\nu,p}=\cN\oplus\bE_-(A_d^{-1}B_{dd}),
\end{align}
for $\cN$ as in \eqref{a63}.
\end{lem}

\begin{proof}
\textbf{1. }With $w(z)=p$ the constant layer, define as in
\eqref{C13} the $N'\times N'$ matrix
\begin{align}\label{a66}
G_d(p):=(B^{22}_{dd})^{-1} \left( A_d^{22} - A_d^{21} (
A_d^{11})^{-1} A_d^{12}\right)(p).
\end{align}
A short computation shows
\begin{align}\label{a67}
\bE_-(A_d^{-1}B_{dd}(p))=\{\begin{pmatrix}-(A_d^{11})^{-1}A_d^{12}r^2\\r^2\end{pmatrix}:
r^2\in\bE_-(G_d(p))\}.
\end{align}

\textbf{2. }Consider the linearized profile equation
\eqref{linlayeq} at $p$ with
$\dot{W}=(\dot{w}_1,\dot{w}_2,\dot{w}_3)$. For any
$(q^1,q^2)\in\bR^N$, this equation is easily integrated to yield a
solution with $(\dot{w}_1(z),\dot{w}_2(z))\to(q^1,q^2)$ as
$z\to\infty$:
\begin{align}\label{a68}
\begin{split}
&\dot{w}_1(z)=-(A_d^{11})^{-1}A_d^{12}e^{zG_d(p)}(G_d(p))^{-1}r^2+q^1\\
&\dot{w}_2(z)=e^{zG_d(p)}(G_d(p))^{-1}r^2+q^2\\
&\dot{w}_3(z)=e^{zG_d(p)}r^2, \text{ where }r^2\in \bE_-(G_d(p)).
\end{split}
\end{align}
Setting $\Up_1'(p^1)\dot{w}_1(0)=0$ and $\Up_2'(p^2)\dot{w}_2(0)=0$
we find
\begin{align}\label{a69}
\begin{split}
&\Up_1'(p^1)q^1=\Up_1'(p^1)(A_d^{11})^{-1}A_d^{12}(G_d(p))^{-1}r^2\\
&q^2=-(G_d(p))^{-1}r^2.
\end{split}
\end{align}
This gives
\begin{align}\label{a70}
\dot{\cC}_{\nu,p}=\{(q^1,q^2):\Up_1'(p^1)\left(q^1+(A_d^{11})^{-1}A_d^{12}q^2\right)=0,q^2\in\bE_-(G_d(p))\}.
\end{align}
Together with \eqref{a67}, this implies the result.
\end{proof}

\begin{rem}\label{max}
 A theorem of Rauch \cite{Ra} asserts the converse result that any
maximally dissipative boundary condition may be realized as the
residual boundary condition associated with some symmetric
dissipative viscosity. See also the interesting recent
investigations of Sueur \cite{Su} in which he establishes that any
(nonstrictly) dissipative boundary condition may be realized as the
residual boundary condition associated with some (not necessarily
symmetric) dissipative viscosity.
\end{rem}

\section{The compressible Navier-Stokes and viscous MHD
equations}\label{egs}

In this section we present computations of $\cC$ manifolds for some
classical symmetric-dissipative systems with various boundary
conditions, including standard inflow/outflow conditions for
Navier-Stokes. For the small-amplitude case the results of sections
\ref{globalC} and \ref{stabsec}  apply and give much information
about when profiles satisfy the uniform Evans stability condition or
transversality, and when the reduced hyperbolic boundary conditions,
which are expressed in terms of $\cC$ manifolds, satisfy the maximal
dissipativity or uniform Lopatinski conditions.   In a few cases we
will say something about such properties for large amplitude
profiles.

\subsection{Isentropic Navier--Stokes equations} \label{Isentropic}
We start with the simplest case of noncharacteristic boundary layers
for the isentropic compressible Navier--Stokes equations. In this
case we are able to give a detailed description of possible
boundary-layer connections, including large amplitude layers, and
the resulting residual boundary conditions.

.\\

\textbf{Computation of residual boundary conditions. } Consider the
isentropic Navier--Stokes equations
\begin{align}
\rho_t+ (\rho u)_x +(\rho v)_y&=0,\label{eq:mass}\\
(\rho u)_t+ (\rho u^2)_x + (\rho uv)_y + P_x&=
(2\mu +\eta) u_{xx}+ \mu u_{yy} +(\mu+ \eta )v_{xy},\label{eq:momentumx}\\
(\rho v)_t+ (\rho uv)_x + (\rho v^2)_y + P_y&= \mu v_{xx}+ (2\mu
+\eta) v_{yy} + (\mu+\eta) u_{yx}\label{eq:momentumy}
\end{align}
on the half-space $y>0$, where $\rho$ is density, $u$ and $v$ are
velocities in $x$ and $y$ directions, and $P=P(\rho)$ is pressure,
and $\mu>|\eta|\ge0$ are coefficients of first (``dynamic'') and
second viscosity. We assume a monotone pressure function
\begin{equation}\label{mon}
P'(\rho)>0.
\end{equation}

Seeking boundary-layer solutions $(\rho, u,v)(y)$, and  setting
$z:=y$, we obtain profile equations
\begin{align}\label{one}
\begin{split}
&(\rho v)'=0\\
&(\rho u v)'=\mu u''\\
&(\rho v^2)'+P(\rho)'=\epsilon v'',
\end{split}
\end{align}
where ``$'$'' denotes $d/dz$, and $\epsilon:=2\mu+\eta$. Integrating
$\int^z_{+\infty}$, we obtain
\begin{align}\label{1}
\begin{split}
&\rho v(z)=\rho_\infty v_\infty:=m_\infty\\
&\mu u'=\rho u v-\rho_\infty u_\infty v_\infty\\
&\epsilon v'=\rho v^2+P(\rho)-(\rho_\infty
v_\infty^2+P(\rho_\infty),
\end{split}
\end{align}
where $(\rho_\infty, u_\infty, v_\infty):=(\rho,u,v)(+\infty)$ and
$m:=\rho v$ denotes momentum in the normal ($z$) direction.
Within the set of allowable boundary conditions at $z=0$ in our
abstract framework, we consider a subset consisting of linear
conditions
\begin{align}\label{qq1}
\Gamma U(0)=g=(g^1,g^2,0)
\end{align}
that are
pure Dirichlet ($N''=0)$ or else mixed Dirichlet--homogeneous
Neumann boundary conditions ($N''=N'=2$), and which includes some of
the most commonly used conditions.

 With $\RR^3_+=\{(\rho_\infty,u_\infty,v_\infty):\rho_\infty>0\}$,
$U(z)=(\rho(z),u(z),v(z))$,
$U_\infty=(\rho_\infty,u_\infty,v_\infty)$, and
$U_0=(\rho_0,u_0,v_0)$,  let
\begin{align}\label{2}
\begin{split}
&\cC_{\Gamma,g}=\{U_\infty\in\RR^3_+:\text{there exists }U(z)\text{
satisfying
}\eqref{1}\text{ together with }\\
&\qquad U(+\infty)=U_\infty\text{ and the boundary conditions
\eqref{qq1} at $z=0$ }\}.
\end{split}
\end{align}
For arbitrary $U_0$ with $\rho_0>0$ note that the constant profile
$$
U(z)=U_0
$$
determines an element $U_\infty=U_0\in\cC_{\Gamma,g}$, where
$g=\Gamma U_0$. The goal here is to determine $\cC_{\Gamma,g}$ in
some (not necessarily small) neighborhood of $U_\infty=U_0$, and to
understand how $\cC_{\Gamma,g}$ changes with $g$.
We consider separately the outflow and inflow case.

\medskip
{\bf Outflow with Dirichlet conditions.} We first consider the
outflow case $v_0<0$, with Dirichlet boundary conditions.    We have
$\overline{A}_d^{11}<0$ now, so $N^1_+=0$, $N_b=2$ and thus we
prescribe only the parabolic variables on the boundary:
\begin{equation}\label{NSis}
\Gamma U(0)=g=(u_0,v_0).
\end{equation}

Noting that $m_\infty=m_0<0$ since $v_0<0$ and $\rho_0 >0$, and
rewriting the second equation in \eqref{1} as
\begin{align}\label{3}
\mu u'=m_\infty (u-u_\infty),
\end{align}
we obtain solutions
\begin{align}\label{4}
u(z)=u_\infty\left(1-(1-\frac{u_0}{u_\infty})e^{\frac{m_\infty}{\mu}z}\right)
\end{align}
satisfying the conditions (on $u(z)$) in \eqref{2} with no
restrictions so far except $\rho_0>0$, $u_\infty\in\RR$.

Rewrite the third equation in \eqref{1} using the first equation to
get
\begin{align}\label{5}
\begin{split}
&\epsilon v'=m_\infty
v+P\left(\frac{m_\infty}{v}\right)-\left(m_\infty
v_\infty+P\left(\frac{m_\infty}{v_\infty}\right)\right)=\\
&m_\infty(v-v_\infty)+c^2(\rho,\rho_\infty)\left(\frac{m_\infty}{v}-\frac{m_\infty}{v_\infty}\right)=\\
&\left(m_\infty-\frac{m_\infty c^2(\rho,\rho_\infty)}{v
v_\infty}\right)(v-v_\infty),
\end{split}
\end{align}
where
\begin{equation}\label{cfn}
c^2(\rho,\rho_\infty):= \frac{ P(\rho)- P(\rho_\infty)} { \rho
-\rho_\infty}
>0
\end{equation}
by assumption \eqref{mon}. Note that if we set
\begin{align}
c_\infty=\sqrt{P'(\rho_\infty)}
\end{align}
we have
$c(\rho_\infty,\rho_\infty)=c_\infty$.

With $(u_0,v_0)$ fixed, consider first a candidate state $U_\infty$
for membership in $\cC_{\Gamma,g}$ satisfying
\begin{align}\label{7}
v_\infty<0, \,  v_\infty +c_\infty>0, \text{ so }|v_\infty|<c_\infty.
\end{align}
For $z$ large we have $v\approx v_\infty$, $\rho\approx \rho_\infty$
and thus
\begin{equation}\label{sign}
m_\infty-\frac{m_\infty c^2(\rho,\rho_\infty)}{v v_\infty}>0.
\end{equation}
This implies that \eqref{5} has no nontrivial bounded solution with limit
$v_\infty$ as $z\to +\infty$. In this case, the manifold
$\cC_{\Gamma,g}$ is an $N-N_+=2$ dimensional manifold defined
by the single condition
\begin{align}\label{9}
v_\infty\equiv v_0,
\end{align}
which is the right number of conditions ($N_+=1$) for the Euler
equations at a value $U_\infty$ satisfying \eqref{7}.
Indeed, this is the classical (specified-normal velocity) Euler boundary
condition corresponding to the (specified-normal velocity, no-slip)
Navier--Stokes boundary conditions $(u,v)(0)=(u_0,v_0)$.

In the case that
\begin{align}\label{qq10}
v_\infty +c_\infty<0, \text{ so } |v_\infty|>c_\infty \;(N_+=0),
\end{align}
 we
obtain the opposite inequality
\begin{equation}\label{oppsign}
m_\infty-\frac{m_\infty c^2(\rho,\rho_\infty)}{v v_\infty}<0.
\end{equation}
Thus, in this case we have nontrivial solutions \eqref{5} with
$v(\infty)=v_\infty$.  Provided $v_0$ is close enough to $v_\infty$,
$\cC_{\Gamma,g}$ is a full neighborhood of $U_\infty$, as expected
from the general theory. That is, there are no incoming hyperbolic
characteristics and no boundary conditions imposed on the Euler
equations in this case ($\dim\cC_{\Gamma,g}=N-N_+=3$).

Uniform Evans stability of the above layers in the small amplitude
case follows by Corollary \ref{symmstab}(b).

Examining further, we see that connections to such states $U_\infty$
exist for boundary data $U_0$ such that
\begin{align}\label{rr1}
v^1_*<v_0\leq v_\infty\text{ or }v_\infty\leq v_0< v^2_*,
\end{align}
where the $v^j_*$ are the nearest rest points of \eqref{5} to
$v_\infty$.  This is because all states $v$ in the closed interval
between $v_0$ and $v_\infty$ then satisfy
\begin{equation}\label{signtwo}
m_\infty-\frac{m_\infty c^2(m_\infty/v,m_\infty/v_\infty)}{v
v_\infty}<0.
\end{equation}
In the typical (genuinely nonlinear) case that $P''(\rho)>0$, there
is a single such rest state $v^2_*>v_\infty$.  The states $v^2_*$
and $v_\infty$ then correspond to endstates of stationary viscous
shock waves on the whole line. Layers with $v<v_0$ are in this case
``compressive'', consisting of pieces of the viscous shock from
$v^2_*$ to $v_\infty$, while layers with $v>v_0$ are ``expansive''
(decreasing $\rho$), analogous more to rarefactions; see \cite{MN,
CHNZ} for further discussion.

\begin{rems}\label{solnstructure}
 There are two transitions worth mentioning. One is when
the number of hyperbolic characteristics changes; that is, at the
boundary where inequality \eqref{7} changes its sense. Note that
this transition has to do with the outer, hyperbolic solution
$U_\infty$ and so we cannot deduce that such a transition occurs
from knowledge of the Navier-Stokes boundary data $U_0$ alone, but
must know the solution of the hyperbolic equation.

A second transition has to do with the ``inner'', boundary-layer
structure, when $U_0$ goes out of range of \eqref{rr1}, with
$U_\infty$ held fixed satisfying \eqref{oppsign}.  Consider a
solution of the Navier-Stokes system with boundary data that
includes states $U_0$ with some in and some out of range of
\eqref{rr1}. In this case our description of the solution as
boundary layer plus outer smooth solution fails, and it is a natural
question to ask what happens instead. This question has been
answered for the one-dimensional case in \cite{MN}. The resolution
is that in this case there is a more complicated structure at the
boundary consisting of boundary layer plus shock or rarefaction
waves incoming to the domain: that is, a (nonsmooth)
boundary--Riemann solution.
\end{rems}

\textbf{Outflow with Neumann conditions}. We next consider  the case
of homogeneous Neumann conditions in the outflow case:
\begin{align}\label{qq8}
\Gamma U(0)=(u',v')(0)=(0,0)=g.
\end{align}
Again consider a candidate state $U_\infty$ satisfying $v_\infty<0$
for membership in $\cC_{\Gamma,g}$.  The vanishing of $u'(0)$ and
$v'(0)$ implies that $u(0)$ and $v(0)$ are rest points of \eqref{3}
and \eqref{5} respectively. Thus $u(z)\equiv u_\infty$, $v(z)\equiv
v_\infty$, so $U_\infty$ can only be the endstate of a constant
layer, and
\begin{align}\label{qq13}
\cC_{\Gamma,g}=\{U_\infty\in \bR^3:\rho_\infty>0, v_\infty<0\}.
\end{align}
In the case \eqref{7} we have $N_+=N^2_-=1$, so $N''=2>N^2_-$; thus,
by Corollary \ref{smalltrans} these layers are not transversal, and
thus are not even low frequency Evans stable by Lemma \ref{Rousset}.
Observe that the ``correct" dimension for $\cC_{\Gamma,g}$ in this
case is $N-N_+=2$.

In the case when $U_\infty$ satisfies \eqref{qq10},  we have
$N_+=0$, so $N''=N^2_-=2$ and $\cC_{\Gamma,g}$ has the right
dimension.  As expected there are no Euler boundary conditions and
by Corollary \ref{symmstab}(b), the constant layers are Evans
stable.

\medskip
{\bf Inflow with Dirichlet conditions.}   In the case of inflow
$v_0>0$, so $N^1_+=1$, $N_+=3$ and we prescribe
\begin{align}\label{qq12}
\Gamma U(0)=(\rho_0,u_0,v_0)=g.
\end{align}
Let $U_\infty$ be a candidate state for $\cC_{\Gamma,g}$ with
$v_\infty>0$. The equation for the transverse velocity $u$ decouples
as \eqref{3}; however, due to the opposite sign $m_\infty>0$, this
equation now has no nonconstant solutions converging to $u_\infty$.
Hence, the transverse velocity is specified as
$$
u_\infty=u_0.
$$
Continuing as before, we find two cases, according as
$v_\infty-c_\infty\gtrless 0$, i.e.,
$$
v_\infty\gtrless c_\infty.
$$
In the first case $N_+=3$ and we find as before that there are no
nontrivial solutions of \eqref{5}, whence $v_\infty=v_0$, and
$\rho_\infty=m_\infty/v_\infty= m_0/v_0=\rho_0$. Thus, only the
constant layer is possible, and the induced hyperbolic boundary
conditions are full Dirichlet,
$$
U_\infty=U_0,\;\; \cC_{\Gamma,g}=\{U_0\}
$$
in agreement with the classical Euler conditions.

In the second case $N_+=2$ and there exist nontrivial connections on
the range specified by \eqref{signtwo}. The induced hyperbolic
boundary conditions defining $\cC_{\Gamma,g}$ are
$$
u_\infty=u_0, \quad m_\infty=m_0,
$$
specifying transverse velocity and {\it momentum}, an interesting
variation on \eqref{9} in the outflow case.

Uniform Evans stability of the above layers in the small amplitude
case follows from Corollary \ref{symmstab}(b).

\medbreak

\textbf{Inflow with mixed boundary conditions.}  Finally, we
consider the case of mixed Dirichlet-homogeneous Neumann conditions
\begin{align}\label{qq14}
\Gamma U(0)= (\rho(0),u'(0),v'(0))=(\rho_0,0,0)=g
\end{align}
for the inflow case.    Again, $u(z)\equiv u_\infty$ and $v(z)\equiv
u_\infty$, since $u(0)$ and $v(0)$ are rest points of \eqref{3} and
\eqref{5} respectively.  The mass equation then implies
$\rho(z)\equiv\rho_\infty$, so in particular $\rho_\infty=\rho_0$.
Therefore, a candidate state $U_\infty$ can belong to
$\cC_{\Gamma,g}$ only if it is the endpoint of a constant layer, and
we have
\begin{align}\label{qq15}
\cC_{\Gamma,g }=\{U_\infty\in
\bR^3:v_\infty>0,\rho_\infty=\rho_0\},\;\;\dim \cC_{\Gamma,g}=2.
\end{align}

Consider the cases
\begin{align}\label{qq16}
v_\infty-c_\infty\gtrless 0,
\end{align}
and recall that $N_b=3=N_++N^2_-$. In these cases $N_+=3$, $N^2_-=0$
and $N_+=2$, $N^2_-=1$ respectively, so in both cases $N''=2>N^2_-$.
Corollary \ref{smalltrans} implies that these constant layers are
never transversal; hence they fail to satisfy even low frequency
Evans stability.  Observe that in both cases $\cC_{\Gamma,g}$ fails
to have the ``correct" dimension $N-N_+$.

\bigskip

\subsubsection{Maximal dissipativity/uniform Lopatinski condition.
}\label{maxlop} The question of uniform Evans stability for large
amplitude layers remains a mostly open question. For the
one-dimensional isentropic case with $\gamma$-law equation of state,
it has been shown numerically for Dirichlet boundary conditions that
noncharacteristic boundary layers are stable, independent of
amplitude \cite{CHNZ}. For the full, nonisentropic case on the other
hand, it has been shown for Dirichlet boundary conditions
that, even in one dimension and for $\gamma$-law equation of state,
instabilities may occur \cite{SZ}. Here we show that the residual
boundary conditions determined in the previous subsection are
maximally dissipative (and thus satisfy the uniform Lopatinski
condition) for all amplitudes. When $\Gamma$ is a full Dirichlet
condition ($N''=0$), this conclusion follows without any
computation, since the
residual boundary conditions by the analysis of the previous
subsection have form independent of the amplitude of the boundary
layer, and for small
amplitudes are known (Theorem \ref{realmaxdiss}) to be maximally
dissipative.
We carry out the computations nonetheless, to show how they work out
in this simplest case. See \cite{SZ} for
 one-dimensional calculations involving more general,  not necessarily
 decoupled, boundary conditions.

\medbreak

\textbf{1. Dirichlet conditions, outflow, $v_\infty +c_\infty >
0>v_\infty$.}  The hyperbolic problem is symmetric, with one
incoming, and two  outgoing characteristics. Rewriting the
isentropic Euler equations in $U:=(\rho, u, v)$ coordinates and
linearizing about $(\rho_\infty, u_\infty, v_\infty)$, we obtain
$U_t+A^1U_x+A^2U_y$, where
$$
A_2=
\begin{pmatrix}
v_\infty & 0 & \rho_\infty\\
0 & v_\infty & 0\\
c_\infty^2/\rho_\infty & 0 & v_\infty \\
\end{pmatrix},
$$
while the residual boundary condition (linear, in this case) is
$$
\Gamma_{res} U=g, \qquad \Gamma=
\begin{pmatrix}
0 & 0 & 1
\end{pmatrix},
$$
hence the kernel of the linearized boundary condition is the set of
vectors $(a,b,0)$. Applying the positive definite symmetrizer $
S=\diag\{c_\infty^2/\rho_\infty^2, 1, 1 \}$, we obtain
\begin{align}\label{qq20}
SA_2=
\begin{pmatrix}
v_\infty c_\infty^2/\rho_\infty^2 & 0 & c_\infty^2/\rho_\infty\\
0 & v_\infty & 0\\
c_\infty^2/\rho_\infty & 0 & v_\infty \\
\end{pmatrix},
\end{align}
which is evidently negative definite on the kernel of
$\Gamma_{res}$; hence $\Gamma_{res}$ is maximally dissipative.

\medbreak

\textbf{2. Dirichlet conditions, outflow, $v_\infty +c_\infty<0$.}
In this case $SA_2$ is negative definite, so the residual boundary
conditions are automatically {\it maximally dissipative}. \medbreak

\textbf{3. Dirichlet conditions, inflow, $v_\infty-c_\infty >0$.}
Again, this is a trivial case, in which all hyperbolic modes are
incoming, and the residual boundary conditions are
$\Gamma_{res}U=U$. So $\ker \Gamma_{res}=\{0\}$, and thus
$\Gamma_{res}$ is  maximally dissipative.

\medbreak

\textbf{4. Dirichlet conditions, inflow, $v_\infty-c_\infty <
0<v_\infty$.} In this final case, we have prescription of momentum
$m_\infty$ and transverse velocity $u_\infty$, and the kernel of the
linearized boundary condition is spanned by
$w:=(\rho_\infty,0,-v_\infty)$. Computing
$$
w^TSA_2w= v_\infty(v_\infty^2-c_\infty^2)<0,
$$
we find that the restriction of $SA_2$ to the kernel is again
negative definite, so that the residual boundary condition is indeed
maximally dissipative in agreement with the abstract theory.

\begin{rem}
1.  We remark on the contrast between specification of velocity vs.
momentum in outgoing vs. incoming case. (Specifying $\rho$ or $v$
along with $u$ in the incoming case wouldn't be dissipative in
general.)

2. \textbf{Neumann or mixed conditions.} The only  case in section
\ref{Isentropic} involving Neumann conditions that was Evans stable
was the outflow case when $v_\infty+c_\infty<0$. Again, $SA_2$ is
negative definite so the residual conditions are maximally
dissipative.   In the remaining three cases involving Neumann or
mixed conditions, maximal dissipativity fails. For example, in the
inflow case ($v_\infty>0$) $\Gamma_{res}=(1,0,0)$, and so $SA_2$ is
positive definite on $\ker\Gamma_{res}$.

\end{rem}

\subsubsection{Transversality of large amplitude layers}

We now show how to verify that transversality holds in the three
cases where we constructed large amplitude layers:\medbreak

(a) outflow/Dirichlet/$v_\infty+c_\infty>0$;

(b) outflow/Dirichlet/$v_\infty+c_\infty<0$;

(c) inflow/Dirichlet/$v_\infty-c_\infty<0$. \medbreak

  It is helpful first to reformulate the definition of
transversality geometrically.  Let $w(z)$ be a possibly large
amplitude layer profile converging to
$q_\infty:=(\rho_\infty,u_\infty,v_\infty)$.  Set
\begin{align}\label{xx1}
W_0:=(w(0),w^2_z(0)),\;W_\infty=(q_\infty,0)
\end{align}
and observe that $W_0\in\bR^5$ lies in both the stable manifold of
$W_\infty$, denoted $\cW^s(W_\infty)$, and the center-stable
manifold of $W_\infty$, denoted $\cW^{cs}(W_\infty)$.   The
conditions defining transversality of $w(z)$ in Definition
\ref{deftrans} can be rephrased:\medbreak

\quad (i)\;$\Gamma:T_{W_0}\left(\cW^{s}(W_\infty)\right)\to\bR^{N_b}$ is
injective,

\quad (ii)\;$\Gamma:T_{W_0}\left(\cW^{cs}(W_\infty)\right)\to\bR^{N_b}$ is
surjective.

These may be recognized as the conditions that $\Upsilon$ be
full rank on the stable and center--stable manifolds, respectively,
of $W_\infty$, which correspond by definition to the following
geometric version of transversality.

\begin{lem}[Geometric transversality conditions]\label{geo}
The transversality conditions of Definition \ref{deftrans}
are equivalent to the conditions that

(i')\; the level set $\{W:\, \Upsilon(W)=\Upsilon(W_0)\}$ meets the
stable manifold of $W_\infty$ transversally in phase space $W=(w,
w^2_z)$.

(ii')\; the level set $\{W:\, \Upsilon(W)=\Upsilon(W_0)\}$ meets the
center-stable manifold of $W_\infty$ transversally in phase space
$W=(w, w^2_z)$.
\end{lem}

\begin{rem}The use of the word ``transversal" in (i') is perhaps nonstandard; it is used to suggest a minimal intersection.  We mean $\{W: \Upsilon(W)=\Upsilon(W_0)\}\cap
\cW^s(W_\infty)=\{W_0\}$; that is, the intersection is a single
point.
\end{rem}

%
%

The stable and center-stable manifolds are easily parametrized in
each of cases (a), (b), and (c) above, so conditions (i) and (ii)
can be checked explicitly.  For example, consider case (c), where
$N_b=N_++N^2_-=2+1=3$, the viscous boundary condition is
\begin{align}\label{xx2}
\Gamma U(0)= (\rho_0,u_0,v_0):=U_0
\end{align}
with $U_0$ fixed, $w(0)=U_0$, $w(\infty)=q_\infty$ is fixed, and the
dimensions of the stable and center-stable manifolds above are $1$
and $4$ respectively.   Recalling the induced hyperbolic boundary
conditions
\begin{align}\label{xx5}
u_\infty=u_0,\;\; \rho_\infty v_\infty=\rho_0v_0,
\end{align}
we see that those manifolds can be parametrized as follows:
\begin{align}\label{xx3}
\begin{split}
&\cW^s(W_\infty)=\Big\{\begin{pmatrix}\frac{\rho_\infty v_\infty}{a}\\u_\infty\\a\\0\\f(a,\rho_\infty,v_\infty)\end{pmatrix}: a \in A\Big\} \\
&\cW^{cs}(W_\infty)=\Big\{\begin{pmatrix}\frac{b
d}{a}\\c\\a\\0\\f(a,b,d)\end{pmatrix}: a \in A\subset \bR,
(b,c,d)\in B\subset\bR^3\Big\}
\end{split}
\end{align}
where $B$ is a neighborhood of $q_\infty$, $A$ is a neighborhood of
$v_0$ determined by \eqref{rr1}, and the function $f$ (whose form
turns out not to be important for verifying transversality) can be
read off from \eqref{5}.   Independent vectors spanning the tangent
spaces to $\cW_s(W_\infty)$ (resp. $\cW^{cs}(W_\infty)$) at $W_0$
may be computed by differentiating the  formulas \eqref{xx3} with
respect to $a$ (resp. $a$, $b$, $c$, $d$). With these explicit
formulas the injectivity and surjectivity conditions above are
immediately obvious.

We summarize this discussion in the following
\begin{prop}\label{ac1}
The isentropic Navier-Stokes large-amplitude layer profiles
described in cases (a), (b), (c) at the beginning of this section
are transversal.

\end{prop}

\subsection{Full Navier-Stokes equations}\label{full}
Next we consider the full (nonisentropic) Navier--Stokes equations
\begin{align}
\begin{split}
&(a)\;\rho_t+ (\rho u)_x +(\rho v)_y=0,\\
&(b)\;(\rho u)_t+ (\rho u^2)_x + (\rho uv)_y + p_x= (2\mu +\eta)
u_{xx}+
\mu u_{yy} +(\mu+ \eta )v_{xy},\\
&(c)\;(\rho v)_t+ (\rho uv)_x + (\rho v^2)_y + p_y=
\mu v_{xx}+ (2\mu +\eta) v_{yy} + (\mu+\eta) u_{yx},\\
&(d)\;(\rho E)_t+ (u \rho E)_x + (v \rho E)_y+ (pu)_x + (pv)_y=\\
&\qquad \kappa T_{xx} + \kappa T_{yy}+\Big( (2\mu+\eta)uu_x + \mu
v(v_x+u_y) + \eta uv_y\Big)_x +\\
&\qquad\qquad\qquad\;\;\quad\quad \Big((2\mu+\eta)vv_y + \mu
u(v_x+u_y) + \eta vu_x\Big)_y
\end{split}
\end{align}
where $\rho$ is density, $u$ and $v$ are velocities in the $x$ and
$y$ directions, $p$ is pressure,
and $e$ and $E=e+\frac{u^2}{2} +\frac{v^2}{2}$ are specific internal
and total energy respectively. The constants $\mu>|\eta|\ge0$ and
$\kappa>0$ are coefficients of first (``dynamic'') and second
viscosity and heat conductivity. Finally, $T$ is the temperature and
we assume that the internal energy $e$ and the pressure $p$ are
known functions of density and temperature:
\begin{equation*}
p=p(\rho,T),\quad e=e(\rho,T).
\end{equation*}

It is clear for full (nonisentropic) gas dynamics that the uniform
Lopatinski condition does \emph{not} hold for arbitrary amplitudes,
even for the simplest case of a $\gamma$-law gas. The proof of
one-dimensional ($\eta=0$) viscous instability in \cite{SZ} for
compressive boundary-layers of such a gas \footnote{ Specifically,
for $\gamma>2$
and $2\mu+\eta>\kappa$, in the inflow case
$v_\infty>0>v_\infty-c_\infty$ with full Dirichlet conditions
imposed at the boundary, instability was shown for layers consisting
of a sufficiently large portion of a sufficiently large-amplitude
$1$-shock.} was based on showing that the Lopatinski determinant (a
multiple of $\lambda$) restricted to the positive real axis could
change sign as parameters, including amplitude, were varied; in
particular, the determinant could vanish, yielding hyperbolic
instability.  So we cannot hope to show that the residual boundary
conditions independent of amplitude as before.  A brief examination
reveals, likewise, that the residual boundary conditions are rather
complicated to describe: in short, there appears to be no hope of
other than a {\it local} analysis near the limiting constant layer.
We carry this out below, and compute the linearized boundary
condition for the linearized hyperbolic problem at a constant layer.

The profile equations are, setting $\nu := 2 \mu + \eta$:
\begin{eqnarray}
(\rho v)'  & = & 0 \nonumber\\
(\rho u v )' & = & \mu u'' \nonumber\\
(\rho v^2)' + p' & = & \nu v'' \nonumber\\
\big( \rho v (e + \frac{u^2 + v^2}{2} ) + pv \big)'  & = & \kappa
T'' + \frac{1}{2}(\mu u^2 + \nu v^2)''\nonumber
\end{eqnarray}
Let us introduce the unknown $m=\rho v$. Writing the equations in
terms of the unknowns $(m, u,v,T)$ and integrating once gives the
following equations with $p= P(v,T)$ and $p_\infty = P(v_\infty,
T_\infty)$:
\begin{eqnarray}\label{profeq}
(a)\;\;m  & = & m_\infty \\
(b)\;\;u' & = & \frac{m_\infty}{\mu}(u-u_\infty) \nonumber\\
(c)\;\;v'  & = &  \frac{m_\infty}{\nu}(v-v_\infty ) + \frac{p - p_\infty}{\nu} \nonumber \\
(d)\;\;T ' & = & \frac{m_\infty}{\kappa}(e-e_\infty)
- \frac{m_\infty}{2\kappa}(u-u_\infty)^2 \nonumber \\
& &  - \frac{m_\infty}{2\kappa}(v-v_\infty)^2 +
\frac{p_\infty}{\kappa} (v-v_\infty) \nonumber
\end{eqnarray}
The equation for $X =(u,v,T)$ reads $X' = F(X; m_\infty,
X_\infty)$ where $X_\infty$ and $m_\infty$ are parameters.\\

\textbf{Case of an outgoing flow.}  Let $U=(m,u,v,T)=(m,X)$ and
consider an outgoing flow, $v< 0$, so $N_b=N'+N^1_+=3+0=3$. We fix a
state $U_0\in\cU_\partial$ for which $v_0<0$ consider boundary
conditions that are just Dirichlet conditions on $X$:
\begin{align}\label{tt1}
\Gamma U(0)=X(0)=X_0.
\end{align}
As usual we define $\cC_{\Gamma,X_0}$ as the set of states
$(m_\infty,X_\infty)$ such that there exists a solution $(m,X)$ of
the equations (\ref{profeq}) on $[0, +\infty[$, such that $X(0) =
X_0$ and $U(\infty)  = ( m_\infty , X_\infty ) $.  Clearly, $U_0\in
\cC_{\Gamma,X_0}$, and near $U_0$ we can write $\cC_{\Gamma,X_0}$ as
the union
\begin{equation}
\cC_{\Gamma,X_0} = \cup_{m_\infty}\cC_{m_\infty,X_0}
\end{equation}
where $\cC_{m_\infty,X_0}$ is the set of $X_\infty \in\bR^3$ such
that there exists a solution to equations \ref{profeq} (b-c) on $[0,
+\infty[$  satisfying $X(0) = X_0$ and $X(\infty) = X_\infty$.
Corollary \ref{smalltrans} shows that the constant layer $U(z)=U_0$
is transversal, and thus $\cC_{\Gamma,X_0}$ is  a $\cC^\infty$
manifold near $U_0$. By the argument of \cite{Se}, Lemma 15.2.5 (or
alternatively, by an argument similar to our proof of Lemma
\ref{a64}), the tangent space $T_{U_0} \cC_{\Gamma,X_0}$ is given by
\begin{align}\label{tt10}
T_{U_0} \cC_{\Gamma,X_0} = \bR\times \EE^-\big( D_{X}F(X_0;m_0,X_0)
\big).
\end{align}

Hence we are lead to calculate the stable invariant subspace of $D_X
F(X_0; m_0,X_0)$. One finds
\begin{equation}
D_XF(X;m,X) = \left(
\begin{array}{ccc}
 \frac{m}{\mu}  & 0 & 0 \\
\\
0 & { \frac{1}{\nu} } (m+ P'_v) &
 { \frac{1}{\nu} } P'_{ T } \\
\\
0  &  \frac{1}{\kappa}(P_\infty+ m  e'_v) & \frac{m}{\kappa} e'_T
\end{array}
\right)
\end{equation}
Since $m_0< 0$, the matrix $D_XF(X_0;m_0,X_0)$ has at least one
negative eigenvalue which is $m_0 / \mu$. It has exactly two
eigenvalues with negative real part (in fact real negative
eigenvalues) if and only if
\begin{equation} \label{detextra}
\mathrm{det} \left[
\begin{array}{cc}
 { \frac{1}{\nu} } (m+ P'_{ v } ) &
 { \frac{1}{\nu} } P'_{ T} \\
\\
\frac{1}{\kappa}(P_\infty+ m e'_v) & \frac{m}{\kappa} e'_T
\end{array}
\right] \ < \ 0.
\end{equation}
In that case $\EE^-\big( D_XF(X_0;m_0,X_0) \big)$ has dimension $2$
and $T_{U_0 }\cC_{\Gamma,X_0}$ has dimension 3. Let us call
$\lambda_- $ the second negative eigenvalue of $D_XF(X_0;m_0,X_0)$.
Then, the tangent space to $\cC_{\Gamma,X_0}$ at the point
$(m_0,X_0)$ is defined by the equations:
\begin{equation} \label{resequwithtemp}
(\dot{m},\dot{u},\dot{v},\dot{T}) \in T_{(m_0,X_0) }\cC_{\Gamma,X_0}
\iff
\begin{array}{c}
\left(\frac{1}{\nu}  (m_0+ P'_{ v })-\lambda_- )\right)\dot{v}  +
 \frac{1}{\nu}P'_T \dot{T} = 0  .
\end{array}
\end{equation}
We do not know whether these boundary conditions have already
appeared in the theory of the Euler equations, or if they have a
special physical meaning. As we already know from the general
theory, they are maximally dissipative for the Euler equations.

\subsection{MHD equations}\label{MHDsec}

The equations of isentropic magnetohydrodynamics (MHD) are
\begin{equation}
\label{mhdeq} \left\{ \begin{aligned}
 & \D_t \rho +  \div (\rho u) = 0
 \\
 &\D_t(\rho  u) + \div(\rho u^tu)+ \na p + H \times \curl H =
\eps \mu \Delta u + \eps(\mu+\eta) \nabla \div u
 \\
 &
 \D_t H + \curl (H \times u) = \eps \sigma \Delta H
 \end{aligned}\right.
\end{equation}
\begin{equation}\label{divfree}
\div H=0,
\end{equation}
where $\rho\in \RR$ represents density, $u\in \RR^3$ fluid velocity,
$p=p(\rho)\in \RR$ pressure, and $H\in \RR^3$ magnetic field, with
viscosities $\mu>|\eta|\ge0$
and magnetic resistivity $\sigma \ge 0$; {\it for simplicity, we
consider here the case $\sigma=0$.} When  $H\equiv 0$ \eqref{mhdeq}
reduces to the equations of isentropic fluid dynamics.

The equations \eqref{mhdeq} are not yet in a form that satisfies the
assumptions of section \ref{equations}.  For example, the
noncharacteristic condition is violated for every state
$(\rho,u,H)$. The equations may be put in conservative form using
identity
\begin{equation}
H\times \curl H= (1/2)\div (|H|^2I-2H^tH)^\trans + H\div H
\end{equation}
together with constraint \eqref{divfree} to express the second
equation as
\begin{equation}\label{cons2}
 \D_t(\rho  u) + \div(\rho u^tu)+ \na p + (1/2)\div (|H|^2I-2H^tH)^\trans= \eps \mu \Delta u + \eps(\mu+\eta) \nabla \div u.
\end{equation}
They may be put in symmetrizable (but no longer conservative) form
by a further change, using identity
\begin{equation}
\curl (H \times u) =
 (\div u) H+ (u\cdot \na)H -(\div H) u- (H\cdot \na)u
\end{equation}
together with constraint \eqref{divfree} to express the third
equation as
\begin{equation}\label{symm3}
 \D_t H +
 (\div u) H+ (u\cdot \na)H - (H\cdot \na)u= \sigma \eps \Delta H.
\end{equation}
Forgetting  the constraint equation, we get a $7 \times 7$ symmetric
system that satisfies all our structural assumptions except for
(H4).

\begin{rem}\label{yy1}
1) Define
\begin{align}\label{yy2}
c^2 = p'(\rho) > 0 , \quad  v = H/ \sqrt{\rho},  \quad  b= | \hat
\xi \times v |,  \quad
  \hat \xi = \xi/ | \xi |,
\end{align}
$$
  \begin{aligned}
  &c_f^2  := \frac{1}{2}
  \Big( c^2 + |v |^2   + \sqrt{ (c^2 - | v |^2)^2 + 4 b^2 c^2  }
  \Big),
  \\
 &   c_s^2  := \frac{1}{2}
  \Big( c^2 + | v| ^2   -  \sqrt{ (c^2 - | v | ^2)^2 +  4 b^2  c^2   }
  \Big).
  \\
\end{aligned}
$$

The boundary $x_3=0$ is noncharacteristic for the hyperbolic part
when
\begin{align}\label{yy3}
u_3\notin\{0,\pm v_3,\pm c_s(n),\pm c_f(n)\},
\end{align}
where $c_s(n)$ and $c_f(n)$ are the slow and fast speeds computed in
the normal direction $n=(0,0,1)$.  Lemma 8.2 of \cite{GMWZ6} shows
that if we assume in addition
\begin{align}\label{yy4}
0<|v|\neq c,\;\; |u_3|>|v_3|,
\end{align}
then Hypothesis (H4$'$) is satisfied.   For the stability analysis,
which requires the construction of $K$-families of symmetrizers
(\cite{GMWZ6}, Definition 3.5), we must use \eqref{cons2},
\eqref{symm3}.   For the purposes of deriving profile equations,
describing $\cC$ manifolds, and computing linearized residual
boundary conditions, the form \eqref{mhdeq}, \eqref{divfree} is more
convenient to use.

2)  As the divergence-free condition is preserved by the evolution
of the equations, we are free to ignore it in establishing stability
of solutions.  Results on instability, however, must be examined to
check whether associated unstable modes are true, divergence-free
instabilities or only apparent, spurious instabilities.

\end{rem}

 With $\nu := 2
\mu + \eta$ the profile equations are:
\begin{eqnarray}\label{MHDprofequ}
(\rho u_3)'  & = & 0 \\
(\rho u_1 u_3 - H_1 H_3)'  & = & \mu u_1'' \nonumber\\
(\rho u_2 u_3 - H_2 H_3)' & = & \mu u_2'' \nonumber\\
\big( \rho u_3^2 + p +  {\frac{1}{2}}(H_1^2 + H_2^2 - H_3^2) \big)'
& = &
\nu  u_3'' \nonumber\\
(H_1 u_3 - u_1 H_3)' & = & 0 \nonumber\\
(H_2 u_3 - u_2 H_3)' & = & 0 \nonumber\\
H_3' & = & 0\nonumber
\end{eqnarray}
By choosing the set of new unknowns $m := \rho u_3$, $\alpha := (H_1
u_3 - u_1 H_3)$, $\beta := (H_2 u_3 - u_2 H_3)$, $H_3$, $u_1$,
$u_2$, $u_3$ and integrating once, the system reduces to a system on
$u=(u_1, u_2, u_3)$, where $M:=(m , \alpha, \beta , H_3) \in \RR^4$
and  $u_\infty = (u_{1\infty}, u_{2\infty}, u_{3\infty}) \in \RR^3$
are parameters:
\begin{eqnarray}\label{yy8}
 \; M&= &M_\infty\nonumber\\
 \;u_1' & = &  \frac{m}{\mu}\big( u_1 - u_{1\infty} \big)
- \frac{\alpha H_3}{\mu} \big(  \frac{1}{u_3} -
\frac{1}{u_{3\infty}}\big) - \frac{H_3^2}{\mu}
\big(  \frac{u_1}{u_3} -  \frac{u_{1\infty}}{u_{3\infty}} \big)\\
 u_2'  & = &  \frac{m}{\mu}\big( u_2 - u_{2\infty } \big)
- \frac{\beta H_3}{\mu} \big( \frac{1}{u_3} -
\frac{1}{u_{3\infty}}\big) - \frac{H_3^2}{\mu}
\big(  \frac{u_2}{u_3} -  \frac{u_{2\infty}}{u_{3\infty}} \big)\nonumber\\
u_3' & = &  \frac{m}{\nu}\big(  u_3 - u_{3\infty} \big) +
\frac{1}{\nu} \big( p(\frac{m}{u_3}) - p(\frac{m}{u_{3\infty}})
\big) +
\nonumber\\
& + & \frac{1}{2\nu} \big( H_1^2  -  H_{1\infty}^2 \big)
 +  \frac{1}{2\nu} \big( H_2^2 - H_{2\infty}^2 \big)
-  \frac{1}{2\nu} \big( H_3^2 - H_{3\infty}^2 \big)\nonumber.
\end{eqnarray}
In the last equation it is understood that $H_1 = \frac{\alpha + u_1
H_3}{u_3}$ and $H_2 = \frac{\beta + u_2 H_3}{u_3}$. As in the
calculations of section \ref{full}, the system has the form $u' =
F(u;M,u_\infty)$. The Jacobian matrix is
\begin{equation*}
D_uF( u;M,u) = \left[
\begin{array}{ccc}
a & 0 & e/\mu \\
0 & a & d/\mu \\
e/\mu & d/\nu & b
\end{array}
\right],
\end{equation*}
where $M:=(m,\alpha,\beta,H_3)$ and
\begin{equation}\label{notationsMHD}
\begin{array}{lcl}
a := \frac{m}{\mu} - \frac{H_3^2}{\mu u_3} & \quad  & b:=
\frac{m}{\nu} \big( 1- \frac{P'(\rho)}{u_3^2} \big)
- \frac{H_1^2 + H_2^2}{ \nu u_3} \\
\\
e := \frac{H_1 H_3}{u_3} & \quad  & d := \frac{H_2 H_3}{u_3}.
\end{array}
\end{equation}
The eigenvalues of $D_uF(u;M,u)$ are
\begin{equation}
a\text{ and }\lambda_\pm:= \frac{a+b}{2} \pm \sqrt { \
\frac{(a-b)^2}{4} + \frac{e^2 + d^2 }{\mu \nu} },
\end{equation}
and we have
\begin{equation}\label{prod}
\lambda_+ \lambda_-  =  ab - (e^2 + d^2)/ \mu\nu.
\end{equation}

\textbf{Case of an outgoing flow.}  Let
$U=(m,\alpha,\beta,H_3,u)=(M,u)$ and fix a state
$U_0=(M_0,u_0)\in\cU_\partial$ which satisfies the conditions
\eqref{yy3} and \eqref{yy4} and for which the flow is outgoing,
$u_3< 0$, so $N_b=N'+N^1_+=3+0=3$. We consider boundary conditions
that are just Dirichlet conditions on $u$:
\begin{align}\label{yy9}
\Gamma U(0)=u(0)=u_0.
\end{align}
As usual we define $\cC_{\Gamma,u_0}$ as the set of states
$(M_\infty,u_\infty)$ such that there exists a solution $(M,u)$ of
the equations \eqref{yy8} on $[0, +\infty[$, such that $u(0) = u_0$
and $U(\infty)  = ( M_\infty , u_\infty ) $.  Clearly, $U_0\in
\cC_{\Gamma,u_0}$, and near $U_0$ we can write $\cC_{\Gamma,u_0}$ as
the union
\begin{equation}
\cC_{\Gamma,u_0} = \cup_{M_\infty}\cC_{M_\infty,u_0},
\end{equation}
where $\cC_{M_\infty,u_0}$ is the set of $u_\infty \in\bR^3$ such
that there exists a solution to the last three equations in
\eqref{yy8} on $[0, +\infty[$ satisfying $u(0) = u_0$ and $u(\infty)
= u_\infty$.   Corollary \ref{smalltrans} shows that the constant
layer $U(z)=U_0$ is transversal,  and thus $\cC_{\Gamma,X_0}$ is  a
$\cC^\infty$ manifold near $U_0$.  By the argument of section
\ref{full},
 the tangent space $T_{U_0} \cC_{\Gamma,u_0}$ is given by
\begin{align}\label{yy11}
T_{U_0} \cC_{\Gamma,u_0} = \bR^4\times \EE^-\big(
D_{u}F(u_0;M_0,X_0) \big).
\end{align}
We proceed to compute the stable subspace of
$\cF_0:=D_{u}F(u_0;M_0,X_0)$.

The condition \eqref{yy4} implies that $\cF_0$ has at least one
negative eigenvalue, namely $a$, with corresponding eigenvector
$(-\frac{d}{\nu},\frac{e}{\mu},0)$. The matrix $\cF_0$ will have
exactly one, two, or three negative eigenvalues depending on whether
\begin{align}\label{yy12}
\lambda_->0,\;\lambda_-<0<\lambda_+,\;\text{ or }\lambda_+<0
\end{align}
respectively.   From \eqref{yy11} we see that the corresponding
dimensions of $\cC_{\Gamma,u_0}$ are $5$, $6$, and $7$ respectively.
By Proposition \ref{C22prop} the dimension of $\cC_{\Gamma,u_0}$ is
$N-N_+$, where $N_+$ is the  number of positive eigenvalues of
$\oA_3$.  The eigenvalues of $\oA_3$ are
\begin{align}\label{yy13}
\lambda_0  =  u_3 ,\; \lambda_{\pm s}  = u_3 \pm c_s(n),\;
 \lambda_{\pm 2} = u_3  \pm  v_3   ,\; \lambda_{\pm f} =
u_3 \pm c_f(n),
\end{align}
so we observe that the cases \eqref{yy12} correspond  to the cases
\begin{align}\label{yy14}
|u_3|<c_s(n),\;c_s(n)<|u_3|<c_f(n),\;c_f(n)<|u_3|
\end{align}
respectively. For example, consider the case
$\lambda_-<0<\lambda_+$, which by
formula (\ref{prod}) occurs if and only if
\begin{align}\label{yy15}
ab-\frac{e^2+d^2}{\mu\nu}<0.
\end{align}
%
A short computation now shows that the stable subspace of $\cF_0$ is
spanned by
\begin{align}\label{yy16}
\begin{split}
&\{(-\frac{d}{\nu},\frac{e}{\mu},0),(\frac{e}{\mu},\frac{d}{\mu},\lambda_--a)\},\\
&\text{ where
}\lambda_--a=\frac{b-a}{2}-\sqrt{\frac{(a-b)^2}{4}+\frac{e^2+d^2}{\mu\nu}}
\end{split}
\end{align}
The linearized residual boundary condition is thus expressed by
\begin{align}\label{yy17}
(\dot M,\dot u)\in T_{U_0}\cC_{\Gamma,u_0}\Leftrightarrow \dot
u\cdot
\left(\frac{e}{\mu}(a-\lambda_-),\frac{d}{\nu}(a-\lambda_-),\frac{d^2}{\mu\nu}+\frac{e^2}{\mu^2}\right)=0,
\end{align}
where $a$,$b$,$e$,$d$ are evaluated at $U_0$.

\begin{rem}\label{gj}
The hypotheses and conclusions of Theorem \ref{mainresult} on Evans
stability and the existence of small viscosity limits apply to small
amplitude layers for all the physical examples considered in section
5 except

a)isentropic NS/inflow ($v_0>0$)/mixed Dirichlet-Neumann boundary
conditions \eqref{qq14}

b)isentropic NS/outflow ($v_0<0$)/$v_\infty+c_\infty>0$/Neumann
boundary conditions \eqref{qq8}.

Additional examples for the full Navier Stokes and viscous MHD
equations where Theorem \ref{mainresult} applies can be deduced from
Corollary \ref{symmstab}(b).

\end{rem}

\bigbreak
\appendix

\section{Construction of Approximate solutions}\label{approximate}

In this appendix, we give the construction of approximate solutions,
following an approach similar to that used in \cite{GG,GMWZ7}. A new
feature here is that we obtain global approximate solutions on a
domain $\Omega$ with compact closure and smooth boundary.   The same
construction works on unbounded domains whose boundary coincides
with a half-space outside a compact set.

We seek high-order approximate solutions to

\begin{align}\label{B1}
\begin{split}
&(a)\cL_\eps(u):=
 A_0(u)u_t  + \sum_{j=1}^d   A_j(u) \D_{j} u    -
 \eps \sum_{j,k= 1}^d \D_{j} \big( B_{jk}(u) \D_{k} u \big) = 0,\\
&(b)\Up(u,\partial_Tu^2,\partial_\nu u^2)=(g_1,g_2,0)\text{ on
}\partial\Omega
\end{split}
\end{align}
which converge to a given solution $u^0(t,x)$ of the inviscid
hyperbolic problem:
\begin{align}\label{B2}
\begin{split}
&\cL_0(u^0)=0\text{ on }[-T_0,T_0]\times\Omega\\
&u^0(t,x_0)\in \cC(t,x_0)\text{ for }(t,x_0)\in
[-T_0,T_0]\times\partial\Omega,
\end{split}
\end{align}
where $\cC(t,x_0)$ is the endstate manifold defined in Assumption
\ref{C5} (see also Prop. \ref{C30}). Using the cutoff $\chi(x)$ and
the normal coordinates $(x_0,z)$ in a collar neighborhood of
$\partial \Omega$ defined in section \ref{multi-scale}, we look for
an approximate solution of the form
\begin{align}\label{B3}
\begin{split}
&u_a(t,x)=\sum_{0\le j\le M} \eps^j\cU^j(t,x, \frac{z}{\epsilon})+\epsilon^{M+1}u^{M+1}(t,x),\\
&\mathcal{U}^j(t,x,\frac{z}{\epsilon})=
\chi(x)V^j(t,x_0,\frac{z}{\epsilon}) + u^j(t,x).
\end{split}
\end{align}
Here $u^0$ satisfies\eqref{B2} and $V^0$  is given by
\begin{align}\label{B4}
V^0(t,x_0,Z)=W(Z,t,x_0,u^0(t,x_0))-u^0(t,x_0),
\end{align}
for a profile $W(Z,t,x_0,u^0(t,x_0))$ as in Assumption \ref{C5}.

The $V^j_\pm(Z,x_0,t)$ are boundary layer profiles constructed to be
exponentially decreasing to $0$ as $Z\to \pm\infty$.    For the
moment we just assume enough regularity so that all the operations
involved in the construction make sense.   A precise statement is
given in Prop. \ref{esoln}.

\subsection{Profile equations}\label{profile}

We substitute \eqref{B3} into \eqref{B1} and write the result as
\begin{align}\label{B5}
\sum^M_{-1}\epsilon^j\mathcal{F}^j(t,x,Z)|_{Z=\frac{z}{\epsilon}}+\epsilon^M
R^{\epsilon,M}(t,x),
\end{align}
where we separate $\cF^j$ into slow and fast parts
\begin{align}\label{B10}
\mathcal{F}^j(t,x,Z)=F^j(t,x)+G^j(t,x_0,Z),
\end{align}
and the $G^j$ decrease exponentially to $0$ as $Z\to\pm\infty$.

The interior profile equations are obtained by setting the $F^j,G^j$
equal to zero.  In the following expressions for $G^j(t,x_0,Z)$, the
functions $u^j(t,x)$ and their derivatives are evaluated at
$(t,x_0)$. With $W=W(Z,t,x_0,u^0(t,x_0))$ set
\begin{align}\label{B11}
\begin{split}
&\bL(t,x_0,Z,\partial_Z)v:=A_{\nu(x_0)}(W)v_Z+(d_uA_\nu(W)\cdot v)W_Z-\\
&\quad\frac{d}{dZ}\left(B_\nu(W)v_Z\right)-\frac{d}{dZ}\left(
(v\cdot d_u B_\nu(W))W_Z \right),
\end{split}
\end{align}
the operator determined by the linearizing the profile equations
about $W$, and
\begin{align}\label{B11a}
\cL_0v:=
 A_0(u^0)v_t  + \sum_{j=1}^d   A_j(u^0) \D_{j} v.
\end{align}
We have
\begin{align}\label{B12}
\begin{split}
&F^{-1}(t,x)=0\\
&G^{-1}(t,x_0,Z)=A_\nu(W)W_Z-\frac{d}{dZ}(B_\nu(W)W_Z),
\end{split}
\end{align}
\begin{align}\label{B13}
\begin{split}
&F^0(t,x)=\cL_0u^0\\
&G^0(t,x_0,Z)=\bL(t,x_0,Z,\partial_Z)\cU^1-Q^0(t,x_0,Z),
\end{split}
\end{align}
where $Q^0$ decays exponentially as $Z\to+\infty$ and depends only
on $(u^0,V^0)$. For $j\geq 1$ we have
\begin{align}\label{B14}
\begin{split}
&F^j(t,x)=\cL_0u^j-P^{j-1}(t,x)\\
&G^j(t,x_0,Z)=\bL(t,x_0,Z,\partial_Z)\cU^{j+1}-Q^j(t,x_0,Z),
\end{split}
\end{align}
where $Q^j$ decays exponentially as $Z\to+\infty$ and $P^j$, $Q^j$
depend only on $(u^k,V^k)$ for $k\leq j$.   Note that dependence on
the cutoff $\chi(x)$ occurs only in the $P^j$.

In writing out the boundary profile equations, we note first that
the boundary condition \eqref{B1}(b) is equivalent for $\epsilon>0$
to
\begin{align}\label{B15}
\Up(u,\epsilon\partial_Tu^2,\epsilon\partial_\nu u^2)=(g_1,g_2,0).
\end{align}
With $\cU^j(t,x,Z)=(\cU^{j,1},\cU^{j,2})$ always evaluated at
$(t,x_0,0)$ and
\begin{align}\label{B16}
\Up'(\cU^0)(v,0,v^2_Z):=\begin{pmatrix}\Up_1'(\cU^0)v^1\\\Up_2'(\cU^0)v^2\\K_N\partial_Zv^2\end{pmatrix},
\end{align}
the boundary profile equations at order $\epsilon^j$ take the form:
\begin{align}\label{B17}
\Up(\cU^0,0,\partial_Z\cU^{0,2})=(g_1,g_2,0) \quad(\text{order }
\epsilon^0),
\end{align}
\begin{align}\label{B18}
\Up'(\cU^0)(\cU^1,0,\partial_Z\cU^{1,2})=(0,0,c_{1,3}(t,x_0)) \quad
(\text{order }\epsilon^1),
\end{align}
\begin{align}\label{B19}
\Up'(\cU^0)(\cU^j,0,\partial_Z\cU^{j,2})=(c_{j,1},c_{j,2},c_{j,3})
\quad (\text{order }\epsilon^j, j\geq 2),
\end{align}
where the $c_{j,k}(t,x_0)$ depend just on the $\cU^p$ and their
first derivatives for $p\leq j-1$.

\subsection{Solution of the profile
equations}\label{solutionprofile}

The solution of the profile equations given below assumes
transversality of $W(Z,u^0(t,x_0))$ and the uniform Lopatinski
condition, as well as the existence of a $K$-family of smooth
inviscid symmetrizers. Recall from Lemma \ref{Rousset} that the
first two conditions both follow from the low frequency uniform
Evans condition.

\textbf{1.} The interior equations $G^{-1}=0$ and $F^0=0$ and the
boundary equation \eqref{B17} are satisfied because of our
assumptions about $u^0$  and $W(Z,t,x_0,u^0(t,x_0))$.

\textbf{2. Construction of $(\cU^1,u^1)$.}     We construct the
functions $\cU^1(t,x,Z)$ and $u^1(t,x)$ from the equations $G^0=0$,
$F^1=0$, and the boundary equation \eqref{B18}.  $\cU^1$ will be a
sum of three parts
\begin{align}\label{B20}
\begin{split}
&\cU^1(t,x,Z)=\cU^1_a+\cU^1_b+\cU^1_c,\text{ where }\\
&\cU^1_k(t,x,Z)=u^1_k(t,x)+V^1_k(t,x_0,Z),\;k=a,b,c.
\end{split}
\end{align}

First use the exponential decay of $Q^0$ to find an exponentially
decaying solution $V^1_a(t,x_0,Z)$ to
\begin{align}\label{B21}
\begin{split}
&\bL(t,x_0,Z,\partial_Z)V^1_a=Q^0(t,x_0,Z)\text{ on }\pm Z\geq 0\\
&V^1_a\to 0\text{ as }Z\to +\infty,
\end{split}
\end{align}
and define $u^1_a(t,x)\equiv 0$.   This problem is easily solved
after first conjugating to a constant coefficient ODE using the
operators $P$ defined in Lemma \ref{conjugation}.

Next, for $\cU^1_a$ fixed as above,  use part (ii) of the definition
of transversality (Definition \ref{deftrans})
to see that we can solve for $\cU^1_b(t,x_0,Z)\in \cal S$ satisfying
\begin{align}\label{B22}
\begin{split}
&\bL(t,x_0,Z,\partial_Z)\cU^1_b=0\text{ on } Z\geq 0\\
&\Up'(\cU^0)\left(\cU^1_a+\cU^1_b,0,\partial_Z(\cU^{1,2}_a+\cU^{1,2}_b)\right)=(0,0,c_{1,3}(t,x_0)).
\end{split}
\end{align}
Recalling the definition of $\cS$ from Lemma \ref{count}, we see
that $\cU^1_b$ has limits as $Z\to\infty$. Define
\begin{align}\label{B23}
\begin{split}
&u^1_b(t,x_0):=\lim_{Z\to \infty}\cU^1_b(t,x_0,Z), \\
&V^1_b(t,x_0,Z):=\cU^1_b(t,x_0,Z)-u^1_b(t,x_0),
\end{split}
\end{align}
and let $u^1_b(t,x)$ be any smooth extension of $u^1_b(t,x_0)$ to
$[-T_0,T_0]\times\Omega$.

Finally, for an appropriate choice of $u^1_c(t,x_0)$ we need
$\cU^1_c(t,x_0,Z)$ to satisfy
\begin{align}\label{B24}
\begin{split}
&\bL(t,x_0,Z,\partial_Z)\cU^1_c=0\\
&\Up'(\cU^0)(\cU^1_c,0,\partial_Z\cU^{1,2}_c)=0\\
&\lim_{z\to+\infty}\cU^1_c(t,x_0,Z)=u^1_{c}(t,x_0).
\end{split}
\end{align}
According to the characterization of $T_{q}\cC(t,x_0)$ given in
Remark \ref{vvv1}, this is possible if and only if $u^1_c(t,x_0)\in
T_{u^0(t,x_0)}\cC(t,x_0)$.   Thus, we first solve for $u^1_c(t,x)$
satisfying the linearized inviscid problem
\begin{align}\label{B25}
\begin{split}
&\cL_0u^1_c=P^0-\cL_0u^1_b \\
&u^1_c(t,x_0)\in T_{u^0(t,x_0)}\cC(t,x_0).
\end{split}
\end{align}

This problem requires an initial condition in order to be
well-posed.  The right side in the interior equation of \eqref{B25}
is initially defined just for $t\in [-T_0,T_0]$. With a $C^\infty$
cutoff that is identically one in $t\geq -T_0/2$, we can modify the
right side to be zero in $t\leq -T_0+\delta$, say. Requiring $u^1_c$
to be identically zero in $t\leq -T_0+\delta$, we thereby obtain a
problem for $u^1_c$ that is forward well-posed since $u^0$ satisfies
the uniform Lopatinski condition. Thus, there exists a solution to
\eqref{B25} on $[-\frac{T_0}{2},T_0]$. This allows us to obtain
$\cU^1_c(t,x_0,Z))$ satisfying \eqref{B24} and to define
\begin{align}\label{B26}
V^1_c(t,x_0,Z):=\cU^1_c(t,x_0,Z)-u^1_c(t,x_0).
\end{align}

By construction the functions  $(\cU^1,u^1)$ satisfy the equations
$G^0=0$, $F^1=0$, and the boundary conditions \eqref{B18}.

\textbf{3. Contruction of $(\cU^j,u^j)$, $j\geq 2$.}  In the same
way, for $j\geq 2$ we use the equations $G^{j-1}=0$, $F^j=0$, and
the boundary conditions \eqref{B19} to determine the functions
$(\cU^j,u^j)$.   The corrector $\epsilon^{M+1}u^{M+1}$ is chosen
simply to solve away an $O(\epsilon^{M+1})$ error that remains in
the boundary conditions after the construction of $\cU^M$.

In the next Proposition we formulate a precise statement summarizing
the construction of this section.  The regularity assertions in the
Proposition are justified as in \cite{GMWZ4}, Prop. 5.7.  Regularity
is expressed in terms of the following spaces:
\begin{defi}\label{espaces}
1.  Let $H^s$ (resp. $H^s_b$) be the standard Sobolev space on
$[-T_0,T_0]\times\Omega$ (resp. $[-T_0,T_0]\times\partial\Omega$).

2. Let $\tilde{H}^s$ be the set of functions $V(t,x_0,Z)$ on
$[-T_0,T_0]\times\partial\Omega\times \overline{\bR}_+$ such that
$V\in
C^\infty(\overline{\bR}_+,H^s([-T_0,T_0]\times\partial\Omega))$ and
satisfies
\begin{align}\label{B27}
|\partial_Z^kV(t,x_0,Z)|_{H^s_b}\leq C_{k,s}e^{-\delta|Z|}\text{ for
all }k
\end{align}
for some $\delta>0$.

\end{defi}

\begin{prop}[Approximate solutions]\label{esoln}
Assume (H1)-(H6) (with (H4$'$) replacing (H4) in the
symmetric-dissipative case).   For given integers $m\geq 0$ and
$M\geq 1$ let
\begin{align}\label{B28}
s_0>m+\frac{7}{2}+2M+\frac{d+1}{2}.
\end{align}
Suppose that the inviscid solution $u^0$ as in \eqref{B2} satisfies
the uniform Lopatinski condition and that the profiles
$W(Z,u^0(t,x_0))$ are transversal.
Assume $u^0\in H^{s_0}$ and  $u^0|_{\partial\Omega}\in H^{s_0}_b$.
Then one can construct $u_a$ as in \eqref{B3} satisfying:
\begin{align}\label{B29}
\begin{split}
&\cL_\epsilon u_a=\epsilon^M R^M(t,x)\text{ on }[-\frac{T_0}{2},T_0]\times\Omega\\
&\Up(u_a,\partial_Tu^2_a,\partial_\nu u^2_a)=(g_1,g_2,0)\text{ on
}\partial\Omega.
\end{split}
\end{align}
We have
\begin{align}\label{B30}
u^j(t,x)\in H^{s_0-2j},\; V^j(t,x_0,Z)\in\tilde{H}^{s_0-2j},
\end{align}
and $R^M(t,x)$ satisfies
\begin{align}\label{B31}
\begin{split}
&(a)\;|(\partial_t,\partial_{x_0},\epsilon\partial_{z})^\alpha
R^M|_{L^2}\leq
C_\alpha\text{ for }|\alpha|\leq m+\frac{d+1}{2}\\
&(b)\;|(\partial_t,\partial_{x_0},\epsilon\partial_{z})^\alpha
R^M|_{L^\infty}\leq C_\alpha\text{ for }|\alpha|\leq m.
\end{split}
\end{align}
In a collar neighborhood of $\partial\Omega$, $\partial_{x_0}$
denotes an arbitrary vector field tangent to $\partial\Omega$ . Away
from such a neighborhood, $\partial_{x_0}$ can be a completely
arbitrary vector field.
\end{prop}

\section{The Tracking Lemma and construction of symmetrizers}\label{trackapp}

In this appendix, we discuss further the high-frequency
analysis of Section \ref{HF}, in particular completing
the proof of Theorem \ref{aa3} by a treatment of the
remaining (much easier) elliptic case.
Recall that Theorem \ref{aa3}, together with its easy consequences
Corollary \ref{hfcrit} and Proposition \ref{symmhf},
allows us to reduce to considering only a
compact set of frequencies in the stability analysis of many kinds
of layers.   We applied this result in showing, for example, that
Evans stability of small amplitude layers follows from Evans
stability of the limiting constant layer.

In the process we prove a useful and previously unremarked relation
between the tracking lemma  of \cite{ZH,MaZ3,PZ} and the
construction of high-frequency symmetrizers as in
\cite{MZ1,GMWZ3,GMWZ4,GMWZ6}.
In particular, we state the result (used implicitly in \cite{GMWZ6})
that existence of a $k$-family of symmetrizers implies continuity
of decaying subspaces in the high-frequency limit (tracking).
A similar argument was used in \cite{MZ3} to establish continuity
of subspaces in the low-frequency limit.

\subsection{Abstract setting}
Consider a generalized resolvent ODE
\begin{equation}\label{geneq}
U'-\check \cG(z,p, \eps)U=F,
\end{equation}
where $'$ denotes $\partial_z$ and $(p,\eps)$ comprise frequencies
and model parameters, in the limit as $\eps\to 0$. We are interested
in the situation, as in the high-frequency regime, that $\check \cG$
is ``slowly varying'' in the sense that $\check \cG'$ is small in an
appropriate sense compared to $\check \cG$.

The following proposition gives one version of this notion
in a simple case.

\begin{prop}\label{bb1}
(a) Assume the $M\times M$ matrix $\check\cG(z,p,\eps)$ is a
$C^\infty$ function of
\begin{align}\label{bb5}
(z,p,\eps)\in [0,\infty)\times P\times (0,1]
\end{align}
for some parameter set $P$, and that $|\check\cG|\geq c_1>0$ for
$c_1$ independent of $(z,p,\eps)$. Suppose
\begin{align}\label{bb6}
\frac{|\partial_z\check \cG|}{|\check \cG|^2}\to 0 \text{ as }
\eps\to 0 \text{ uniformly with respect to }(z,p),
\end{align}
and suppose
\begin{align}\label{bb7}
\frac{\gap(\check \cG)}{|\check \cG|}\ge c_0>0 \text{ for }c_0
\text{ independent of }(z,p,\eps).
\end{align}
Here, $\gap(M)$ denotes the spectral gap of a matrix $M$, defined as
the minimum absolute value of the real parts of the eigenvalues of
$M$. There exists a conjugating transformation $T(\check \cG)$
taking $\check \cG$ to block-diagonal form
\begin{equation}\label{simple2}
T\check \cG T^{-1}= \begin{pmatrix} M_+ & 0\\ 0 & M_-\end{pmatrix},
\quad \Re M_+>\eta(z,p,\eps) \quad \hbox{\rm and} \quad \Re
M_-<-\eta(z,p,\eps),
\end{equation}
where $\eta:=\gap(\check \cG)\ge c_0 |\check \cG|$.   For fixed
$c_0>0$ as in \eqref{bb7} the matrix $T(\check\cG(z,p,\eps))$
satisfies
\begin{align}\label{bb4}
(a)\;|T|\leq C,\;\;\;(b)\;|T^{-1}|\leq C,\;\;\;(c)\;|T_z|\leq C
|\check\cG_z|/|\check\cG|,
\end{align}
uniformly with respect to $(z,p,\eps)$.

(b) Setting $U=T^{-1}V$ in  \eqref{geneq}, we obtain
\begin{equation}\label{blockdiag}
\begin{aligned}
V'- \tilde \cG(z,p,\eps)V =\tilde F,\\
\tilde \cG  := \tilde \cG_{\rm p}+
\Theta(z,p,\eps),\\
\tilde \cG_{\rm p}:=
\begin{pmatrix} M_+ & 0\\ 0 & M_-\end{pmatrix}, \quad
|\Theta| \le \delta,
\end{aligned}
\end{equation}
where $\delta(z,p,\eps)=\delta:= C |\D_z \check \cG|/| \check\cG|$,
and so
\begin{align}\label{bb2}
\delta/\eta\to 0 \text{ as }\eps\to 0 \text{ uniformly with respect
to }(z,p).
\end{align}

\end{prop}

\begin{proof}

\textbf{(a). } Consider the matrix $\cK:=\check\cG/|\check\cG|$
which has a spectral gap uniformly bounded away from $0$ by
assumption. It follows by standard matrix perturbation theory
\cite{Kat} that there exists a smooth transformation $T(\cK)$ taking
$ \cK$ to block-diagonal form
\begin{equation}\label{simple}
T \cK T^{-1}= \begin{pmatrix} N_+ & 0\\ 0 & N_-\end{pmatrix}, \quad
\Re N_+>c_0 \quad \hbox{\rm and} \quad \Re N_-<-c_0
\end{equation}
and satisfying \eqref{bb4}(a), (b).  Clearly the same $T$ satisfies
\eqref{simple2}.   Moreover, with obvious notation we have
\begin{align}\label{bb9}
|T_z|\leq | dT/d\cK|\;|d\cK/d\check\cG|\; |\check\cG_z|\leq
C|\check\cG_z|/|\check\cG|.
\end{align}

\textbf{(b). } Defining $V$ by $U=T^{-1}V$, substituting in
\eqref{geneq}, and using  \eqref{simple} yields \eqref{blockdiag}
with
\begin{align}\label{bb10}
\tilde F=TF, \;\;\Theta=-T(T^{-1})_z,
\end{align}
so by \eqref{bb4} $|\Theta|\leq C|\check\cG_z|/|\check\cG|$. Thus,
\eqref{bb2} follows from the assumption \eqref{bb6}.

\end{proof}

For example, in the application of Proposition \ref{bb1} to the
``elliptic zone" in section \ref{ell}, we will have
\begin{align}\label{bb11}
\eps=|\zeta|^{-1},\; |\check\cG|\sim |\zeta|,\; |\check\cG_z|\sim
|\zeta|, \text{ and so }  \delta\sim 1,\;\; \eta\sim |\zeta|.
\end{align}
Depending on the application, $\eta$ may be bounded, go to zero, or
go to infinity as $\eps\to 0$, and likewise $|\check \cG|$. In some
sense this is artificial, since the scale of $\check \cG$ (and thus
of $\eta$) may be changed arbitrarily by rescaling the independent
variable $z$; however, it is convenient not to have to rescale.

More generally, following \cite{MaZ3}, we take the existence of a
transformation to form \eqref{blockdiag}--\eqref{bb2} as {\it defining}
the notion of a slowly-varying coefficient $\cG$.
The construction of such transformations must in general be done
quite carefully by hand, and does not follow by a simple argument
like that of Proposition \ref{bb1}; see for example the treatment
in Section 7, \cite{GMWZ6}, and here in the proof of Theorem \ref {aa3}
of the general case away from the elliptic zone.
Proposition \ref{bb1} suffices to treat strictly parabolic systems,
as pointed out in \cite{GZ,ZH,Z3}.

\subsection{Tracking}
For systems \eqref{geneq} as above, we have the following version of
the ``tracking lemma" of \cite{MaZ3}. The lemma implies that in the
modified coordinates of \eqref{blockdiag}, under assumption \eqref{bb2},
the subspace of initial data at $z_0\in \bR$  of decaying (resp.
growing) solutions of the homogeneous equation approximately
``tracks'' the stable (resp. unstable) subspace of the
principal part $\tilde \cG_{\rm p}=\blockdiag\{M_+,M_-\}$ evaluated
at $z_0$
in the sense that one subspace approaches the other uniformly as
$\eps\to 0$.

In view of \eqref{bb2} we will often suppress the dependence of
$\delta$ and $\eta$ on $(z,p)$ in what follows and simply write
$\delta(\eps)$, $\eta(\eps)$.

\begin{prop}[\cite{MaZ3}] \label{reduction}
Consider an approximately diagonalized system
\eqref{blockdiag} with $\tilde F\equiv 0$ satisfying bound \eqref{bb2}.

(i) For all $0<\epsilon\le \epsilon_0$, there exist (unique) linear
transformations $\Phi_1^\epsilon(z,p)$ and $\Phi_2^\epsilon(z,p)$,
with $C^\infty$ dependence on $(z,p,\eps)$ for which the graphs
\begin{align}\notag
\{(Z_1, \Phi^\epsilon_2(z,p) Z_1)\} \text{ and }
\{(\Phi^\epsilon_1(z,p) Z_2 ,Z_2)\}
\end{align}
are invariant under the flow of \eqref{blockdiag}. The graphs
consist precisely of the initial data at $z=z_0$ of solutions of
\eqref{blockdiag} that are respectively exponentially growing and
decaying. Moreover, the functions $\Phi^\eps_j$ satisfy
$$ |\Phi^\epsilon_1|, \, |\Phi^\epsilon_2| \le C
\delta(\epsilon)/\eta(\epsilon) \, \text{\rm for all } z.
$$

(ii)In particular, the subspace $E_-(p,\eps)$ of data at $z=0$ for
which the solution of \eqref{blockdiag} decays as $z\to +\infty$ is
given by the graph $\{(\Phi^\eps_1(0,p)v_-, v_-):v_-\in\bC^{dim
M_-}\}$ and converges as $ \eps \to 0$ to $\tilde E_-:= \{(0,
v_-):v_-\in\bC^{dim M_-}\}$.
\end{prop}

\medskip
{\bf Proof.} This can be proved by a contraction mapping argument
carried out on the ``lifted'' equations governing the flow of the
conjugating matrices $\Phi^\eps_j$; see Appendix C, \cite{MaZ3} for
details.  Proposition \ref{cc10} provides an alternative proof of part(ii),
which is the only part we use in this paper,
in the case that $\tilde \cG$ exponentially approaches a limit
as $z\to +\infty$ for each fixed $p$, $\eps$ (not necessarily
uniformly), which holds always in our applications.
A related proof based on energy estimates/invariant cones
appears in \cite{ZH,Z1}.
\medskip

\subsection{Symmetrizers}

With the same initial preparations (i.e., reduction to form
\eqref{blockdiag}), one may also obtain directly bounds on the
inhomogeneous resolvent equation by the method of Kreiss
symmetrizers \cite{K,MZ1}. Consider \eqref{geneq} on $[0,+\infty)$,
augmented with some specified boundary condition
\begin{equation}\label{eBC}
\Gamma(p, \eps)U=G, \text{ where }|\Gamma(p,\eps)|\leq C,
\end{equation}
uniformly for $p\in P$, $\eps\in (0,1]$. The corresponding boundary
condition for the conjugated problem \eqref{blockdiag} is then
\begin{align}\label{cc1}
\tilde\Gamma(p,\eps)V:=\Gamma(p,\eps)T^{-1}(0,p,\eps)V=G.
\end{align}
Defining $\tilde E_-(p,\eps)$ as the stable subspace of the
principal part $\blockdiag\{M_+,M_-\}$ evaluated at $z_0=0$ (i.e.,
$\tilde E_-:=\{(0,v):v\in\bC^{dim M_-}\}$) and defining the
``frozen-coefficients Evans function''
\begin{equation}\label{tildeD}
\tilde D(p,\eps):=\det( \tilde E_-, \ker \tilde\Gamma),
\end{equation}
 we have the following  result.

\begin{prop}\label{gensymm}
Consider the problem \eqref{geneq}, \eqref{eBC} under the
assumptions of Proposition \ref{reduction}. Assume this problem is
 ``frozen-coefficients stable'' in the sense that
$|\tilde D|\ge c_0>0$ for $\eps>0$ sufficiently small, uniformly
with respect to $p\in P$. Then for some $C>0$
that can be taken independent of $p\in P$ and $\eps$ sufficiently
small,
\begin{equation}\label{genbds}
\sqrt{\eta} \|U\|_{L^2(\RR^+)} + |U(0)|\le C(\|\tilde
F\|_{L^2(\RR^+)}/\sqrt{\eta} + |G|).
\end{equation}
\end{prop}

\begin{proof}
\textbf{1. }In view of the properties \eqref{simple2}, \eqref{bb4}
of the conjugating transformation $T$, it suffices to prove the
estimate \eqref{genbds} for $V$ satisfying the conjugated problem
\eqref{blockdiag} and the boundary condition \eqref{cc1}.

\textbf{2. }Defining $S_k:=\blockdiag\{k\Id, -\Id\}$, $k\ge 1$, we
have for $\eps>0$ sufficiently small,
$$
\begin{aligned}
&(i) S_k=S_k^*,\quad  |S_k|\le k;\\
&(ii) \Re S_k \tilde \cG_p \ge \eta(\eps);\\
&(iii) S_k v\cdot v \ge k| \Pi_+v|^2-| \Pi_-v|^2,\\
\end{aligned}
$$
Here $ \Pi_\pm$ are the projections onto the unstable and stable
subspaces of $\tilde \cG_p(0,p,\eps)$.   Writing $v=(v_+,v_-)$ we
have
\begin{align}\label{cc2}
\Pi_+v=(v_+,0), \;\;\Pi_-=(0,v_-).
\end{align}
Thus $S_k$ is a $k$-family of symmetrizers in the sense of
\cite{GMWZ6}.

Taking the real part of the $L^2$ inner product of $-S_k V$ with
\eqref{blockdiag} and integrating by parts, we obtain
$$
\begin{aligned}
(1/2)S_k V(0)\cdot V(0) + \langle V, \Re (S_k\tilde \cG)V\rangle &=
\Re\langle -S_k V, \Theta V\rangle +
\Re\langle -S_k V, \tilde F\rangle \\
&\le |S_k|( \delta \|V\|^2 + \|V \| \|\tilde F\|).
\end{aligned}
$$
>From this we may deduce using (i)--(iii) and $\delta/\eta\to 0$
that, for $\eps>0$ sufficiently small,
\begin{equation}\label{semifinal}
k| \Pi_+ V(0)|^2 + (\eta/2)\|V\|^2 \le k \|V \| \|\tilde F\| + |
\Pi_- V(0)|^2.
\end{equation}
But $|\det (\tilde E_-,\ker \tilde\Gamma)|\ge c_0$ implies that
$|\tilde\Gamma v|\ge c|v|$ for $v\in \tilde E_-$, where
$c=c(c_0,|\tilde\Gamma|)>0$. Thus,
$$
\begin{aligned}
|\Pi_- V(0)|\le |\tilde\Gamma  \Pi_-V(0)| &\le |\tilde\Gamma
V(0)|+|\tilde\Gamma \Pi_+ V(0)|
\leq |G| + |\tilde\Gamma| | \Pi_+ V(0)|.\\
\end{aligned}
$$
Combining with \eqref{semifinal}, taking $k$ sufficiently large
relative to $|\tilde \Gamma|$, and using Young's inequality to bound
$k\|\tilde F\| \|V\|\le (\eta/4)\|V\|^2 + (|k|^2/\eta)\|\tilde
F\|^2$,
 we obtain the result.
\end{proof}

The next Proposition (more precisely, its proof)
formulates the new observation that the
conclusion of part (ii) of Proposition \ref{reduction} is a
consequence of the estimate \eqref{genbds}; that is, information
provided by the tracking lemma may be deduced as a consequence of
the basic symmetrizer construction.

\begin{prop}\label{cc10}
Consider an $M\times M$ system \eqref{geneq} satisfying the
assumptions of Proposition \ref{reduction}.   Let $\bE_-(p,\eps)$ denote
the subspace of initial data at $z=0$ for which the solution of
\begin{align}\label{cc4}
U'-\check\cG(z,p,\eps)U=0
\end{align}
decays to $0$ as $z\to\infty$.  Let $\tilde \bE_-(p,\eps)$ denote
the stable subspace of $\check\cG(0,p,\eps)$.  Assuming
\begin{align}\label{cc5}
\dim \bE_-(p,\eps)=\dim \tilde \bE_-(p,\eps)\text{ for }p\in P,
\;\eps\text{ small}.
\end{align}
we have
\begin{align}\label{cc6}
\bE_-(p,\eps)\to \tilde \bE_-(p,\eps)\text{ as }\eps\to 0.
\end{align}

\end{prop}

\begin{rem}\label{ccc6}
Condition \eqref{cc5} holds, for example, if
$\tilde \cG$ exponentially approaches a limit
as $z\to +\infty$ for each fixed $p$, $\eps$ (not necessarily
as $z\to +\infty$, by the conjugation lemma,
Lemma \ref{conjugation}.
This is always the case for our applications here.
\end{rem}

\begin{proof}

\textbf{1. } Let $E_-(p,\eps)$ and $\tilde E_-(p,\eps)$ be the
spaces appearing in Proposition \ref{reduction}.  They are the exact
analogues of $\bE_-$ and $\tilde \bE_-$ for the conjugated system
$V'-\tilde\cG V=0$.  From \eqref{simple2} and the properties of the
conjugator $T$ we have
\begin{align}\label{cc7}
\begin{split}
& \bE_-(p,\eps)=T^{-1}(0,p,\eps)E_-(p,\eps)\\
&\tilde \bE_-(p,\eps)=T^{-1}(0,p,\eps)\tilde E_-(p,\eps),
\end{split}
\end{align}
so it is equivalent to show
\begin{align}\label{cc8}
E_-(p,\eps)\to \tilde E_-(p,\eps)=\{(0,v_-):v_-\in\bC^{dim
M_-}\}\text{ as }\eps\to 0.
\end{align}

\textbf{2. }The proof depends on the fact that the estimate
\eqref{genbds} holds for \emph{any} boundary condition satisfying
the hypotheses of Proposition \ref{gensymm}.   Let $W\subset\bC^M$
be any subspace \emph{transverse} to the fixed space $\tilde
E_-(p,\eps)$ and such that $\dim W^\perp=\dim\tilde E_-(p,\eps)$.
Defining $\Gamma_W$ to be orthogonal projection onto $W^\perp$, we
have
\begin{align}\label{cc11}
\ker\Gamma_W=W,\;\; \rank \;\Gamma_W = \dim\tilde E_-(p,\eps).
\end{align}
Since $|\det(\tilde E_-, \ker \Gamma_W)|\ge c_0>0$, the  hypotheses
of Proposition \ref{gensymm} are satisfied by the conjugated system
\eqref{blockdiag} with the boundary condition $\Gamma_W$. Assuming
that $V$ is any decaying solution of $V'-\tilde\cG V=0$, we have
$V(0)\in E_-(p,\eps)$.  From the estimate \eqref{genbds} for the
conjugated problem with $\tilde F=0$ we deduce
\begin{align}\label{cc12}
|V(0)|\leq C|\Gamma_W V(0)|\text{ for }\eps\text{ sufficiently
small}.
\end{align}
By \eqref{cc5} and Lemma \ref{ZZZ} this implies
\begin{align}\label{cc13}
|\det(E_-(p,\eps),\ker\Gamma_W|\geq c>0,
\end{align}
where $c$ depends on $C$ and $|\Gamma_W|=1$. Therefore, $W$ is also
transverse to $E_-(p,\eps)$ for $\eps$ sufficiently small.  Since we
are free to make different choices of $W$ tranverse to $\tilde
E(p,\eps)$, this can only be true if \eqref{cc6} holds.

\end{proof}

As an immediate consequence of Proposition \ref{gensymm} and
Proposition \ref{reduction}, part (ii), we obtain
\begin{cor}\label{fullsymm}
Suppose \eqref{geneq} satisfies the assumptions of Proposition
\ref{reduction}, and let
\begin{align}\label{cc15}
\bD(p,\eps):=\det (\bE_-, \ker \Gamma).
\end{align}
Then \eqref{geneq}, \eqref{eBC} is ``uniformly Evans
stable'' in the sense that $|\bD|\ge c_0>0$ for $\eps>0$
sufficiently small if and only if \eqref{geneq}, \eqref{eBC} is
``frozen coefficients" stable in the sense that
\begin{align}\label{cc14}
\tilde\bD(p,\eps):=\det (\tilde\bE_-, \ker \Gamma)\geq c_1>0
\end{align}
for $\eps$ sufficiently small.   In either case the estimate
\eqref{genbds} holds.

\end{cor}

\begin{rem}
In this presentation, we have subsumed large parts of the
usual symmetrizer construction into the preparatory transformations
to the approximately block-diagonal form \eqref{blockdiag}.
One of the main points of this development is to demonstrate
that {\it tracking and symmetrizer estimates are essentially automatic
once we can reduce a system to form \eqref{blockdiag}--\eqref{bb2}}.
Moreover, all high-frequency estimates derived in \cite{ZH,Z3,GMWZ4,GMWZ6}
were either obtained originally in this way, or may be rephrased in this form.
\end{rem}

\begin{rem}\label{dd15}
Another consequence of Corollary \ref{fullsymm} is that uniform
high-frequency Evans stability implies the maximal estimate
\eqref{maxesthf}
for $\zeta\in\cE$.  Indeed, the estimate \eqref{genbds} is
equivalent to \eqref{maxesthf} in this application.
\end{rem}

\subsection{Application to High Frequencies: Proof of Theorem \ref{aa3}}\label{apps}

We conclude by illustrating through explicit computations the
application of these methods to the high-frequency analysis of
Section \ref{HF}.

Recall from \eqref{linq3} the linearized eigenvalue equation $Lu=f$,
where
\begin{align}
L=-\cB(z)\partial^2_z+\cA(z,\zeta)\partial_z+\cM(z,\zeta)
\end{align}
with coefficients given by
\begin{equation}
\label{coefflin}
\begin{aligned}
\cB (z) &= B_{dd} (w(z))\\
  \cA (z, \zeta) &= A_d(w(z) ) -  \sum_{j=1}^{d-1} i\eta_j \big( B_{j d} +  B_{d, j} \big)(w(z))   +
   E_d (z)
   \\
   \cM(z, \zeta) &=  (i \tau + \gamma) A_0(w(z) ) +  \sum_{j=1}^{d-1} i
    \eta_j  \big( A_j(w(z)) + E_j(z) \big)
   \\
   & \qquad   \qquad \qquad \qquad   \quad +  \sum_{j, k = 1}^{d-1} \eta_j \eta_k B_{j, k}(w(z) )   +   E_0
   (z).
   \end{aligned}
\end{equation}
The $E_k $ are functions independent of $\zeta$ which involve
derivatives of $w$ and thus converge to $0$ at an exponential rate
when $z$ tends to infinity. Moreover, we note that
\begin{equation}
\label{Einfty}
 E^{11}_k = 0, \quad     E^{12}_k = 0 \quad   \mathrm{ for }\  k > 0.
\end{equation}
With \eqref{struc1}, we also remark that $\cM^{12}$ does not depend
on $\tau$ and $\gamma$.

As in \eqref{linq4} we rewrite the linearized problem as a
first-order system
 \begin{equation}
\label{Blinq4} \D_z U  - \cG(z, \zeta) U = F  , \quad   \Gamma
(\zeta)  U _{| z = 0} = G,
\end{equation}
\begin{equation}\label{cG}
\cG  = \begin{pmatrix}
     \cG^{11} & \cG^{12}  & \cG^{13}    \\  0 & 0&  \Id \\
      \cG^{31} & \cG^{32}  & \cG^{33}
\end{pmatrix},
\end{equation}
where $U =  ( u , \D_z u^2)  = (u^1, u^2, \D_z u^2)   \in \CC^{N+
N'}$ and $\zeta=(\gamma,\tau, \eta)$.

Here,
\begin{eqnarray*}
\cG^{11} = -  (\cA^{11})^{-1} \cM^{11}, && \cG^{31} =
(\cB^{22})^{-1} ( \cA^{21}\cG^{11} + \cM^{21})  ,
\\  \cG^{12} = -  (\cA^{11})^{-1} \cM^{12}, &
 &
\cG^{32} =   (\cB^{22})^{-1} ( \cA^{21}\cG^{12} + \cM^{22}),
\\
 \cG^{13} = -  (\cA^{11})^{-1} \cA^{12},  &&
\cG^{33} =   (\cB^{22})^{-1} ( \cA^{21}\cG^{13} + \cA^{22}),
\end{eqnarray*}

Note that $\cG^{11}$, $\cG^{12}$, $\cG^{31}$ and $\cG^{33}$ are
first order (linear or affine  in $\zeta$), that  $\cG^{32}$ is
second order (at most quadratic in $\zeta$) and that $\cG^{13} $ is
of order zero (independent of $\zeta$). We denote by
$\cG^{ab}_{\mathrm{p}} $ their principal part (leading order part as
polynomials). We note that
\begin{equation}
\label{princcG} \cG^{ab}_{\rmp} (z, \zeta) =  G^{ab}_\rmp (w(z),
\zeta)
 \quad  \mathrm{when} \  (a, b) \ne (3, 1) ,
\end{equation}
with
\begin{equation*}
\begin{array}{lll}
G^{11}_{\mathrm{p}}(u, \zeta)  & = -  (A_d^{11}(u))^{-1}  \big((
\gamma+ i\tau) A_0^{11}(u) +
\sum_{j=1}^{d-1} i \eta_j A_j^{11}(u) \big),  \\
G^{12}_{\mathrm{p}}(u, \zeta)  & = -  (A_d^{11}(u))^{-1}
\sum_{j=1}^{d-1} i \eta_j A_j^{12}(u) \\
G^{13}_{\mathrm{p}}(u)  & = -  (A_d^{11}(u))^{-1}    A_d^{12}(u)
\\
G^{32}_{\mathrm{p}}(u, \zeta)  & =   (B^{22}(u))^{-1}
\sum_{j, k =1}^{d-1}   \eta_j  \eta_k B_{j,k}^{22}(u)   \big), \\
G^{33}_{\mathrm{p}}(u, \zeta)  & =  -  (B^{22}(u))^{-1}
\sum_{j=1}^{d-1} i \eta_j \big(B_{j, d}^{22}(u)  + B^{22}_{d,j} (u)
\big) .
 \end{array}
\end{equation*}
The principal term of $\cG^{3, 1}$ involves derivatives of the
profile $w$. Denoting by $p = \lim_{z \to + \infty}  w(z) =
w(\infty)$ the end state of the profile $w$, we note that the end
state of $\cG^{31}_\rmp$ is
$$
\cG^{31}_{\mathrm{p}}(\infty, \zeta)   =  (B^{22}(p))^{-1} \big((
\gamma+ i\tau) A_0^{21}(p) + \sum_{j=1}^{d-1} i \eta_j A_j^{21}(p)
+ A^{21}_d(p) G^{11}_\rmp(p, \zeta) \big).
$$
There are similar formulas using  the matrices $\overline A_j$ and
$\overline B_{j, k}$ of \eqref{y2}.

\subsubsection{The elliptic zone}\label{ell}

We now prove Theorem \ref{aa3} in the easiest case, which is for
frequencies $\zeta$ lying in the  {\it elliptic zone}
\begin{equation}\label{ellipticzone}
\cE:=\{  (\tau,\gamma,\eta): \gamma \ge \delta  | \zeta |\
\mathrm{and} \ | \eta | \ge \delta | \zeta | \}\ \mathrm{with} \
\delta > 0,
\end{equation}
with $|\zeta|$ sufficiently large.
(Recall that this is the final remaining case not treated in
Section \ref{HF}.)

In this case, the factors
$(1+\gamma)$ and $
 \Lambda(\zeta)  = \big(  \tau^2 + \gamma^2 + | \eta |^4 \big) ^{1/4}
$ appearing in scaling \eqref{reschf} are both of order $|\zeta|$,
and so we may replace \eqref{reschf} by the simpler rescaling
\begin{equation}
\label{Bscale}
\begin{aligned}
J_\zeta (u^1, u^2, u^3) &  := \big(|\zeta|^{\frac{1}{2}}   u^1,
 |\zeta|^\mez  u^2 ,   |\zeta|^{-\mez} u^3 \big):=\check U  \\
 J_\zeta (g^1, g^2, g^3) & := \big( |\zeta|^{\frac{1}{2}}   g^1,
 |\zeta|^\mez  g^2 ,   |\zeta|^{-\mez} g^3 \big)  :=\check G.
 \end{aligned}
\end{equation}
If we define
\begin{align}\label{dd1}
\Gamma^{sc}_e(u^1,u^2,u^3):=(\Gamma_1u^1,\Gamma_2u^2,
K_du^3+|\zeta|^{-1}K_T(\eta)u^2),
\end{align}
then \eqref{Blinq4} may be written equivalently as
\begin{align}\label{dd1a}
\D_z \check U  - \check \cG(z, \zeta) \check U = \check F  , \quad
\Gamma^{sc}_e(\zeta)\check U=\check G
\end{align}
where
\begin{equation}
\label{BdectG} \check \cG =
\begin{pmatrix}
     \cG^{11} &   \cG^{12}  &  | \zeta |  \cG^{13}    \\  0 & 0&  | \zeta | \Id \\
       | \zeta |^{-1}  \cG^{31} &  | \zeta |^{-1}  \cG^{32}  & \cG^{33}
\end{pmatrix} := \begin{pmatrix}
   \cG^{11}    &    \cP^{12}\\
   \cP^{21}   &   \cP^{22}
\end{pmatrix}
\end{equation}
with obvious definitions of $\cP^{a b}$. Note that $\check \cG$ is
of order one, while $\cP^{21}$ is of order zero. Thus
\begin{equation}
\label{dectG1} \check \cG (z, \zeta)   = \check \cG_{\mathrm{p}} (z,
\zeta)   + O(1) , \quad \check \cG_{\mathrm{p}}= \begin{pmatrix}
   \cG_{\mathrm{p}}^{11}    &  \cP^{12}_\rmp  \\    0
      &     \cP^{22}_{\mathrm{p}}
\end{pmatrix}  = O(| \zeta |),
\end{equation}
where the principal part $\check \cG_{\mathrm{p}}$ is in fact
homogeneous degree one in $|\zeta|$, and depends on $z$ only through
the profile $w(z)$ and not its derivatives.

Consider now the Fourier-Laplace transform of the upper triangular
system
\begin{align}\label{trilinprinc}
\begin{split}
&u^1_t + \sum_j \overline A_j^{11}(w(0)) \D_j u^1
+\sum_j \overline A_j^{12}(w(0)) \D_j u^2=0,\\
&u^2_t - \sum_{j,k} \overline B_{jk}^{22}(w(0)) \D_j\D_k u^2=0
\end{split}
\end{align}
written as a first-order system:
\begin{align}\label{dd2}
\partial_zU-\cG_{ut}(0,\zeta)U=0.
\end{align}
With $\check U$ as above observe that we can write \eqref{dd2}
equivalently as
\begin{align}\label{dd3}
\partial_z\check U-\check\cG_\mrp(0,\zeta)\check U=0.
\end{align}

\begin{prop}\label{dd4}
Suppose $\zeta$ lies in the elliptic zone. Let $\check\bE_-(\zeta)$
denote the space of initial data at $z=0$ of decaying solutions of
$\partial_z\check U-\check\cG(z,\zeta)\check U=0$, and let
$\tilde\bE_-(\zeta)$ be the stable subspace of
$\check\cG_p(0,\zeta)$.  Then
\begin{align}\label{dd5}
\check\bE_-(\zeta)\to\tilde\bE_-(\zeta)\text{ as }|\zeta|\to\infty.
\end{align}
\end{prop}

\begin{proof}
With $p:=|\zeta|^{-1}\zeta$ and $\eps:=|\zeta|^{-1}$ we can write
(with slight abuse)
\begin{align}\label{dd6}
\check\cG(z,\zeta)=\check\cG(z,p,\eps)\text{ and
}\check\cG_\mrp(z,\zeta)=\check\cG_\mrp(z,p,\eps).
\end{align}
The system  \eqref{dd1a} thus has the form \eqref{geneq}.  By Lemma
7.3(i) of \cite{GMWZ6}) the matrix $\check\cG(z,p,\eps)$ has a
spectral gap $\eta(\eps)\sim |\zeta|=\frac{1}{\eps}$.  Since
\begin{align}\label{dd7}
|\check\cG(z,\zeta)|\sim|\zeta| \text{ and
}|\check\cG_z(z,\zeta)|\sim|\zeta|,
\end{align}
Proposition \ref{bb1} implies there is a conjugator $T$ satisfying
\eqref{simple2},\eqref{bb4}.  Hence $\delta(\eps)\sim 1$ by
\eqref{bb9},\eqref{bb10}, and thus $\delta/\eta\to 0$ as $\eps\to
0$.

By Remark \ref{ccc6} we can now apply Proposition \ref{cc10} to
conclude  that $\check\bE_-(\zeta)$ approaches the stable subspace
of $\check\cG(0,p,\eps)$ as $\eps\to 0$.   The $O(1)$ term in
\eqref{dectG1} introduces an $O(\eps)$ difference between the stable
subspace of $\check\cG(0,p,\eps)$ and that of
$\check\cG_\mrp(0,p,\eps)$ (there is an $O(\eps)$ difference between
the corresponding conjugators $T$), so we conclude \eqref{dd5}.
\end{proof}

\begin{proof}[Proof of Theorem \ref{aa3} for the elliptic zone.]

Here we use the notation of the previous Proposition.

 \textbf{1.
}For $\zeta\in\cE$ and $\Gamma^{sc}_e$ as in \eqref{dd1a} define a
slightly modified rescaled Evans function
\begin{align}\label{dd8}
D^{sc}_e(\zeta)=\det(\check\bE_-(\zeta),\ker\Gamma^{sc}_e).
\end{align}
Using Lemma \ref{ZZZ} in the same argument that showed the
equivalence of estimate \eqref{stabcondhf} and the uniform Evans
stability condition Definition \ref{hfevans}, we see that for
$\zeta\in\cE$,
\begin{align}\label{dd9}
\begin{split}
&\text{ there exist }c_0, R \text{ such that }|D^{sc}(\zeta)|\geq
c_0
\text{ for }|\zeta|\geq R\Leftrightarrow\\
&\text{ there exist }c_1, R' \text{ such that }|D^{sc}_e(\zeta)|\geq
c_1 \text{ for }|\zeta|\geq R'.
\end{split}
\end{align}

\textbf{2. }Defining a frozen-coefficient Evans function
\begin{align}\label{dd10}
\tilde D^{sc}(\zeta)=\det(\tilde\bE_-(\zeta),\ker\Gamma^{sc}_e),
\end{align}
we conclude from Corollary \ref{fullsymm} and Proposition \ref{dd4}
that the conditions \eqref{dd9} are equivalent to
\begin{align}\label{dd11}
\text{ there exist }c_2, R'' \text{ such that }\tilde
|D^{sc}(\zeta)|\geq c_2\text{ for }|\zeta|\geq R''.
\end{align}
Recall that $\tilde\bE_-(\zeta)$ is the stable subspace of
$\check\cG_\mrp(0,\zeta)$, which defines the problem \eqref{dd3}
equivalent to the upper triangular system \eqref{trilinprinc}.

\textbf{3. }To complete the proof  we write
\begin{align}\label{dd12}
\Gamma^{sc}_e=(\Gamma_1,\Gamma^{sc}_*)
\end{align}
and observe that
$$
\ker \Gamma^{sc}_e=
\begin{pmatrix}\ker \Gamma_1 \\ 0 \end{pmatrix}
\oplus
\begin{pmatrix} 0\\ \ker \Gamma_*^{sc} \end{pmatrix}
$$
>From the form of $\check\cG_\mrp$ we see that $\tilde\bE_-(\zeta)$
consists of vectors of the form
\begin{align}\label{dd13}
\begin{pmatrix}v_a\\0\end{pmatrix}, \;v_a\in S(\cG^{11}_\mrp)\text{
and }\begin{pmatrix}*\\v_b\end{pmatrix},\;v_b\in S(\cP^{22}_\mrp),
\end{align}
where $S(M)$ denotes the stable subspace of $M$. After performing
obvious column operations on $\tilde D^{sc}(\zeta)$ to express the
determinant in upper block triangular form, we thus obtain (up to a
sign)
\begin{align}\label{dd14}
\tilde D^{sc}(\zeta)=\det \Big( S( \cG_{\mathrm{p}}^{11}), \ker
\Gamma_1\Big) \times \det \Big( S( \cP^{22}_{\mathrm{p}}), \ker
\Gamma_*^{sc}\Big).
\end{align}
Although $\Gamma^{sc}_*(\zeta)$ here differs slightly from
$\Gamma^{sc}_*(\zeta)$ in \eqref{aa2}, the argument that established
the equivalence \eqref{dd9} allows us to conclude that $D^2(\zeta)$
\eqref{aa2} is bounded away from $0$ for $|\zeta|$ large if and only
if the second factor on the right in \eqref{dd14} is. The first
factor equals $D^1(\zeta)$ \eqref{aa2}, so this completes the proof.

\end{proof}

\begin{rem}\label{Bprincpartrem}
For coupled boundary conditions, the
above analysis shows that the high-frequency Evans condition
is more complicated,
involving at least the upper triangular system \eqref{trilinprinc}
rather than the decoupled \eqref{princpart}.
\end{rem}

\subsubsection{The general case}\label{general}
In the remaining frequency regimes
$$
C_\delta:=
\{\zeta: 0\le \gamma \le \delta |\zeta|\}\cup
\{\zeta: |\eta|\le \delta |\zeta|\},
$$
$\delta>0$ sufficiently small, the reduction to form
\eqref{blockdiag} is considerably more complicated, in particular
involving multiplication of hyperbolic modes $u^1$ by an exponential
weighting function to obtain definiteness of $M_\pm$. Moreover, the
resulting terms $M_\pm$ depend in hyperbolic modes both on
derivatives of the profile and on the chosen exponential weight.
Thus, it is no longer true that the effective principal part $\tilde
\cG$ after rescaling/appropriate coordinate transformations involves
only frozen coefficients of the triangular system
\eqref{trilinprinc} as in the previous case. {\it However, it is
still true that the associated stable (resp. unstable) subspaces of
$\tilde \cG$ depend only on those coefficients}: indeed, only on the
coefficients of the fully decoupled system \eqref{linprinc}. Thus,
the above conclusions about the rescaled Evans function and uniform
high-frequency stability (based on the form of the stable subspaces
and not on $M_\pm$) hold true in this case as well. The maximal
estimates follow likewise from Proposition \ref{gensymm} as before,
once form \eqref{blockdiag} has been achieved.
For detailed treatments, including in particular an implicit
reduction to form \eqref{blockdiag}, see Section \ref{HF}
of this paper and Sections 7.2--7.4 of \cite{GMWZ6}.

\bigbreak


\begin{thebibliography}{GMWZ6}


\bibitem[Br]{Br}
Braslow, A.L., {\it A history of suction-type laminar-flow control
with emphasis on flight research}, NSA History Division, Monographs
in aerospace history, number 13 (1999).







\bibitem[BRa]{BRa}
Bardos, C. and Rauch, J., {\it Maximal positive boundary value
problems as limits of singular perturbation problems,} Trans. Amer.
Math. Soc. 270 (1982), no. 2, 377--408.




\bibitem[BSZ]{BSZ}
Benzoni--Gavage, S., Serre, D.,  and Zumbrun, K., {\it Alternate
Evans functions and viscous shock waves}, SIAM J. Math. Anal. 32
(2001), no. 5, 929--962 (electronic).






\bibitem[CHNZ]{CHNZ}
N.~Costanzino, J.~Humpherys, T.~Nguyen, and K.~Zumbrun, {\it
Spectral stability of noncharacteristic boundary layers of
isentropic Navier--Stokes equations,} Preprint, 2007.



\bibitem[Co]{Co}
Coppel, W. A., \emph{Stability and asymptotic behavior of
differential equations}, D.C. Heath, Boston.


\bibitem[CP]{CP}
Chazarain J. and Piriou, A., \emph{Introduction to the Theory of
Linear Partial Differential Equations}, North Holland, Amsterdam,
1982.










\bibitem[FS1]{FS1}
H.~Freist{\"u}hler and P.~Szmolyan. \emph{Spectral stability of
small shock waves}, Arch. Ration. Mech. Anal. 164 (2002), no. 4,
287--309.

\bibitem[FS2]{FS2}
H.~Freist{\"u}hler and P.~Szmolyan. \emph{Spectral stability of
small-amplitude viscous shock waves in several dimensions},
Preprint, 2007.




\bibitem[GG]{GG}
Grenier, E. and Gu\`es, O., \emph{Boundary layers for viscous
perturbations of noncharacteristic quasilinear hyperbolic problems},
J. Differential Eqns. 143 (1998), 110-146.

\bibitem[GMWZ3]{GMWZ3}
Gues, O., Metivier, G., Williams, M., and Zumbrun, K., {\it
Existence and stability of multidimensional shock fronts in the
vanishing viscosity limit}, Arch. Ration. Mech. Anal.  175  (2005),
no. 2, 151--244.


\bibitem[GMWZ4]{GMWZ4}
Gues, O., Metivier, G., Williams, M., and Zumbrun, K., \emph{Paper
4}, {\it Navier-Stokes regularization of multidimensional Euler
shocks}, Ann. Sci. \'Ecole Norm. Sup. (4)  39  (2006),  no. 1,
75--175.


\bibitem[GMWZ6]{GMWZ6}
Gues, O., Metivier, G., Williams, M., and Zumbrun, K., {\it Viscous
boundary value problems for symmetric systems with variable
multiplicities}, to appear, J. Diff. Eq.


\bibitem[GMWZ7]{GMWZ7}
Gu\`es, O., M\'etivier, G., Williams, M., and Zumbrun, K.,
\emph{Nonclassical multidimensional viscous and inviscid shocks},
 Duke Math. Journal, 142, (2008), 1-110.

\bibitem[GR]{GR}
Grenier, E. and Rousset, F., \emph{Stability of one dimensional
boundary layers by using Green's functions}, Comm. Pure Appl. Math.
54 (2001), 1343-1385.



\bibitem[GS]{GS} Gisclon, M. and Serre, D.,
\textit{Conditions aux limites pour un syst\`eme strictement
hyperbolique fournies par le sch\'ema de Godunov}. RAIRO Mod\'el.
Math. Anal. Num\'er. 31 (1997), 359--380.


\bibitem[GZ]{GZ}
Gardner, R. and Zumbrun, K., \emph{The gap lemma and geometric
criteria instability of viscous shock profiles}, CPAM 51. 1998,
797-855.





\bibitem[HLZ]{HLZ}
Humpherys, J., Lafitte, O., and Zumbrun, K.,
\emph{Stability of isentropic viscous shock profiles in the high-mach
  number limit,} Preprint (2007).

\bibitem[HLyZ1]{HLyZ1} Humpherys, J., Lyng, G., and Zumbrun, K.,
\emph{Spectral stability of ideal-gas shock layers}, Preprint
(2007).

\bibitem[HLyZ2]{HLyZ2} Humpherys, J., Lyng, G., and Zumbrun, K.,
\emph{Multidimensional spectral stability of large-amplitude
Navier-Stokes shocks}, in preparation.


\bibitem[IS]{IS} Iftimie, D. and Sueur, F.
{\it Viscous boundary layers for the Navier--Stokes equations
with the Navier slip condition,}
preprint (2008).


\bibitem[K]{K}
Kreiss, H.-O., \emph{Initial boundary value problems for hyperbolic
systems}, Comm. Pure Appl. Math. 23 (1970), 277-298.










\bibitem[KaS1]{KaS1}
Kawashima, S. and Shizuta, Y., \emph{Systems of equations of
hyperbolic-parabolic type with applications to the discrete
Boltzmann equation}, Hokkaido Math. J., 14 (1985), 249-275.

\bibitem[KaS2]{KaS2}
Kawashima, S. and  Shizuta, Y., \emph{On the normal form of the
symmetric hyperbolic-parabolic systems associated with the
conservation laws}, Tohoku Math. J. (1988), 449-464.


\bibitem[Kat]{Kat} Kato, T.,
{\it Perturbation theory for linear operators}. Classics in
Mathematics. Springer-Verlag, Berlin (1995), Reprint of the 1980
edition.

\bibitem[KNZ]{KNZ} S. Kawashima, S. Nishibata, and P. Zhu,
\emph{Asymptotic stability of the stationary solution to the
compressible
  Navier-Stokes equations in the half space,}
Comm. Math. Phys. 240 (2003), no. 3, 483--500.


\bibitem[M2]{M2}
Majda, A., \emph{The stability of multidimensional shock fronts},
Mem. Amer. Math. Soc. No. 275, AMS, Providence, 1983.


\bibitem[MaZ2]{MaZ2}  Mascia, C. and Zumbrun, K.,
{\it Stability of small-amplitude
viscous shock profiles for dissipative hyperbolic-parabolic
systems}, Comm. Pure Appl. Math.  57  (2004),  no. 7, 841--876.

\bibitem[MaZ3]{MaZ3}  Mascia, C. and Zumbrun, K.,
{\it Pointwise {G}reen function bounds for shock profiles of systems
with real viscosity,} Arch. Ration. Mech. Anal. 169 (2003), no. 3,
177--263.



\bibitem[Met2]{Met2} Metivier, G.,
\textit{The Block Structure Condition for Symmetric Hyperbolic
Problems}, Bull. London Math.Soc., 32 (2000), 689--702


\bibitem[Met3]{Met3} Metivier, G.,
\emph{Stability of multidimensional shocks}, Advances in the theory
of shock waves,  Progress in Nonlinear PDE, 47, Birkh\"auser,
Boston, 2001.


\bibitem[Met4]{Met4} Metivier, G.,
\emph{Small Viscosity and Boundary Layer Methods,}
 Birkh\"{a}user, Boston 2004.


\bibitem[MN]{MN}  Matsumura, A. and Nishihara, K.,
{\it Large-time behaviors of solutions to an inflow problem in the
half
  space for a one-dimensional system of compressible viscous gas,}
Comm. Math. Phys., 222 (2001), no. 3, 449--474.


\bibitem[MP]{MP} Majda, A. and Pego, R.L,
{\it Stable viscosity matrices for systems of conservation laws}, J.
Diff. Eqs. 56 (1985) 229--262.







\bibitem[MZ1]{MZ1}
M\'etivier, G. and Zumbrun, K., \textit{Viscous Boundary Layers for
Noncharacteristic Nonlinear Hyperbolic Problems}, Memoirs AMS,   826
(2005).





\bibitem[MZ2]{MZ2} M\'etivier, G. and Zumbrun, K.,
\textit{Hyperbolic Boundary Value Problems for Symmetric Systems
with Variable Multiplicities},    {J.  Diff.  Equ.}, 211 (2005)
61-134.

\bibitem[MZ3]{MZ3} M\'etivier, G. and Zumbrun, K.,
\textit{Symmetrizers and continuity of stable subspaces for
parabolic--hyperbolic boundary value problems.}
 Discrete Contin. Dyn. Syst.  11  (2004),  no. 1, 205--220.

\bibitem[NZ]{NZ} T. Nguyen and K. Zumbrun,
{\it Long-time stability of large-amplitude noncharacteristic
boundary layers for hyperbolic--parabolic systems},
preprint (2008).

\bibitem[PZ]{PZ}
Plaza, R. and Zumbrun, K., \emph{An Evans function approach to
spectral stability of small-amplitude shock profiles}, Discrete
Contin. Dyn. Syst. 10 (2004) 885--924.

\bibitem[Ra]{Ra} Rauch, J.,
{\it Symmetric positive systems with boundary characteristic of
constant multiplicity,} Trans. Amer. Math. Soc. 291 (1985), no. 1,
167--187.

\bibitem[R2]{R2} Rousset, F.,
{\it Inviscid boundary conditions and stability of viscous boundary
layers,}  (English summary) Asymptot. Anal. 26 (2001), no. 3-4,
285--306.

\bibitem[R3]{R3} Rousset, F.,
{\it  Stability of small amplitude boundary layers for mixed
hyperbolic-parabolic systems,} Trans. Amer. Math. Soc.  355  (2003),
no. 7, 2991--3008.

\bibitem[S]{S} H. Schlichting,
{\em Boundary layer theory}, Translated by J. Kestin. 4th ed.
McGraw-Hill Series in Mechanical
 Engineering. McGraw-Hill Book Co., Inc., New York, 1960.

\bibitem[Se]{Se} D. Serre,
{\em Systems of Conservation Laws 2}, Translated by I.N. Sneddon,
Cambridge University Press, Cambridge, 2000.


\bibitem[Su]{Su} Sueur, F.,
{\it A few remarks on a theorem by J. Rauch,}  Indiana Univ. Math.
J. 54 (2005), no.4, 1107--1143.

\bibitem[SZ]{SZ}
Serre, D. and Zumbrun, K., {\it Boundary layer stability in real
vanishing-viscosity limit}, Comm. Math. Phys. 221 (2001), no. 2,
267--292.



\bibitem[TW]{TW} Temam, R. and Wang, X.,
{\it Boundary layers associated with incompressible Navier--Stokes
equations: the noncharacteristic boundary case}, J. Diff. Eq. 179
(2002) 647--686.

\bibitem[YZ]{YZ} Yarahmadian, S. and Zumbrun, K.,
{\it Pointwise Green function bounds and long-time stability of
large-amplitude noncharacteristic boundary layers}, Preprint (2008).

\bibitem[Z1]{Z1}
Zumbrun, K., \emph{Multidimensional stability of planar viscous
shock waves}, Advances in the theory of shock waves, 304-516.
Progress in Nonlinear PDE, 47, Birkh\"auser, Boston, 2001.

\bibitem[Z2]{Zlax}
Zumbrun, K., \emph{Refined Wave--tracking and Nonlinear Stability of
Viscous Lax Shocks}, Methods Appl. Anal. 7 (2000), 747--768.

\bibitem[Z3]{Z3} Zumbrun, K.,
\textit{Stability of large-amplitude shock waves of compressible
Navier--Stokes equations}. Handbook of Fluid Mechanics III,
S.Friedlander, D.Serre ed., Elsevier North Holland (2004).


\bibitem[ZH]{ZH}
Zumbrun, K. and Howard, P., \emph{Pointwise semigroup methods and
stability of viscous shock waves}, Indiana Univ. Math. J. 47. 1998,
741-871.

\bibitem[ZS]{ZS}
Zumbrun, K. and Serre, D. \emph{Viscous and inviscid stability of
multidimensional planar shock fronts}, Indiana Univ. Math. J. 48.
1999, 937-992.


\end{thebibliography}
\end{document}